\documentclass[12pt,reqno]{amsart}
\usepackage{amsmath,amssymb,amsthm,mathrsfs,mathtools}
\usepackage[left=1.4in,right=1.4in,top=1.2in,bottom=1.2in]{geometry}
\usepackage{enumitem}
\setlist[enumerate]{itemsep=0.5ex}
\usepackage{cite}
\usepackage{tikz}
\usetikzlibrary{arrows.meta,positioning,calc}
\usepackage{array,booktabs,tabularx}
\usepackage{float}
\usepackage{microtype}
\usepackage[colorlinks=true,linkcolor=blue,citecolor=blue,urlcolor=blue]{hyperref}
\usepackage{bookmark}
\usepackage[nameinlink,noabbrev,capitalise]{cleveref}
\hypersetup{
 pdftitle={Projections in reduced group C*-algebras with prescribed trace values in (0,1)},
 pdfauthor={Guoliang Yu},
 pdfsubject={Prescribed traces and failure of reduced Baum--Connes surjectivity},
 pdfkeywords={Baum--Connes conjecture, reduced group C*-algebra, canonical trace, expanders, graphical small cancellation}
}
\theoremstyle{plain}
\newtheorem{theorem}{Theorem}[section]
\newtheorem{proposition}[theorem]{Proposition}
\newtheorem{lemma}[theorem]{Lemma}
\newtheorem{corollary}[theorem]{Corollary}
\theoremstyle{definition}
\newtheorem{definition}[theorem]{Definition}

\theoremstyle{remark}
\newtheorem{remark}[theorem]{Remark}

\numberwithin{equation}{section}
\newcommand{\Z}{\mathbb Z}
\newcommand{\Q}{\mathbb Q}
\newcommand{\R}{\mathbb R}
\newcommand{\C}{\mathbb C}

\newcommand{\Ftwo}{\mathbb F_2}
\newcommand{\Cr}{C_r^*}
\newcommand{\Xhat}{\widehat X}
\newcommand{\one}{\mathbf 1}
\newcommand{\eps}{\varepsilon}
\DeclareMathOperator{\supp}{supp}
\DeclareMathOperator{\diam}{diam}
\DeclareMathOperator{\girth}{girth}
\DeclareMathOperator{\dist}{dist}
\DeclareMathOperator{\spec}{spec}
\DeclareMathOperator{\Tr}{Tr}
\DeclareMathOperator{\Adj}{Adj}
\DeclareMathOperator{\Cay}{Cay}

\DeclareMathOperator{\im}{im}
\DeclareMathOperator{\Vertices}{Vert}
\DeclareMathOperator{\Edges}{Edges}
\DeclareMathOperator{\PSL}{PSL}

\newcommand{\abs}[1]{\lvert#1\rvert}
\newcommand{\norm}[1]{\lVert#1\rVert}

\title[Projections with prescribed trace values]
{Projections in reduced group $C^*$-algebras with prescribed trace values in $(0,1)$}
\author[G. Yu]{Guoliang Yu}
\address{Department of Mathematics, Texas A\&M University,
College Station, TX 77843, USA}
\email{guoliangyu@tamu.edu}
\date{}

\begin{document}
\begin{abstract}
Building on OpenAI's preprint, we construct, for every real number
$r\in(0,1)$, a finitely generated discrete group $G_r$ and a projection
$e_r\in C_r^*(G_r)$ whose canonical trace is $r$.  The construction
encodes the geometry of expanders in the dynamics of a semidirect product
of groups.  For irrational $r$, the $K$-theory class $[e_r]$ is not in
the image of the reduced Baum--Connes map.  The groups $G_r$ have
torsion and satisfy the maximal strong Novikov conjecture.
\end{abstract}
\maketitle

\begingroup
\small
\tableofcontents
\endgroup

\section{Introduction}
\label{sec:introduction}

The Baum--Connes conjecture provides an algorithm for computing the $K$-theory
of reduced group $C^*$-algebras \cite{BC,BCH}; see \cite{Valette} for an
introduction to the conjecture.  The canonical trace on a reduced group
$C^*$-algebra gives a numerical obstruction to the surjectivity of the
reduced Baum--Connes map.  L\"uck's trace theorem says that every class in the image of the reduced
Baum--Connes map has trace in the subring $\Lambda^G\subseteq\Q$ generated by the inverses
of the orders of the finite subgroups of $G$ \cite[Theorem~0.3]{Luck}; in
particular, this trace is rational.  Thus a projection of irrational
trace in a reduced group $C^*$-algebra produces a K-theory element that is not in the
image of the reduced Baum--Connes map.  The OpenAI
preprint~\cite{OpenAI2026} constructs a projection with irrational
canonical trace in the reduced group $C^*$-algebra of a group; for this
group, the reduced Baum--Connes map is therefore not surjective.
Building on OpenAI's work, we construct projections in reduced group
$C^*$-algebras with any prescribed trace in $(0,1)$.  Our group is
finitely generated and depends on the prescribed trace.

Let $G$ be a countable discrete group.  Its reduced group $C^*$-algebra
$\Cr(G)$ is the operator-norm closure of the complex group ring acting
by left translation on $\ell^2(G)$.  Its canonical trace is
\[
 \tau_G(a)=\langle a\delta_1,\delta_1\rangle,
 \qquad
 \tau_G\left(\sum_g a_g g\right)=a_1.
\]
In particular, $\tau_G(1)=1$, and a projection $e=e^*=e^2$ has trace in
$[0,1]$; the values $0$ and $1$ are attained by the projections $0$ and $1$.
We prove that every value in $(0,1)$ occurs.

\begin{theorem}\label{thm:main}
For every real number $r\in(0,1)$ there exist a finitely generated discrete
group $G_r$ and a projection $e_r\in\Cr(G_r)$ such that
\[
 \tau_{G_r}(e_r)=r.
\]
\end{theorem}

The cases $r=0,1$ are supplied by the zero and identity projections.
Finite direct sums of projections, together with differences of their
$K_0$-classes, then realize every real number as the trace of a
$K_0$-class.

Write $\underline EG$ for the classifying space for proper $G$-actions.
The reduced Baum--Connes map \cite{BCH} is
\[
 \mu_0^G:K_0^G(\underline EG)\longrightarrow K_0(\Cr(G)),
\]
where the left-hand side is the equivariant $K$-homology of
$\underline EG$ with $G$-compact supports, that is, the direct limit of
$K_0^G(Y)$ over the $G$-invariant subspaces $Y\subseteq\underline EG$ with
compact quotient $Y/G$.

Our main application to the Baum--Connes conjecture is the following.

\begin{corollary}\label{cor:nonsurjective}
For every irrational $r\in(0,1)$, the class $[e_r]$ in
\Cref{thm:main} is outside the image of $\mu_0^{G_r}$, and its rationalized
class is outside the image of $\mu_0^{G_r}\otimes\Q$.
\end{corollary}

The proof is given in \Cref{sec:prescribed}.

In contrast, the maximal Baum--Connes assembly map is integrally
injective for all $G_r$ (\Cref{max:main}); that is, each $G_r$ satisfies
the maximal strong Novikov conjecture.

We also consider the invariant
\[
 T(G)=\{\tau_G(e):e=e^*=e^2\in\Cr(G)\},
\]
where $\tau_G$ is the canonical trace on $\Cr(G)$.  It is invariant under group
isomorphisms and, when the reduced group $C^*$-algebra has a unique
tracial state, under isomorphisms of the reduced group $C^*$-algebras.
For the groups
\[
 \widetilde G_r=G_r*\Z,
\]
the algebras $\Cr(\widetilde G_r)$ are simple and have a unique tracial
state, and the inclusion $G_r\subseteq\widetilde G_r$ preserves the
projection $e_r$ and its canonical trace.  Countability of $T(G)$ for
countable groups then shows that our construction yields continuum many
examples of pairwise nonisomorphic simple reduced group $C^*$-algebras
with unique tracial states; see \Cref{prescribed:remarks}.
 We remark that continuum examples of nonisomorphic reduced
group $C^*$-algebras were previously discovered in~\cite{BCI,CDD}.

The reduced Baum--Connes map is known to be an isomorphism for large
classes of groups.  The first results of this kind go back to Pimsner and
Voiculescu, whose computation of the $K$-theory of the reduced
$C^*$-algebras of free groups \cite{PV} implies the conjecture for free
groups and shows that $\Cr(F_n)$ contains no projections other than $0$
and $1$.  Pimsner extended these methods to crossed products by groups
acting on trees \cite{Pim}.  Higson and Kasparov proved the conjecture for all
a-T-menable groups, that is, groups admitting a proper isometric action
on a Hilbert space \cite{HK}.  Lafforgue proved it for many groups with
property (T), including cocompact lattices in $\mathrm{SL}_3(\R)$,
$\mathrm{SL}_3(\mathbb C)$, and $\mathrm{SL}_3(\mathbb Q_p)$ \cite{Laf02}.  For
hyperbolic groups it was proved by
Lafforgue \cite{Laf02,Laf12} and by Mineyev and Yu \cite{MinYu}.

Expanders have long played a role in failures of Baum--Connes-type
surjectivity. Higson--Lafforgue--Skandalis developed expander
obstructions for groupoids, coarse spaces, and the Baum--Connes map
with coefficients~\cite{HLS}. Willett and Yu developed higher index
theory for sequences of expanders whose girths tend to infinity
and for Gromov monster groups~\cite{WY1,WY2}. For the associated
coarse disjoint unions, the coarse Baum--Connes map is injective
but not surjective, whereas its maximal version is an isomorphism.
For the Gromov monster groups considered there, the Baum--Connes
map with certain coefficients is also injective but not surjective.
For an account of these results, see~\cite{WYbook}.
Li, Nowak, and Pooya used L\"uck's trace obstruction to propose a
strategy for counterexamples: a group of type $F_{n+1}$ whose $n$-th
$\ell^2$-Betti number is not in $\Lambda^G$, and for which zero is an
isolated point of the spectrum of the $n$-th Laplacian, has a reduced
Baum--Connes map that is not surjective~\cite{LNP}.
Gromov's constructions~\cite{Gromov} introduced the use of expanders
in finitely generated groups. Here we use graphical small-cancellation
labellings, following Osajda~\cite{Osajda}, to obtain isometric
embeddings of expanders into Cayley graphs.

The starting point of our work is the construction
in~\cite{OpenAI2026}, which encodes the geometry of expanders in
the dynamics of a semidirect product. A uniform spectral gap  ensures that the kernel
projection belongs to the reduced group $C^*$-algebra.
Our labelling and linear-algebra refinements allow the trace of
this projection to be prescribed throughout a small interval
using binary expansions. A direct product with a finite cyclic
group then gives any prescribed trace in $(0,1)$.

\subsection{Organization}
\label{sec:outline}
This article is organized as follows.
In \Cref{sec:strategy}, we explain the strategy of the construction in
seven steps.
In \Cref{sec:labeling}, we prove the labelling theorem.
In \Cref{sec:geometry}, we prove the geometric estimates for
intersections between translated relator graphs.
In \Cref{sec:data}, we choose the expanders, the finite vector space
$E$, and the subspaces $\mathcal C_k$.
In \Cref{mod:section}, we construct the compact group $\Omega$,
the subgraphs $\mathcal G_\omega$, and the semidirect product $H$;
we also prove surjectivity of the restriction maps and the resulting
uniform distributions.

In \Cref{op:section}, we prove the uniform spectral gap, and in
\Cref{trace:section}, we construct the projection and compute its trace.
In \Cref{sec:prescribed}, a direct product with a finite cyclic group
completes the construction.  We then prove the consequences for
reduced assembly and obtain continuum many pairwise nonisomorphic simple
reduced group $C^*$-algebras with unique tracial states.
Finally, in \Cref{sec:maximal}, we express $G_r$ and
$\widetilde G_r$ as direct limits of groups of finite asymptotic
dimension and prove integral injectivity of their maximal assembly
maps for proper actions.

\section{Strategy of construction}
\label{sec:strategy}

In this section, we give a detailed road map of the construction and of the proofs.  The
precise definitions, statements, and proofs are given in
\Cref{sec:labeling} to \Cref{sec:prescribed}.

\smallskip
\noindent\emph{The goal and the main idea.}
We first construct projections with sufficiently small prescribed
traces.  The group $H$ is a semidirect product $V\rtimes\Gamma$, where
$\Gamma$ is finitely generated and its Cayley graph contains a sequence
of expanders.  We define a compact abelian group $\Omega$ of functions
on $\Gamma$ and let $V$ be its discrete Pontryagin dual.  The action of
$\Gamma$ on $\Omega$ gives the identification
\[
 \Cr(H)\cong C(\Omega)\rtimes_r\Gamma.
\]
The geometry of the expanders is encoded in this action.  A final tensor
product with the algebra of a finite cyclic group gives every trace in
$(0,1)$, while the resulting group remains a semidirect product with a
locally finite abelian normal subgroup.

An \emph{expander family} is a sequence $(\Theta_k)$ of finite connected
$d$-regular graphs, with $d$ independent of $k$, such that their numbers
of vertices $n_k=|\Vertices(\Theta_k)|$ tend to infinity and the smallest
positive eigenvalues of their \emph{Laplacians}
\[
 \Delta_k=dI-\Adj(\Theta_k)
\]
have a common lower bound $\rho>0$.  Here $(\Adj(\Theta_k)\xi)(v)$ is the
sum of the values of $\xi$ at the neighbors of $v$.  Connectedness implies
that the kernel of $\Delta_k$ consists exactly of the constant functions.
The spectral-gap condition is equivalently
\[
 \langle\Delta_k\xi,\xi\rangle\ge\rho\norm\xi_2^2
 \qquad\text{whenever}\qquad
 \sum_{v\in\Vertices(\Theta_k)}\xi(v)=0,
\]
with the same $\rho$ for every $k$.

Together with the estimate on the remaining vertices in Step~5, this
expander gap gives a self-adjoint element $T\in\C[H]$ such that
\[
 \spec_{\Cr(H)}(T)\subseteq\{0\}\cup[\eps,2d]
 \qquad(\eps>0).
\]
If $\chi$ is a continuous function equal to $1$ at $0$ and zero on
$[\eps,2d]$, then $b=\chi(T)$ is a projection in $\Cr(H)$.  In the
representations associated to points of $\Omega$, it is the projection
onto constant functions on certain complete expander copies.

The trace is the Haar average of the diagonal coefficient of this
projection at the identity.  The projection onto constants on a
connected graph with $n$ vertices has diagonal coefficients $1/n$.
There are $n$ translated copies of that graph through the identity.  If
each is complete with probability $q$, and distinct complete copies are
disjoint, their total contribution to the trace is
$n\cdot(1/n)\cdot q=q$.  We arrange these probabilities to be $2^{-k}$.
Selecting a set of indices therefore gives a binary subseries, and these
subseries fill an interval of small trace values.

The finite vector space in which the vertex values lie is
$E=\Ftwo^m$, for a fixed $m$.  For a set $D$, we write $E^D$ for the
vector space over $\Ftwo$ of all functions $D\to E$; the \emph{support}
of $u\in E^D$ is the set of points of $D$ at which $u$ is nonzero.  We
use a $k$-dimensional subspace
$\mathcal C_k\subseteq E^{\Vertices(\Theta_k)}$, with dimension taken
over $\Ftwo$.  It has $2^k$ elements, although each vertex can take
$2^m$ values.  The integer $m$ provides enough different vertex values
to distinguish all sets of outgoing edge labels.  The integer $k$
determines the probability $2^{-k}$.

The sizes of the graphs must grow.  The ambient space
$E^{\Vertices(\Theta_k)}$ has dimension $mn_k$ over $\Ftwo$, so
$k\le mn_k$ is necessary.  In fact, the construction chooses
$2k\le n_k\le8k/3$ (\cref{data:family}).  An irrational trace requires
infinitely many nonzero terms $2^{-k}$, hence arbitrarily large $k$ and
unbounded graph sizes.  Although $n_k$ cancels in
$n_k\cdot(1/n_k)\cdot2^{-k}$, the dimension requirement still forces
these sizes to grow.

The spectral gap must remain uniform as the sizes grow.  Each finite
connected graph has a positive smallest nonzero Laplacian eigenvalue,
but these eigenvalues can tend to zero.  For the cycle with $n$ vertices,
for example, this eigenvalue is $4\sin^2(\pi/n)$.  The expander condition
prevents this loss of a common lower bound.  It controls the operator
on each complete copy, while the degree-defect estimate in Step~5
controls the vertices that do not belong to any complete copy.
The resulting gap makes the kernel projection an element of the
reduced group $C^*$-algebra by continuous functional calculus.  Without
such a gap, bounded Borel functional calculus alone would place the
kernel projection in the group von Neumann algebra.

\smallskip
\noindent\emph{Step 1: graph operators in a crossed product, and the group $H$.}
We first describe the framework in which the operator $T$ is constructed.
Let $\Gamma$ be a group generated by a finite set $S$, and let $X$ be its
Cayley graph, with vertex set $\Gamma$ and an edge between $h$ and $hs$ for
every $h\in\Gamma$ and $s\in S\cup S^{-1}$.  We assume that the symbols
in $S\cup S^{-1}$ represent pairwise distinct nonidentity elements of
$\Gamma$, so $X$ is a simple graph.  Let $\Omega$ be a compact space on
which $\Gamma$ acts by homeomorphisms, and suppose that every point
$\omega\in\Omega$ determines a subgraph $\mathcal G_\omega$ of $X$, with vertex set
$\Gamma$, in which every vertex has degree at most $d$.  We require that
$\mathcal G_{g\omega}=g\mathcal G_\omega$ for all $g\in\Gamma$, and that for each
$s\in S\cup S^{-1}$ the set of points $\omega$ for which the edge $\{1,s\}$ belongs
to $\mathcal G_\omega$ is open and closed.  Let $F_s\in C(\Omega)$ be the indicator
function of this set.  Then the edge $\{h,hs\}$ belongs to $\mathcal G_\omega$
exactly when $F_s(h^{-1}\omega)=1$, so the operator
$\Delta_\omega=dI-\Adj(\mathcal G_\omega)$ on $\ell^2(\Gamma)$ is given by
\[
 (\Delta_\omega\xi)(h)=d\,\xi(h)-\sum_{s\in S\cup S^{-1}}F_s(h^{-1}\omega)\,\xi(hs).
\]
This is the structure of an element of the reduced crossed product
$C(\Omega)\rtimes_r\Gamma$.  Writing $u_g$ for the unitary corresponding to
$g\in\Gamma$, put
\[
 T=d\cdot1-\sum_{s\in S\cup S^{-1}}F_s\,u_s\in C(\Omega)\rtimes_r\Gamma.
\]
For $\omega\in\Omega$, let $\pi_\omega$ be the representation of
$C(\Omega)\rtimes_r\Gamma$ on $\ell^2(\Gamma)$ given by
$(\pi_\omega(f)\xi)(h)=f(h^{-1}\omega)\xi(h)$ for $f\in C(\Omega)$ and
$(\pi_\omega(u_g)\xi)(h)=\xi(hg)$.  Then $\pi_\omega(T)=\Delta_\omega$ for every $\omega$, so the
single element $T$ gives all the operators $\Delta_\omega$.  Moreover, the
norm of an element of $C(\Omega)\rtimes_r\Gamma$ is the supremum of its norms
in the representations $\pi_\omega$; we prove this in the case needed here, where
$\Omega$ is a compact abelian group (\Cref{mod:representations}).  Hence a
spectral gap for the operators $\Delta_\omega$ that is uniform in $\omega$ gives a
spectral gap for $T$.

Now suppose that $\Omega$ is a compact abelian group and that $\Gamma$ acts
on it by automorphisms.  The Fourier transform identifies $C(\Omega)$ with
the reduced $C^*$-algebra of the Pontryagin dual $V$ of $\Omega$, the
discrete abelian group of continuous homomorphisms from $\Omega$ to the
circle.  The group $\Gamma$ acts on $V$, and
\[
 C(\Omega)\rtimes_r\Gamma\cong\Cr(V\rtimes\Gamma).
\]
Thus $T$ becomes an element of the reduced group $C^*$-algebra of the
discrete group $H=V\rtimes\Gamma$.  It lies in the group ring $\C[H]$ if each
$F_s$ is a finite linear combination of characters of $\Omega$.  It remains
to choose $\Gamma$, $\Omega$, and the subgraphs $\mathcal G_\omega$ so that the
operators $\Delta_\omega$ have a spectral gap independent of $\omega$, and so that
the trace of the projection onto the kernel of $T$ can be computed.  We
make these choices in Steps 2 to 4.

\smallskip
\noindent\emph{Step 2: expanders inside a group.}
Recall that the \emph{girth} of a graph is the length of its shortest
cycle.  We use $6$-regular expanders $\Theta_k$, one for each sufficiently
large integer $k$, whose girths tend to infinity (\Cref{sec:data}).  The condition that the girths tend to
infinity is an additional requirement for the small-cancellation and
overlap arguments.  To place these graphs
inside a group we use labels.
Fix a finite set $S$, and let $S^{-1}=\{s^{-1}:s\in S\}$ be a set of formal
inverses.  A \emph{labeling} of a graph $\Theta$ by $S$ chooses, for each
edge of $\Theta$, a direction and an element $s\in S$; the edge traversed in
the chosen direction has \emph{label} $s$, and traversed in the opposite
direction it has label $s^{-1}$.  A \emph{path} in $\Theta$ is a sequence of
vertices $v_0,v_1,\ldots,v_n$ in which $v_{j-1}$ and $v_j$ are joined by an
edge for each $j$; vertices may be repeated.  A path is \emph{closed} if
$v_n=v_0$, and \emph{simple} if its vertices are pairwise distinct.  Reading the labels of
the edges from $v_{j-1}$ to $v_j$ in order gives a word $s_1s_2\cdots s_n$
with $s_j\in S\cup S^{-1}$, which we call the \emph{label} of the path.  For
example, the closed path around a square whose edges, traversed in order,
have labels $a$, $b$, $a^{-1}$, $b^{-1}$ has label $aba^{-1}b^{-1}$.

Given labelings of all the graphs $\Theta_k$ by the same finite set $S$, we
define $\Gamma$ by generators and relations: $\Gamma$ is the quotient of the
free group on $S$ by the normal subgroup generated by the labels of all
closed paths in all the $\Theta_k$.  In other words, $\Gamma$ is generated
by the elements of $S$, and its only relations are those forced by the
requirement that the label of every closed path, read as a product of
generators and their inverses, equals the identity in $\Gamma$.  In the
square example, the relation $aba^{-1}b^{-1}=1$ says that $a$ and $b$
commute.  Thus the graphs together with their labelings determine the
group: different labelings of the same graphs give different relations,
and in general different groups.

These relations are exactly what is needed to map each $\Theta_k$ into
the Cayley graph $X$ of $\Gamma$, the graph with vertex set $\Gamma$ in which
$g$ and $gs$ are joined by an edge with label $s$ for every $g\in\Gamma$ and
$s\in S$.  Fix a vertex $o$ of $\Theta_k$, and send each vertex $v$ to the
product in $\Gamma$ of the labels along a path from $o$ to $v$.  This is well
defined, because two such paths differ by a closed path, whose label is
the identity in $\Gamma$, and it sends edges to edges.  In general, however,
such a map need not be injective, and it may shrink distances, because
the relations of $\Gamma$ can identify vertices of $\Theta_k$ or bring them
closer together.  The graphical small-cancellation theory of Gromov and
Osajda shows that this does not happen if the labeling satisfies two
combinatorial conditions.  First, the labeling is \emph{reduced}: the
edges at any vertex have different labels.  Then a path is determined by
its starting vertex and its label, because at each step the next edge is
the unique edge with the next label.  Second, a long simple path is
determined by its label alone, without knowing its starting vertex.  More
precisely, we fix a constant $\lambda$ with $0<\lambda\le1/24$ and require
the following: if a simple path in $\Theta_k$ has length at least $\lambda$
times the girth of $\Theta_k$, then no other simple path, in $\Theta_k$ or in
any other graph $\Theta_i$, has the same label.

The second condition does not hold automatically, even for reduced
labelings.  For example, label all edges of a cycle of length $n$ by the
same symbol $a$, oriented in the same direction around the cycle.  This
labeling is reduced, since at each vertex the two edges have labels $a$
and $a^{-1}$.  However, for each $\ell<n$, the $n$ simple paths of length
$\ell$ that go around the cycle in the positive direction, one starting at
each vertex, all have the label $a^\ell$.  The effect on the group is
drastic.  If a second cycle, of length $m$, is labeled in the same way,
then $\Gamma=\langle a\mid a^n=a^m=1\rangle$ is cyclic of order
$\gcd(n,m)$, and if $\gcd(n,m)<n$, the first cycle does not embed in the
Cayley graph of $\Gamma$.  Under the two conditions above, by contrast,
every $\Theta_k$ embeds isometrically into the Cayley graph $X$ of $\Gamma$.  In
\Cref{sec:labeling}, we construct labelings with this property, using one
finite set $S$ for all the graphs.  We call the translates
$g\Theta_k\subseteq X$, for $g\in\Gamma$, the \emph{copies} of $\Theta_k$.

This produces many expanders inside one group, but it does not yet
produce the desired operator.  Different copies overlap, so $X$ is not a
disjoint union of expanders, and the Laplacian of $X$ has no useful
kernel.  We show in \Cref{sec:geometry} that the overlaps are
nevertheless controlled.  Form the coned-off graph $\Xhat$ by adding
one vertex $c_P$ for each copy $P$ and joining it to every vertex of $P$.
Fix a root $R$, which may be a group vertex or a cone point.  For every
copy $P$ with $c_P\ne R$, all intersections $P\cap Q$ with different
copies satisfying
\[
 d_{\Xhat}(c_Q,R)\le d_{\Xhat}(c_P,R)
\]
lie in one intrinsic ball $B_P$ of $P$ of radius
$r_P\le\girth(P)/16$.  The chosen girth is greater than $16$, so
$2r_P+1<\girth(P)$; hence the induced ball is a tree.  These
\emph{overlap estimates} are used twice below.

\smallskip
\noindent\emph{Step 3: the compact group $\Omega$ and its dual.}
We now specify the vertex values.  After the finite labeling alphabet
$S$ has been chosen, choose an integer $m\ge1$ such that $E=\Ftwo^m$
contains at least as many nonzero vectors as there are $d$-element
subsets of $S^{\pm1}$.  Assign distinct nonzero vectors $a_J\in E$ to
these subsets $J$, and define
\[
 \mathcal L(a_J)=J,
 \qquad
 \mathcal L(a)=\varnothing\quad\text{for all other }a\in E.
\]
In particular, $\mathcal L(0)=\varnothing$.  For every function
$\omega\colon\Gamma\to E$, define $\mathcal G_\omega$ to have vertex set
$\Gamma$, with the edge $\{h,hs\}$ present exactly when
\[
 s\in\mathcal L(\omega(h))
 \quad\text{and}\quad
 s^{-1}\in\mathcal L(\omega(hs)).
\]
Every vertex has degree at most $d$.  With the translation action
$(g\omega)(h)=\omega(g^{-1}h)$, the graphs satisfy
$\mathcal G_{g\omega}=g\mathcal G_\omega$.  The condition for a given edge
depends on only two vertex values, so its indicator is continuous in
the product topology.

For $v\in\Vertices(\Theta_k)$, let $J_k(v)\subseteq S^{\pm1}$ be the set
of labels of the oriented edges of $\Theta_k$ starting at $v$.  Define
\[
 w_k\colon\Vertices(\Theta_k)\longrightarrow E,
 \qquad w_k(v)=a_{J_k(v)}.
\]
If $\varphi_k\colon\Theta_k\to X$ is the chosen based embedding, every
copy has the form $P=g\varphi_k(\Theta_k)$.  Its label-preserving
identification is $\iota_P=g\varphi_k\colon\Theta_k\to P$.  Different
$g$ give different copies (\cref{geo:copies}).  We identify $P$ with its
vertex set when discussing functions and put
\[
 w_P=w_k\circ\iota_P^{-1}\colon P\longrightarrow E.
\]
Equivalently, $w_P(h)=a_{J_P(h)}$, where $J_P(h)$ is the set of labels of
the $d$ oriented edges of $P$ starting at $h$.

A copy $P$ is \emph{complete for $\omega$} if $\omega|_P=w_P$.  Then all
edges of $P$ belong to $\mathcal G_\omega$, and their labels exhaust the
$d$ labels specified at each vertex.  Thus $P$ is an entire connected
component of $\mathcal G_\omega$.

We must restrict the functions $\omega$.  If all functions
$\Gamma\to E$ were allowed, we could take $\omega=w_P$ on
$P\setminus\{v\}$ for one vertex $v$, and set $\omega=0$ elsewhere.
Only the $d$ edges of $P$ incident at $v$ would disappear.  For a nonzero
vector $\xi$ and a nonnegative self-adjoint operator $A$, the
\emph{Rayleigh quotient} is
\[
 R_A(\xi)=\frac{\langle A\xi,\xi\rangle}{\norm\xi^2}.
\]
It depends on the vector and need not be an eigenvalue.  In finite
dimensions, the smallest positive eigenvalue, when it exists, is the
minimum over nonzero vectors orthogonal to the kernel.  In the example
above, take $A=dI-\Adj(\mathcal G_\omega)$ and
$\xi=\one_{P\setminus\{v\}}$, the indicator function of
$P\setminus\{v\}$.  Then
\[
 R_A(\xi)=\frac d{\abs P-1}.
\]
Each connected component of $P\setminus\{v\}$ contains a vertex adjacent
to $v$ in $P$, where the potential $d-\deg_{\mathcal G_\omega}$ is positive.
The restriction of $A$ to functions supported on $P\setminus\{v\}$
therefore has zero kernel, and its smallest eigenvalue lies in
$(0,d/(\abs P-1)]$.  These eigenvalues tend to zero as the copies grow.
We exclude this possibility by requiring an incomplete copy to differ
from the reference function at a fixed positive proportion of its
vertices.

Choose a subspace
\[
 \mathcal C_k\subseteq E^{\Vertices(\Theta_k)},
 \qquad \dim_{\Ftwo}\mathcal C_k=k,
 \qquad w_k\in\mathcal C_k,
\]
with the additional properties described in Step~4.  For each copy $P$
of $\Theta_k$, put
\[
 \mathcal C_P
 =\{u\circ\iota_P^{-1}:u\in\mathcal C_k\}\subseteq E^P.
\]
It has dimension $k$ over $\Ftwo$ and contains $w_P$.  Its orthogonal
complement is
\[
 \mathcal C_P^\perp
 =
 \left\{z\in E^P:
   \sum_{h\in P}z(h)\cdot u(h)=0
   \text{ for every }u\in\mathcal C_P
 \right\}.
\]
Here $\cdot$ denotes the standard dot product on $E=\Ftwo^m$, and the
sum is taken in $\Ftwo$.  The subspace $\mathcal C_P^\perp$ has
dimension $mn_k-k$.

Define
\[
 \Omega=\{\omega\colon\Gamma\to E:
       \omega|_P\in\mathcal C_P\text{ for every copy }P\}.
\]
Here every copy means every translate of every graph in the subfamily
used to define $\Gamma$.  The conditions are linear and involve finitely many
coordinates at a time.  They make $\Omega$ a closed subgroup of the
compact product $E^\Gamma$, invariant under translation by $\Gamma$.
Write $\mu$ for its Haar probability measure, the probability measure
invariant under addition of elements of $\Omega$.  Step~4 shows that
each restriction map $\Omega\to\mathcal C_P$ is onto.  Its pushforward
of $\mu$ is therefore uniform measure on the $2^k$ functions in
$\mathcal C_P$, and
\[
 \mu\{\omega:\omega|_P=w_P\}=2^{-k}.
\]

We now obtain the discrete group from this compact group.  Put
\[
 V=\operatorname{Hom}_{\mathrm{cont}}(\Omega,\Ftwo),
 \qquad H=V\rtimes\Gamma.
\]
A continuous homomorphism $v\colon\Omega\to\Ftwo$ is automatically
$\Ftwo$-linear, and $\chi_v(\omega)=(-1)^{v(\omega)}$ identifies it
with a continuous character of $\Omega$.  A character is a continuous
homomorphism to the unit circle; here its values are $1$ and $-1$
because every element of $\Omega$ has order dividing two.
Thus $V$ is the Pontryagin dual of $\Omega$.

For $h\in\Gamma$ and $1\le j\le m$, let
\[
 v_{h,j}(\omega)=\omega(h)_j.
\]
Every continuous linear functional on $\Omega$ depends on finitely many
coordinates.  Its induced functional on the finite-dimensional image
of the restriction map extends to the corresponding full finite
product of copies of $E$.  Hence it is a linear combination of the
functionals $v_{h,j}$.  Moreover, $v_{h,j}=h v_{1,j}$ for the dual
translation action.  It follows that $V$ is generated as an
$\Ftwo[\Gamma]$-module by $v_{1,1},\ldots,v_{1,m}$, and $H$ is generated
by these $m$ elements together with the generators of $\Gamma$.

The equations for $\Omega$ also give the presentation
\[
 V\cong E^{(\Gamma)}/W,
 \qquad
 W=\sum_P\mathcal C_P^\perp,
\]
where $E^{(\Gamma)}$ consists of finitely supported functions
$\Gamma\to E$.  To form the sum defining $W$, we extend each
$z\in\mathcal C_P^\perp$, initially defined only on $P$, to a function
$\widetilde z\colon\Gamma\to E$ by
\[
 \widetilde z(h)=
 \begin{cases}
  z(h),&h\in P,\\
  0,&h\notin P.
 \end{cases}
\]
Here $0$ is the zero vector of $E$: the values on $P$ are unchanged,
and every vertex outside $P$ is assigned zero.  Since $P$ is finite,
$\widetilde z$ is finitely supported.  We identify $\mathcal C_P^\perp$
with its image in $E^{(\Gamma)}$ under $z\mapsto\widetilde z$.
The sum ranges over every translate of every selected relator graph
and is an algebraic sum: every element of $W$ is a finite sum of these
extended functions.  We prove this presentation in \Cref{mod:functions}
and use it for the finite-stage approximations in
\Cref{sec:maximal}.  The duality definition already gives
$\Cr(H)\cong C(\Omega)\rtimes_r\Gamma$.

For clarity, the coefficient functions of the operator in Step~1 can
be written explicitly.  Define $f_s\colon\Omega\to\{0,1\}$ by
\[
 f_s(\omega)=
 \begin{cases}
  1,&s\in\mathcal L(\omega(1)),\\
  0,&s\notin\mathcal L(\omega(1)).
 \end{cases}
\]
Thus $f_s(\omega)=1$ if and only if the label $s$ of the oriented edge
$1\to s$ belongs to $\mathcal L(\omega(1))$.  The reverse edge $s\to1$
has label $s^{-1}$ and initial vertex $s$, so its condition is
$s^{-1}\in\mathcal L(\omega(s))$.  With
$\alpha_s(f)(\omega)=f(s^{-1}\omega)$ and
$(s^{-1}\omega)(1)=\omega(s)$, we have
\[
 \alpha_s(f_{s^{-1}})(\omega)=
 \begin{cases}
  1,&s^{-1}\in\mathcal L(\omega(s)),\\
  0,&s^{-1}\notin\mathcal L(\omega(s)).
 \end{cases}
\]
The indicator of the condition that $\{1,s\}$ belongs to
$\mathcal G_\omega$ is therefore $F_s=f_s\alpha_s(f_{s^{-1}})$.  Since
$u_sfu_s^*=\alpha_s(f)$,
\[
 T=d1-\sum_{s\in S^{\pm1}}f_su_sf_{s^{-1}}.
\]
Every $f_s$ depends on the finite vector $\omega(1)\in E$, so it is a
finite linear combination of characters by Fourier expansion on $E$.
Thus $T$ belongs to $\C[H]$.

\smallskip
\noindent\emph{Step 4: extension of functions and uniform distributions.}
Set $\eta=1/(100m)$.  We choose $\mathcal C_k$ so that every nonzero
function in $\mathcal C_k$ or $\mathcal C_k^\perp$ is nonzero at more
than $\eta n_k$ vertices.  These two requirements have complementary
consequences.

First, if $u,u'\in\mathcal C_k$ are distinct, then $u-u'$ is nonzero,
so the functions differ at more than $\eta n_k$ vertices.  In
particular, on an incomplete copy $P$ the functions $\omega|_P$ and
$w_P$ differ at more than $\eta|P|$ vertices.  At each such vertex $h$,
the set $\mathcal L(\omega(h))$ is empty or is a different $d$-element
set from $J_P(h)$, by the injectivity of $J\mapsto a_J$.  At least one
of the edges of $P$ incident at $h$ is therefore missing from
$\mathcal G_\omega$.

Second, values prescribed on a sufficiently small set of vertices can
always be extended.  More precisely, if $F\subseteq P$ and
$|F|<\eta|P|$, the restriction map
\[
 \mathcal C_P\longrightarrow E^F
\]
is onto.  If it were not, a nonzero linear functional on $E^F$ would
vanish on its image.  Extending the coefficients of this functional by
zero outside $F$ would give a nonzero element of $\mathcal C_P^\perp$
supported on $F$, which contradicts the choice of the orthogonal
complement.

This local extension property and the geometry imply that every
$u_0\in\mathcal C_{P_0}$ extends to an element of $\Omega$.  To see this,
first take a finite family of copies containing $P_0$, and root the
coned-off graph at $c_{P_0}$.  If there is another copy, choose one,
called $Q$, whose cone point is farthest from this root.  The overlap
lemma places all intersections of $Q$ with the other copies in a
single ball $B_Q$ with fewer than $\eta|Q|$ vertices.  Remove $Q$ and
solve the extension problem on the remaining family by induction.
The values already assigned prescribe a function on the intersections
with $Q$, and the local extension property extends that function to
an element of $\mathcal C_Q$.  This gives a consistent assignment on
the whole finite family, agreeing with $u_0$ on $P_0$.

The conditions on copies and the condition $\omega|_{P_0}=u_0$ are
closed conditions in the compact product $E^\Gamma$.  Since each finite
collection has a solution, compactness gives a simultaneous solution.
Thus $\Omega\to\mathcal C_{P_0}$ is onto.  The pushforward of Haar
measure is the uniform measure on the finite group $\mathcal C_{P_0}$,
so a copy of type $k$ is complete with probability $2^{-k}$
(\cref{mod:marginals}).

These requirements are compatible: arbitrary values on a small set
can extend, but changing one element of $\mathcal C_k$ to another
requires changing its values on many vertices.  The extension property
allows agreement on small intersections; the separation property
excludes restrictions that differ from $w_P$ at only a few vertices
and would give small eigenvalues.

The spaces exist by the elementary random-subspace argument in
\cref{code:exist}.  The space of functions on $\Theta_k$ has dimension
$mn_k$ over $\Ftwo$.  Relatively few of its functions are nonzero at
only $\eta n_k$ or fewer vertices.  Choosing a random $k$-dimensional
subspace containing $w_k$, the probability that it or its orthogonal
complement contains any such nonzero function is less than one.
The counting estimate is given in \Cref{code:exist}.
Randomness here proves existence of the fixed spaces $\mathcal C_k$;
Haar probability on $\Omega$ is used afterward to compute the trace.

\smallskip
\noindent\emph{Step 5: the spectral gap.}
The kernel of $\pi_\omega(T)$ consists of the functions that are constant on
each complete copy and vanish elsewhere.  The gap has two sources.  On a
complete copy, $\pi_\omega(T)$ is the Laplacian of an expander, so its nonzero
eigenvalues are at least $\rho$; this is where the expander property is
used.  On the vertices that do not belong to any complete copy, the
separation property of Step~4 and the overlap estimates show that every
finite set $B$ of such vertices satisfies
\[
 \abs{\partial_{\mathcal G_\omega}B}+\sum_{h\in B}\bigl(d-\deg_{\mathcal G_\omega}(h)\bigr)
 =d\abs B-2\abs{\Edges(\mathcal G_\omega[B])}\ \ge\ \frac c2\abs B,
\]
where $\mathcal G_\omega[B]$ is the subgraph of $\mathcal G_\omega$ spanned by $B$ and
$c>0$ is a constant (\cref{op:globaldeficit}).  The first term counts the
edges of $\mathcal G_\omega$ leaving $B$, and the second measures the degree
deficit relative to $d$.  This is a discrete isoperimetric inequality of
Cheeger type.  Integrating it over the level sets of a function gives a
uniform positive lower bound for $dI-\Adj(\mathcal G_\omega)$ on the complement
of the complete copies (\cref{op:gap}).  The separation property is essential: a copy
that was complete except for a few missing edges would carry an almost
constant function with very small energy, and the gap would close.  Both
bounds are uniform in $\omega$, so the spectrum of $T$ in $\Cr(H)$ is contained
in $\{0\}\cup[\eps,2d]$ for some $\eps>0$, and $b=\chi(T)$ is a projection
in $\Cr(H)$.

\smallskip
\noindent\emph{Step 6: the trace.}
For the kernel projection $b=\chi(T)$, the canonical trace is
\[
 \tau_H(b)=\int_\Omega
 \langle\pi_\omega(b)\delta_1,\delta_1\rangle\,d\mu(\omega)
\]
(\cref{mod:representations}).  If the identity lies in a complete copy
$P$ of $\Theta_k$, the corresponding diagonal coefficient is $1/n_k$,
because $\pi_\omega(b)$ is the projection onto constants on $P$.
Otherwise that coefficient is zero.  Exactly $n_k$ copies of
$\Theta_k$ contain the identity, one for each vertex of $\Theta_k$.
Each is complete with probability $2^{-k}$.  Distinct complete copies
are disjoint, so at most one complete copy contains the identity.
Consequently,
\[
 \tau_H(b)=\sum_k n_k\cdot\frac1{n_k}\cdot2^{-k}
           =\sum_k2^{-k}
\]
(\cref{trace:exact}).  Each selected expander $\Theta_k$ contributes
$2^{-k}$ to the trace.

\smallskip
\noindent\emph{Step 7: the prescribed trace.}
For $A\subseteq\{K,K+1,\ldots\}$, let $F(S)$ be the free group on
the fixed labeling alphabet $S$, and define
\[
 \begin{split}
 \mathcal R_A
   &=\{\operatorname{lab}(\gamma): k\in A,\;
       \gamma\text{ is a closed path in }\Theta_k\},\\
 \Gamma_A
   &=F(S)/\langle\!\langle\mathcal R_A\rangle\!\rangle_{F(S)}.
 \end{split}
\]
Here $\operatorname{lab}(\gamma)$ is the word obtained by reading the
oriented edge labels along $\gamma$, and
$\langle\!\langle\mathcal R_A\rangle\!\rangle_{F(S)}$ denotes the
smallest normal subgroup of $F(S)$ containing $\mathcal R_A$.
For each $k\in A$, let
$\varphi_{A,k}\colon\Vertices(\Theta_k)\to\Gamma_A$ be the vertex map
of the based embedding into $\Cay(\Gamma_A,S)$.  Put
$\iota_{g,k}(v)=g\varphi_{A,k}(v)$ for $g\in\Gamma_A$.
Recall from Step~3 that $E=\Ftwo^m$, where $\Ftwo$ is the field with
two elements and $m$ was chosen so that distinct nonzero vectors
$a_J\in E$ can be assigned to all $d$-element subsets
$J\subseteq S^{\pm1}$.  Define
\[
 \Omega_A=
 \left\{\omega\colon\Gamma_A\to E:\;
 \begin{array}{l}
  \omega\circ\iota_{g,k}\in\mathcal C_k
       \text{ for every }k\in A\\
  \text{and every }g\in\Gamma_A
 \end{array}\right\}.
\]
Thus a copy of $\Theta_k$ uses the subspace of dimension $k$, and the
condition is imposed on every translate of each graph in the subfamily.
Put $V_A=\widehat\Omega_A$ and $H_A=V_A\rtimes\Gamma_A$.
Under the identification
$\Cr(H_A)\cong C(\Omega_A)\rtimes_r\Gamma_A$, let $u_{A,s}$ be the
canonical unitary corresponding to $(0,s)\in H_A$, for
$s\in S^{\pm1}$.  Using the map $\mathcal L$ from Step~3, define
$f_{A,s}\in C(\Omega_A)$ by
\[
 f_{A,s}(\omega)=
 \begin{cases}
  1,&s\in\mathcal L(\omega(1)),\\
  0,&s\notin\mathcal L(\omega(1)).
 \end{cases}
\]
The corresponding operator is
\[
 T_A=d\one-\sum_{s\in S^{\pm1}}
         f_{A,s}u_{A,s}f_{A,s^{-1}}\in\Cr(H_A).
\]
This is the construction of Step~3 applied to $\Gamma_A$ and $\Omega_A$.
All the preceding estimates apply with the same constants, so Step~5 gives
$\sigma(T_A)\subseteq\{0\}\cup[\eps,2d]$, with $\eps$ independent of $A$.
Using the continuous real function $\chi$ from Step~1, with $\chi(0)=1$
and $\chi=0$ on $[\eps,2d]$, define
\[
 b_A=\chi(T_A)\in\Cr(H_A).
\]
Since $\chi^2=\chi$ on $\sigma(T_A)$, continuous functional calculus
shows that $b_A$ is a projection; in the regular representation, it is
the orthogonal projection onto $\ker T_A$.  Step~6 gives
\[
 \tau_{H_A}(b_A)=\sum_{k\in A}2^{-k}.
\]
Binary expansions therefore realize every $q\in[0,2^{1-K}]$, with the
group $H_A$ allowed to depend on $q$.

To obtain an arbitrary $r\in(0,1)$, take a direct product with a fixed
finite cyclic group $C_N$.  If $e\in C^*(C_N)$ has trace $a/N$ and
$b_A$ has trace $q$, the orthogonal sum
\[
 B=e\otimes1+(1-e)\otimes b_A\in\Cr(C_N\times H_A)
\]
has trace $a/N+(1-a/N)q$.  Equivalently, this algebra is a direct sum
of $N$ copies of $\Cr(H_A)$: the projection is the identity in $a$
summands and $b_A$ in the others, and its trace is the average of these
traces.  \Cref{sec:prescribed} chooses $a$ and $q\in[0,2^{1-K}]$ so
that either $B$ or $1-B$ has trace $r$.  The resulting group
$G_r=C_N\times H_A\cong(C_N\times V_A)\rtimes\Gamma_A$ is again a
semidirect product with a locally finite abelian normal subgroup.

\subsection{Choosing the constants}
\label{sec:constants}

In this subsection, we list the constants and their order of choice.
First choose the geometric constants and the small-cancellation
parameter $\lambda$.  Start with an expander family indexed by
all sufficiently large $k$, and label the entire family with one
finite alphabet $S$ using \cref{lab:main}.  Next choose $m$ so that
$E=\Ftwo^m$ has enough nonzero vectors for the sets of outgoing labels,
and put $\eta=1/(100m)$.  Finally restrict to the indices $k\ge K$ so that
the overlap balls have fewer than $\eta n_k$ vertices and all required
girth bounds hold.  This restriction uses the existing labeling and
does not change $S$ or $m$.  All these constants, and the constants
derived from them later, are fixed before the subset $A$ or the
prescribed real number $r$ is chosen.

\section{Small-cancellation labelings with many graphs of each girth}
\label{sec:labeling}

In this section, we show that an infinite family of finite graphs,
containing many graphs of the same girth, can be labeled by finitely many
letters in such a way that every sufficiently long word is the label of at
most one simple path in the whole family (\Cref{lab:main}).  By the
graphical small-cancellation theory of Gromov and Osajda, the labels of
the closed paths in these graphs are then the relators of a finitely
generated group $\Gamma$ whose Cayley graph contains every graph of the
family isometrically (\Cref{lab:embedding}).  We call the graphs of
such a family \emph{relator graphs}.  This section, \Cref{sec:geometry},
and the choice of expanders in \Cref{sec:data} carry out Step 2 of the
strategy in \Cref{sec:strategy}.  We first recall the relevant
definitions.

\emph{Alphabets and labelings.}
Let $S$ be a finite set, let $S^{-1}=\{s^{-1}:s\in S\}$ be a set of formal
inverses, disjoint from $S$, and put $S^{\pm1}=S\cup S^{-1}$.  We extend
$s\mapsto s^{-1}$ to an involution of $S^{\pm1}$ without fixed points by
setting $(s^{-1})^{-1}=s$.  We call $S^{\pm1}$ an \emph{alphabet} and its
elements \emph{letters}; a \emph{word} is a finite sequence of letters.
Each edge of a graph has two orientations, and for an oriented edge $e$ we
write $\bar e$ for the same edge with the opposite orientation.  A
\emph{labeling} of a graph $\Theta$ by $S^{\pm1}$ is a map $\ell$ from the
set of oriented edges of $\Theta$ to $S^{\pm1}$ such that
$\ell(\bar e)=\ell(e)^{-1}$ for every oriented edge $e$.  A labeling is
\emph{reduced} if the oriented edges starting at any given vertex have
pairwise distinct labels.

\emph{Paths and their labels.}
A \emph{path} in $\Theta$ is a finite sequence $(e_1,\ldots,e_n)$ of
oriented edges such that $e_{j+1}$ starts at the endpoint of $e_j$; the
number $n$ is its \emph{length}.  Its \emph{label} is the word
$\ell(e_1)\ell(e_2)\cdots\ell(e_n)$.  A path is \emph{closed} if it ends
at its starting vertex, and \emph{simple} if its vertices are pairwise
distinct.  Paths are oriented: a path and its reverse
$(\bar e_n,\ldots,\bar e_1)$ are different paths, and the label of the
reverse is the formal inverse of the label.

\emph{The problem.}
The labeling argument of \cite[Lemma~2.2]{Osajda} passes to a sparse
subsequence of graphs.  We need one graph for each sufficiently large
integer, so our family contains many graphs of comparable girth.  Two
kinds of coincidences must be excluded: a word might be the label of
simple paths in two different graphs of the family, or the label of two
different simple paths in the same graph.  We exclude the second kind by
an elementary random coloring (\Cref{lab:within}), and the first kind by
the Lov\'asz local lemma (\Cref{lab:lll}), extending the argument of
\cite[Lemma~2.2]{Osajda}.

\begin{lemma}\label{lab:lll}
Let $(E_i)_{i\in J}$ be finitely many events.  Suppose that $N_i\subset
J\setminus\{i\}$ is a dependency neighborhood: $E_i$ is independent of the
sigma algebra generated by all $E_j$ with $j\notin N_i\cup\{i\}$.
If $0<x_i<1$ and
\[
 \Pr(E_i)\le x_i\prod_{j\in N_i}(1-x_j),
\]
then $\Pr(\bigcap_{i\in J}E_i^c)>0$.
\end{lemma}
\begin{proof}
Induction on $|B|$ proves
\begin{equation}\label{lab:conditional}
 \Pr\bigl(E_i\mid\bigcap_{j\in B}E_j^c\bigr)\le x_i
 \qquad (i\notin B).
\end{equation}
All conditioning events are positive by the preceding inductive bounds.
For the inductive step partition $B$ as $B_1=B\cap N_i$ and
$B_2=B\setminus N_i$.  Independence bounds the numerator, after conditioning
on $\bigcap_{B_2}E_j^c$, by $\Pr(E_i)$.  Express the additional conditioning
probability for $B_1$ as a product in any order.  Each factor is at least
$1-x_j$ by the induction hypothesis, since fewer than $|B|$ events have
already been conditioned on.  Consequently the conditional probability
in \eqref{lab:conditional} is at most
\[
 \frac{\Pr(E_i)}{\prod_{j\in B_1}(1-x_j)}\le x_i.
\]
If $B_1$ is empty, independence proves the claim directly.  Finally expand
$\Pr(\bigcap_J E_i^c)$ as a product of successive conditional probabilities;
it is at least $\prod_J(1-x_i)>0$.
\end{proof}

\begin{lemma}\label{lab:within}
Fix $D\ge3$, $A>0$, and $0<\lambda\le1/24$.  There is a finite alphabet,
depending only on these constants, with the following property.  Every
finite connected simple graph $\Theta$ of maximum degree at most $D$,
girth $\sigma$, and diameter at most $A\sigma$, with $\lambda\sigma>3$,
has a reduced labeling by this alphabet in which different simple paths
of length $\lceil\lambda\sigma\rceil$ have different labels.
\end{lemma}
\begin{proof}
Write $\ell_0=\lceil\lambda\sigma\rceil$ for the length of the paths to be
compared.  Color the edges of $\Theta$ with at most $2D-1$ colors so that
edges with a common vertex have different colors; a greedy coloring
succeeds because an edge meets at most $2D-2$ other edges.  Give each edge
an arbitrary reference orientation.  The \emph{first labeling} assigns to
an edge of color $c$ the letter $a_c$ in its reference orientation and
$a_c^{-1}$ in the other orientation.  It is reduced, because the edges at a
vertex have different colors.

Next, choose independently for each edge a second color $t$, uniformly
from $\{1,\ldots,M\}$.  The \emph{full labeling} assigns to an edge with
colors $c$ and $t$ the letter $a_{c,t}$ in its reference orientation and
$a_{c,t}^{-1}$ in the other orientation.  It is reduced, because the first
labeling is.

Consider different simple paths $P=(e_1,\ldots,e_{\ell_0})$ and
$Q=(f_1,\ldots,f_{\ell_0})$.  Their labels in the full labeling can agree
only if their labels in the first labeling agree.  In that case no $e_j$ and $f_j$ are the same unoriented
edge.  Indeed, opposite orientations give different first letters; the
same orientation gives a common vertex at the same position.  Reducedness
then determines both paths from this vertex, forwards and backwards, and
would imply $P=Q$.

For a pair whose labels in the first labeling agree, equality of the
labels in the full labeling imposes the $\ell_0$ equations $t(e_j)=t(f_j)$,
where $t(e)$ is the second color of $e$.  Form the multigraph whose vertices are the
unoriented edges occurring in $P\cup Q$ and whose edges are these equations.
It has no loops.  Each vertex has degree at most two, since each simple path
uses an edge at most once.  Every nontrivial component is a path or a cycle,
where a double edge is permitted as a cycle of length two.  If a component
has $v$ vertices and $e$ edges, all its $v$ colors must agree, an event
of probability $M^{-(v-1)}$.  Since $v-1\ge e/2$, multiplying over the
components bounds the probability of all equations by $M^{-\ell_0/2}$.

A degree bound gives
\[
 |\Theta|\le D^{\diam\Theta+1}\le D^{A\sigma+1}.
\]
There are at most $|\Theta|D^{\ell_0}$ paths of length $\ell_0$, hence at
most $|\Theta|^2D^{2\ell_0}$ ordered pairs.  Since $\sigma\le\ell_0/\lambda$,
the probability that some pair of different simple paths has equal labels
is at most
\[
 D^2\bigl(D^{2A/\lambda+2}M^{-1/2}\bigr)^{\ell_0}.
\]
Choose an integer $M>D^{4A/\lambda+8}$.  This expression is less than one
for $\ell_0\ge4$.  Hence some choice of second colors gives a labeling with
the required property.
\end{proof}

\begin{theorem}\label{lab:main}
Let $(\Theta_i)_{i\in I}$ be a finite or countable family of finite connected
simple graphs.  Suppose their degrees are at most $D\ge3$, their girths
are $\sigma_i$, and
\[
 \diam\Theta_i\le A\sigma_i,\qquad \lambda\sigma_i>3,
 \qquad 0<\lambda\le1/24.
\]
Put $\gamma_i=\lfloor\lambda\sigma_i\rfloor$.  Assume that there are
$C_0,B\ge1$ such that
\begin{equation}\label{lab:multiplicity}
 \#\{i\in I:\gamma_i=m\}\le C_0B^m\qquad(m\ge1).
\end{equation}
Then there are a finite alphabet and a reduced labeling of every
$\Theta_i$ by this alphabet with the following property: the label of a
simple path in $\Theta_i$ of length at least $\lambda\sigma_i$ is not the
label of any other simple path in any of the graphs $\Theta_j$, $j\in I$.
The alphabet and the labeling can be chosen for the full family and then
restricted to any subfamily.
\end{theorem}
\begin{proof}
Each level in \eqref{lab:multiplicity} is finite, so we can enumerate $I$
in nondecreasing order of $\gamma_i$.  We first distinguish paths in
different graphs.  Choose a real number $a>\max\{2,2DB\}$ and put
\[
 S_0=\sum_{i\in I}\gamma_i(2D/a)^{\gamma_i}.
\]
This sum converges by \eqref{lab:multiplicity}.  Set
$C=D^{2A/\lambda+2}$, and choose an even integer
\begin{equation}\label{lab:alphabet}
 L\ge aC\exp(2S_0).
\end{equation}
Use an alphabet with $L$ letters, that is, $S^{\pm1}$ with $\abs S=L/2$.

Suppose the first $n-1$ graphs have already been labeled.  Give each edge
of $\Theta_n$ a reference orientation and label it, in that orientation,
by a letter chosen uniformly and independently from the $L$ letters; the
opposite orientation receives the inverse letter.  For $i<n$ and a simple
path $P$ of length $\gamma_i$ in $\Theta_n$, let $E_{i,P}$ be the event that
the label of $P$ equals the label of some path of length $\gamma_i$ in the
already labeled graph $\Theta_i$.
Since $P$ is simple, its unoriented edges are distinct and receive
independent random labels.
The number of labels of paths of length $\gamma_i$ in $\Theta_i$ is at most
\[
 |\Theta_i|D^{\gamma_i}
 \le D^{A\sigma_i+1+\gamma_i}
 \le C^{\gamma_i},
\]
where $\sigma_i<(\gamma_i+1)/\lambda\le2\gamma_i/\lambda$ was used.
Therefore
\begin{equation}\label{lab:bad-probability}
 \Pr(E_{i,P})\le (C/L)^{\gamma_i}.
\end{equation}

Two such events are independent whenever their paths in $\Theta_n$ have
no common edge.  A fixed path of length $\gamma_i$ meets at most
\[
 \Delta_{ij}:=\gamma_i\gamma_j(2D)^{\gamma_j}
\]
oriented paths of length $\gamma_j$: choose a shared edge, its position and
orientation in the other path, and then its remaining edges.  The events are indexed separately for each earlier graph $\Theta_j$,
also when several $\gamma_j$ have the same value; the labels of paths in
$\Theta_j$ that must be avoided are accounted for in
\eqref{lab:bad-probability}.

Assign $x_i=a^{-\gamma_i}$ to every event of type $i$.  Because
$\log(1-x)\ge-2x$ for $0\le x\le1/2$, we obtain
\begin{align*}
 x_i\prod_{j<n}(1-x_j)^{\Delta_{ij}}
 &\ge a^{-\gamma_i}
       \exp\left(-2\gamma_i\sum_{j<n}
                 \gamma_j(2D/a)^{\gamma_j}\right)\\
 &\ge a^{-\gamma_i}e^{-2S_0\gamma_i}
 \ge (C/L)^{\gamma_i}.
\end{align*}
\Cref{lab:lll} supplies a labeling of $\Theta_n$ for which none of the
events $E_{i,P}$ occurs.  By induction we obtain a labeling of all the
graphs, which we call the \emph{first labeling}, such that for $i<n$ no
label of a path of length $\gamma_i$ in $\Theta_i$ is the label of a simple
path in $\Theta_n$.  Thus two simple paths in these graphs cannot have a
common label of length at least $\lambda\sigma_i$, or of length at least
$\lambda\sigma_n$, by taking an initial subpath of length $\gamma_i$.

Next apply \Cref{lab:within} to every graph separately, with the same
alphabet for all graphs; call the result the \emph{second labeling}.
Finally, label each oriented edge by the pair consisting of its letters in
the first and in the second labeling, with $(a,b)^{-1}=(a^{-1},b^{-1})$.
The resulting alphabet is finite and the labeling is reduced.  Two simple
paths with the same label have the same label in both labelings; by the
first labeling they lie in the same graph, and by the second they are
equal.  This proves the required uniqueness.
\end{proof}

\begin{remark}\label{lab:logarithmic}
If the family is indexed by integers $k\ge K$ and
$\sigma_k\ge c_0\log k$, then \eqref{lab:multiplicity} follows immediately:
$\gamma_k=m$ implies $k<\exp((m+1)/(\lambda c_0))$.  Repetitions of an
unlabeled graph are allowed; they are different instances in the disjoint
family and receive labels separately.
\end{remark}

\begin{theorem}\label{lab:embedding}
For a reduced labeling as in \Cref{lab:main}, with $\lambda<1/24$, let
\[
 \Gamma=\langle S\mid\text{labels of closed paths in all }\Theta_i\rangle.
\]
Fix a vertex $o_i$ of each $\Theta_i$, and let $\varphi_i$ send a vertex
$v$ of $\Theta_i$ to the element of $\Gamma$ represented by the label of a
path from $o_i$ to $v$.  Then $\varphi_i$ is well defined, and it is an
isometric embedding of the vertex set of $\Theta_i$, with its path metric,
into the Cayley graph $\Cay(\Gamma,S)$ with its word metric.
\end{theorem}
The map $\varphi_i$ is well defined because two paths from $o_i$ to $v$
differ by a closed path, whose label is a relator.  A path is
\emph{immersed} if it never traverses an edge and then immediately the
same edge in the opposite direction.  A \emph{piece} is a word that is the
label of two different immersed paths in the family.  The graphical
small-cancellation condition $C'(\lambda)$ requires that every piece that
is the label of a path in $\Theta_i$ have length less than $\lambda\sigma_i$.
\Cref{lab:main} implies this condition: if a word were the label of two
different immersed paths of length at least
$\lambda\min\{\sigma_i,\sigma_j\}$, an initial segment of length
$\lceil\lambda\min\{\sigma_i,\sigma_j\}\rceil<\min\{\sigma_i,\sigma_j\}$
would be a simple path in each of the two graphs, since an immersed path
shorter than the girth is simple.  By \Cref{lab:main}, these two simple
paths with the same label coincide.  Since the labeling is reduced, a path
is determined by its starting vertex and its label, so the two immersed
paths coincide as well, a contradiction.  Under the condition
$C'(1/24)$, \Cref{lab:embedding} is the metric embedding theorem for
graphical small-cancellation presentations; see
\cite[Lemma~3.1 and Theorem~3.2]{Osajda}.  In \Cref{sec:geometry}, we build
on this theorem to control the intersections of the translates of the
embedded graphs.

\section{Geometry of the embedded relator graphs}
\label{sec:geometry}

In this section, we study how the translates of the relator graphs
intersect inside the Cayley graph $X$ of $\Gamma$.  We call these
translates \emph{copies}.  Different copies may overlap, and a single copy
may meet infinitely many others.  Our main result is that these overlaps
are nevertheless controlled simultaneously: once a base point is fixed,
all intersections of a copy $P$ with the copies that are no farther from
the base point lie in one ball of $P$ whose radius is at most
$\girth(P)/16$ (\cref{geo:overlap}).  Because the girth is large, this ball
is a tree with few vertices.  As a consequence, every finitely supported
cycle in $X$ decomposes uniquely as a sum of cycles in copies
(\cref{geo:cycle-sum}).  These are the overlap estimates of Step 2 of the
strategy.  They are used in \Cref{mod:section} to show that the linear
equations imposed on different copies are compatible, and in
\Cref{op:section} to prove the spectral gap.

Following \cite[Sections~2--3]{OpenAI2026}, we obtain the overlap estimate
in three steps: we add a cone point over each copy, we estimate van Kampen
diagrams over the presentation of $\Gamma$, and we use the hyperbolicity
of the resulting coned-off graph.  The diagram estimate and the
hyperbolicity of the coned-off graph are variants of standard facts in
graphical small-cancellation theory; the simultaneous control of all
overlaps of a copy by one ball is the ingredient specific to the present
construction.

\subsection{Relator graphs and the coned-off graph}

In this subsection, we state the properties of the labeled relator graphs
that are used throughout this section, and we introduce the coned-off
graph, obtained from the Cayley graph by adding one new vertex for each
copy.

\begin{definition}[Geometric hypotheses]
\label{geo:hypotheses}
Let $(\Theta_i)_{i\in I}$ be a countable family of finite, connected, simple,
$d$-regular graphs, where $d\geq3$ is fixed, and write
$\sigma_i=\girth\Theta_i$.  Give their oriented edges labels in a finite
alphabet $S^{\pm1}$, with opposite orientations receiving inverse labels.
We require the following properties, for a fixed number $\lambda$:
\begin{enumerate}[label=(\roman*)]
\item $0<\lambda\leq1/24$ and $\lambda\sigma_i>3$ for every $i$;
\item the labeling is reduced: distinct edges leaving a vertex have
      distinct labels;
\item if a simple oriented path in $\Theta_i$ has length at least
      $\lambda\sigma_i$, its label does not occur on any other simple
      oriented path anywhere in the disjoint union of the $\Theta_i$;
\item for
\[
 \Gamma=\langle S\mid \text{labels of closed paths in the }\Theta_i\rangle,
 \qquad X=\Cay(\Gamma,S),
\]
the based label maps $\varphi_i:\Theta_i\longrightarrow X$ are isometric
embeddings on vertices.
\end{enumerate}
At this stage we regard $X$ as the \emph{formal} Cayley graph: its
vertex set is $\Gamma$, and for every $g\in\Gamma$ and $s\in S^{\pm1}$
there is an oriented edge from $g$ to $gs$ labeled $s$, whose reverse is
the oriented edge from $gs$ to $g$ labeled $s^{-1}$.  A priori this graph
might have loops or multiple edges; \Cref{geo:avoidance} shows that it is
simple.  Choose one base vertex in each $\Theta_i$ and normalize
$\varphi_i$ to send it to the identity.
\end{definition}

For $g\in\Gamma$ and $i\in I$, let $P_{g,i}=g\varphi_i(\Theta_i)$ be the
translated labeled subgraph of $X$.  We call $P_{g,i}$ an \emph{indexed
copy}, or simply a \emph{copy}, of \emph{type} $i$.  The graph $\Theta_i$ is
its \emph{source graph}, the element $g$ is its \emph{translating element},
and the map
\[
 \iota_P=g\varphi_i\colon\Theta_i\longrightarrow P=P_{g,i}
\]
is its \emph{identification} with the source graph.  Since $\varphi_i$ is
injective on vertices and preserves labels, $\iota_P$ is an isomorphism of
labeled graphs.  For a path $\gamma$ in $P$, the path $\iota_P^{-1}(\gamma)$ in
$\Theta_i$ is its \emph{preimage} in the source graph; it has the same label
as $\gamma$, and it is simple if $\gamma$ is simple.  By \Cref{geo:copies},
different pairs $(g,i)$ give different labeled subgraphs.  For a copy $P$ we write
$d_P$ for its path metric, that is, the graph distance in $P$; balls for
$d_P$ are called \emph{intrinsic balls} of $P$.  We write
$\sigma(P)=\sigma_i$ for the girth of a copy of type $i$.

\begin{lemma}
\label{geo:copies}
Suppose that two indexed copies $P$ and $Q$ share a simple edge path whose
length is at least $\lambda\sigma(P)$.  Then $P$ and $Q$ are the same indexed
copy, that is, they have the same type and the same translating element.
In particular, different pairs $(g,i)$ give different labeled subgraphs of
$X$.
\end{lemma}

\begin{proof}
Write $P=P_{g,i}$ and $Q=P_{g',j}$, and let $\gamma$ be the shared simple
path, with a chosen orientation.  Its preimages $\iota_P^{-1}(\gamma)$ in
$\Theta_i$ and $\iota_Q^{-1}(\gamma)$ in $\Theta_j$ are simple paths with the
same label as $\gamma$.  The first has length at least $\lambda\sigma_i$.  By
the uniqueness property in \cref{geo:hypotheses}(iii), the two preimages
are the same path; in particular $i=j$, and both preimages start at the same
vertex $v$ of $\Theta_i$.  Then $g\varphi_i(v)$ and $g'\varphi_i(v)$ are both
the initial vertex of $\gamma$, so $g=g'$.  Hence $P=Q$ as indexed copies.

For the last assertion, note that every $\Theta_i$ has a simple path of
length at least $\lambda\sigma_i$, for example an initial segment of a
shortest cycle.  If two indexed copies have the same labeled subgraph of
$X$, they share such a path, so they are equal by the first part.
\end{proof}

Form a graph $\Xhat$ by adjoining, for each copy $P$, a new vertex $c_P$ and
an edge from $c_P$ to each vertex of $P$.  All edges, including the original
generator edges, have length one.  This is the \emph{coned-off graph} of
$X$ relative to the copies.  We refer to vertices of $X$ as group vertices
and to the new vertices as \emph{cone points}.  The graph $\Xhat$ is connected
and geodesic; it need not be locally finite.

\subsection{Minimal diagrams with prescribed boundary sides}
\label{sec:diagrams}

In this subsection, we prove the basic estimate on van Kampen diagrams
over the presentation of $\Gamma$.  A short loop in the coned-off graph can
become a long path in $X$ when each passage through a cone point is
replaced by a path in the corresponding copy.  We keep the endpoints of
these paths and the copies containing them fixed, but allow the paths themselves to
vary, and we choose a diagram with as few faces as possible.  We first
explain the terms used for such loops and diagrams: \emph{sides},
\emph{fixed} and \emph{adjustable} sides, and \emph{arcs}
(\Cref{geo:sides-figure}).  \Cref{geo:short-arcs} then shows that, if the
loop has $m$ sides, a minimal diagram has fewer than $m$ faces and at most
$6m$ arcs, and that every arc shared by two different copies is short.

\emph{Sides of a closed path.}
Let $v=(v_0,v_1,\ldots,v_N)$, with $v_N=v_0$, be a closed edge path in the
coned-off graph $\Xhat$; that is, $v_{t-1}$ and $v_t$ are joined by an edge
for $1\le t\le N$.  We assume that $v_0$ is a group vertex and that the
path is nonempty, that is, $N\ge1$; the path of length zero needs no
filling.  Every edge of
$\Xhat$ is either an edge of the Cayley graph $X$, which we call a
\emph{generator edge}, or an edge joining a vertex of a copy $P$ to its
cone point $c_P$.  Cone points are not adjacent to each other.  Hence, if
$v_t=c_P$, then $0<t<N$ and $v_{t-1},v_{t+1}\in P$; we call
$(v_{t-1},c_P,v_{t+1})$ a \emph{cone passage} through $c_P$.  Consequently
there are unique indices
\[
 0=t_0<t_1<\cdots<t_m=N
\]
such that each subpath $(v_{t_{j-1}},\ldots,v_{t_j})$ is either a single
generator edge (when $t_j-t_{j-1}=1$) or a cone passage (when
$t_j-t_{j-1}=2$).  We call these $m$ subpaths the \emph{sides} of $v$, and
the indices $t_0,\ldots,t_{m-1}$ its \emph{breakpoints}.  Thus a
breakpoint is an index, that is, an element of $\{0,\ldots,N-1\}$, and not
a vertex.  Different indices may give the same vertex, and an index $t$
with $v_t=v_{t_j}$ need not be a breakpoint.  The vertices $v_{t_j}$ are
group vertices, and there are at most $m$ of them.

\emph{Fixed and adjustable sides.}
For $1\le j\le m$ we define a path $\gamma_j$ in the Cayley graph $X$ as
follows.  If the $j$-th side of $v$ is a generator edge, let $\gamma_j$ be
this edge; we call $\gamma_j$ a \emph{fixed side}.  If the $j$-th side is a
cone passage $(x,c_P,y)$, let $\gamma_j$ be a path in the copy $P$ from $x$
to $y$; when $x=y$, the path of length zero is allowed.  We call such a
$\gamma_j$ an \emph{adjustable side}, and $P$ its \emph{assigned copy}.
The concatenation $\gamma=\gamma_1\gamma_2\cdots\gamma_m$ is a closed path
in $X$.  Writing $\abs{\gamma_i}$ for the length of $\gamma_i$, the indices
\[
 s_j=\abs{\gamma_1}+\cdots+\abs{\gamma_j}\qquad(0\le j\le m-1)
\]
are the \emph{breakpoints} of $\gamma$; the vertex of $\gamma$ at index
$s_j$ is the vertex $v_{t_j}$.  The fixed sides are determined by $v$.
The adjustable sides are not: for a cone passage $(x,c_P,y)$, any path in
$P$ from $x$ to $y$ may be used.  In the minimization below, the
adjustable sides range over all such paths, while the fixed sides are
always the given generator edges.

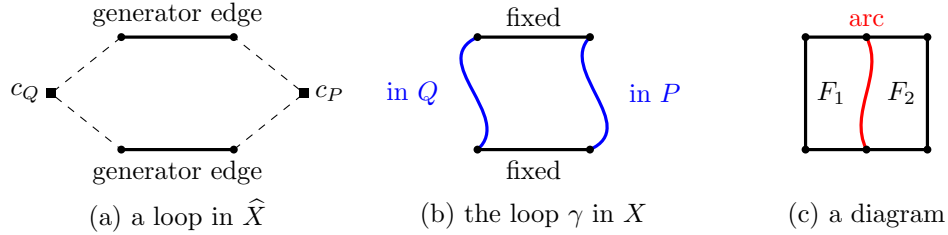
\begin{figure}[htbp]
\centering
\begin{tikzpicture}[scale=0.62,every node/.style={font=\footnotesize}]
 \begin{scope}
  \coordinate (a0) at (0,0); \coordinate (a1) at (2.4,0);
  \coordinate (a2) at (2.4,2.4); \coordinate (a3) at (0,2.4);
  \coordinate (cP) at (3.9,1.2); \coordinate (cQ) at (-1.5,1.2);
  \draw[very thick] (a0)--(a1);
  \draw[very thick] (a2)--(a3);
  \draw[dashed] (a1)--(cP)--(a2);
  \draw[dashed] (a3)--(cQ)--(a0);
  \foreach \p in {a0,a1,a2,a3} \fill (\p) circle (2.4pt);
  \fill (cP) ++(-3pt,-3pt) rectangle ++(6pt,6pt);
  \fill (cQ) ++(-3pt,-3pt) rectangle ++(6pt,6pt);
  \node[right] at (cP) {$c_P$};
  \node[left] at (cQ) {$c_Q$};
  \node[below] at (1.2,0) {generator edge};
  \node[above] at (1.2,2.4) {generator edge};
  \node at (1.2,-1.4) {(a) a loop in $\Xhat$};
 \end{scope}
 \begin{scope}[xshift=7.6cm]
  \coordinate (b0) at (0,0); \coordinate (b1) at (2.4,0);
  \coordinate (b2) at (2.4,2.4); \coordinate (b3) at (0,2.4);
  \draw[very thick] (b0)--(b1);
  \draw[very thick] (b2)--(b3);
  \draw[very thick,blue] (b1) .. controls (3.5,0.5) and (1.7,1.7) .. (b2);
  \draw[very thick,blue] (b3) .. controls (-1.1,1.9) and (0.8,0.7) .. (b0);
  \foreach \p in {b0,b1,b2,b3} \fill (\p) circle (2.4pt);
  \node[blue,right] at (3.0,1.2) {in $P$};
  \node[blue,left] at (-0.6,1.2) {in $Q$};
  \node[below] at (1.2,0) {fixed};
  \node[above] at (1.2,2.4) {fixed};
  \node at (1.2,-1.4) {(b) the loop $\gamma$ in $X$};
 \end{scope}
 \begin{scope}[xshift=14.6cm]
  \coordinate (c0) at (0,0); \coordinate (c1) at (2.6,0);
  \coordinate (c2) at (2.6,2.4); \coordinate (c3) at (0,2.4);
  \draw[very thick] (c0)--(c1)--(c2)--(c3)--cycle;
  \draw[very thick,red] (1.3,0) .. controls (0.9,0.8) and (1.7,1.6) .. (1.3,2.4);
  \foreach \p in {c0,c1,c2,c3} \fill (\p) circle (2.4pt);
  \fill (1.3,0) circle (2.4pt); \fill (1.3,2.4) circle (2.4pt);
  \node at (0.55,1.2) {$F_1$};
  \node at (2.05,1.2) {$F_2$};
  \node[red,above] at (1.3,2.4) {arc};
  \node at (1.3,-1.4) {(c) a diagram};
 \end{scope}
\end{tikzpicture}
\caption{Sides and arcs.  (a) A closed path in the coned-off graph with
four sides: two generator edges and two cone passages, through the cone
points $c_P$ and $c_Q$.  The four marked dots are the vertices at the
breakpoints.  (b) The corresponding
closed path $\gamma$ in $X$: the generator edges are kept as fixed sides,
and each cone passage is replaced by a path in the corresponding copy
(blue); these paths are the adjustable sides.
(c) A van Kampen diagram filling a closed path, with two faces $F_1$ and
$F_2$, each bounded by a closed path in a copy.  The red path separating
the faces is an arc: a maximal path whose interior vertices have degree
two and are unmarked.  Its label occurs in the copies assigned to both
faces, so it is a piece in the classical sense when these two copies are
distinct.}
\label{geo:sides-figure}
\end{figure}

\emph{Diagrams.}
An \emph{admissible diagram} for these data is a van Kampen diagram for
the closed path in $X$ just described.  More precisely, it is a finite
connected plane multigraph, mapped edge by edge to $X$, whose outer
boundary is the chosen closed path with its division into sides, and in
which the boundary of each bounded face is a closed path in some copy,
called the copy \emph{assigned} to that face.  The boundary path of a
face, or of the outer region, is also called its \emph{contour}.  We
write the outer contour as a closed path $(u_0,u_1,\ldots,u_M)$ in the
diagram, mapped to $\gamma$; its breakpoints are the indices $s_j$ of
$\gamma$.  The vertices $u_{s_0},\ldots,u_{s_{m-1}}$ are called
\emph{marked vertices}.  As for paths, an index $t$ for which $u_t$ is a
marked vertex need not be a breakpoint.  Such diagrams exist by the
van Kampen lemma for the presentation of $\Gamma$.  The contour of each
bounded face is labeled by the label of a closed path in some graph
$\Theta_i$, starting at a vertex $v$; if the contour starts at a vertex of
the diagram that is mapped to $h\in\Gamma$, then the copy $P_{g,i}$ with
$g\varphi_i(v)=h$ contains the image of the contour, and we take it as the
copy assigned to the face.  We allow boundary
paths of faces to repeat edges and vertices, so that ordinary planar
edge deletions are possible without assuming that face boundaries are
embedded circles.  An empty diagram is represented by one marked vertex.

\emph{Arcs.}
We now choose the adjustable sides $\gamma_j$ and an admissible diagram
so as to minimize first the number $F$ of bounded faces, and then the
number of edges.  Only the adjustable sides and the diagram are chosen
here; the fixed sides are the given generator edges.  A diagram attaining
this minimum is called \emph{minimal}.
We simplify the diagram by \emph{suppressing} every unmarked vertex of
degree two, that is, we merge the
two edges at such a vertex into a single longer edge.  The edges of the
resulting graph are called \emph{arcs}.  Equivalently, an arc is a
maximal path in the diagram whose interior vertices have degree two and
are unmarked; its length is the number of edges of $X$ that it contains.
Every arc has two incidences with region boundaries, possibly with the
same region.  An arc separating faces
assigned to two different copies $P$ and $Q$ has a label that occurs on
paths in both copies; in classical small-cancellation terminology, such a
label is a \emph{piece}.  We write $b$ for the number of arcs.  An
\emph{incidence} is one occurrence of an edge or an arc on the boundary
of a region.  Thus an arc has two incidences, and a bridge, that is, an
edge whose removal disconnects the diagram, has two incidences on the
boundary of the same region.

\begin{lemma}
\label{geo:short-arcs}
In a minimal diagram for a nonempty closed path as above, so that $m\ge1$,
the following assertions hold.
\begin{enumerate}[label=(\roman*)]
\item The contour of a bounded face assigned to $P$ has length at least $\sigma(P)$,
      and every arc incident to it has length less than
      $\lambda\sigma(P)$.
\item An arc whose two incidences lie on regions with different assigned
      copies has length less than $\lambda\sigma(Q)$ for either of these
      copies $Q$.
      An arc incident to a fixed side has length one.
\item One has
\begin{equation}
 \left(\lambda^{-1}-6\right)F\leq 6m-6,
 \qquad b\leq6m.
 \label{geo:diagram-count}
\end{equation}
\end{enumerate}
\end{lemma}

\begin{proof}
We first give the diagram reductions which will be used in counting arcs.

\emph{Removing a freely trivial face.}
A word is \emph{freely trivial} if it can be reduced to the empty word by
repeatedly deleting two consecutive letters of the form $ss^{-1}$; it is
\emph{freely reduced} if it contains no two consecutive letters of this
form.  The label of the contour of a bounded face cannot be freely trivial.
Indeed, by the van Kampen lemma, the label of the outer contour is freely
equal to a product of conjugates of the labels of the contours of the
bounded faces, each face used once.  A freely trivial factor can be
omitted.  Applying the van Kampen lemma to
the remaining product gives a diagram with fewer bounded faces and the
same boundary word.  With its initial group vertex fixed, the boundary
path in $X$ is determined by that word, so its sides and their assigned
copies are preserved.  This contradicts minimality of $F$.  The label of a closed
path of length less than the girth of a graph is freely trivial, which
proves the lower bound on the lengths of face contours.

\emph{Merging equal assignments.}
Two bounded faces assigned to the same copy cannot share an edge.
Deleting such an edge joins their regions and decreases $F$ by one; the
new contour still runs entirely in that copy.  The edge is not a bridge,
since its incident faces are different.  Further contacts between the
regions merely produce repeated contour vertices or bridges, which are
allowed.  Likewise, a bounded face cannot share an edge with an adjustable
side with the same assigned copy.  Delete the edge and replace its
occurrence in that side by the rest of the face contour.  The side keeps
its endpoints, its position in the outer contour, and its assigned copy,
while $F$ decreases.

\emph{Removing leaves and internal bridges.}
A \emph{leaf} is a vertex of degree one.  Recall that a \emph{bridge} is an
edge whose removal disconnects the diagram.  We call a bridge
\emph{internal} if both of its incidences lie on the contour of a bounded
face.
A \emph{backtrack} is a two-edge path that traverses an edge and then
immediately returns along the same edge in the opposite direction.
There is no unmarked vertex of degree one.  Suppose that $\ell$ is such a
vertex, and let $e$ be the unique edge at $\ell$.  If $e$ lies on the
contour of a bounded face, that contour traverses $e$ and then
immediately $e^{-1}$; deleting $\ell$ and $e$ removes this backtrack and
leaves a closed path in the assigned copy.  Otherwise $e$ lies on the
outer contour, so that $u_t=\ell$ and $u_{t-1}=u_{t+1}$ for some index $t$.
Since $\ell$ is unmarked, $t$ is not a breakpoint, that is, an index
separating two consecutive prescribed sides of the outer contour.
Hence $s_{j-1}<t<s_j$
for some $j$, and the indices $t-1,t,t+1$ all belong to the side
$\gamma_j$.  This side contains at least two edges, so it is not a fixed
side; it is an adjustable side, and deleting the backtrack replaces it by
a shorter path in the same copy with the same endpoints.  In either case
the number of edges decreases and $F$ does not increase, contrary to
minimality.

Nor can a bridge have both incidences on a bounded face.  Remove the open
bridge.  The outer contour belongs to one of the two remaining components;
keep this component and discard the other one.  On the affected face
contour, this removes a subpath of the form $e\gamma e^{-1}$ and joins the
parts on either side.  The remaining contour is still a closed path in its
assigned copy, whether or not $\gamma$ is freely trivial.  The prescribed
outer contour is unchanged, the number of bounded faces cannot increase,
and the edge count decreases.  This is again impossible.

\emph{Shortness of arcs.}
Each incidence of an arc lies either on the contour of a single bounded
face or on the outer contour.  In the second case the arc lies in a single
side $\gamma_j$: the interior vertices of an arc are unmarked, so no
breakpoint occurs in its interior.  In particular, an arc on a fixed side
has length one.  By the preceding reductions, if one incidence of an arc
lies on a bounded face, then the other incidence lies on a bounded face
or an adjustable side with a different assigned copy, or on a fixed
side.

Suppose an arc with different assignments has length greater than two.
Its label, read along the arc, is freely reduced.  To see this, note
that two consecutive letters $ss^{-1}$ would correspond to two consecutive
edges of the arc whose common vertex is unmarked and has degree two.  Its outer endpoints are
distinct diagram vertices: an arc has no repeated internal vertices, and
even an arc whose two endpoints coincide closes up only at its ends; an arc of length two was
excluded.  Contract this two-edge tree, identifying the outer endpoints,
which have the same image in $X$.  The two steps cancel on each incident
contour, so all assignments and sides remain valid.  The
contraction preserves $F$ and decreases the edge count, a contradiction.

If this arc had length at least $\lambda\sigma(Q)$ for an incident copy
$Q$, its first $\lceil\lambda\sigma(Q)\rceil$ edges would be an
immersed path shorter than $\sigma(Q)$, hence a simple path in $Q$.  Its
vertices are distinct in $X$, so the same edges also form a simple path in
the other copy incident to the arc.  These two different copies would
then share a simple path of length at least $\lambda\sigma(Q)$,
contradicting \cref{geo:copies}.  Arcs of length at most two are
already shorter than either threshold, since $\lambda\sigma>3$.
This proves (i) and (ii).

\emph{Counting.}
Let $v$ be the number of vertices after suppression.  Since the diagram
is connected and has a marked vertex, the graph obtained after
suppression is not a circle all of whose vertices have been suppressed.  All unmarked vertices that remain after suppression have
degree at least three.  Euler's formula and this degree bound give
\[
 v-b+F=1,\qquad 2b\geq3(v-m),\qquad b\leq3F+3m-3.
\]
Each bounded face has at least $\lambda^{-1}$ arc incidences by (i).
There are at most $2b$ bounded-face incidences in all, and therefore
\[
 F/\lambda\leq2b\leq6F+6m-6.
\]
This proves the first assertion of \eqref{geo:diagram-count}.  Since
$\lambda\leq1/24<1/12$, it also gives $F\leq m-1$, and then $b\leq6m$.
The argument includes trees and the single-vertex diagram.
\end{proof}

\subsection{Distance estimates and hyperbolicity}

In this subsection, we draw two consequences of the diagram estimate.
First, if two vertices of a copy $P$ are joined by a path of length $L$
in the coned-off graph that avoids the cone point of $P$, then their
distance in $P$ is at most $12(L+1)\lambda\sigma(P)$
(\cref{geo:avoidance}); here $\lambda\sigma(P)$ is the upper bound for
arcs shared by $P$ and another copy.  Second, the coned-off graph is
hyperbolic, with a thin-triangle constant $D_0$ that does not depend on
the relator graphs (\cref{geo:hyperbolicity}).

\begin{proposition}
\label{geo:avoidance}
Let $P$ be a copy and let $c_P$ be its cone point, namely the new vertex
of $\Xhat$ joined by an edge of length one to every vertex of $P$.
If an edge path of length $L$ in $\Xhat$ joins vertices $x,a\in P$ and does
not visit $c_P$, then
\begin{equation}
 d_P(x,a)\leq12(L+1)\lambda\sigma(P).
 \label{geo:avoidance-bound}
\end{equation}
Moreover, the formal Cayley graph $X$ is simple.
\end{proposition}

\begin{proof}
Let $(v_0,\ldots,v_L)$ be the given path, with $v_0=x$ and $v_L=a$, and
close it by the cone passage $(a,c_P,x)$.  The resulting closed path in
$\Xhat$ has at most $L+1$ sides, and this cone passage is its only side
through $c_P$.  Let $\gamma_*$ denote the corresponding adjustable side, a
path in $P$ from $a$ to $x$, and apply \cref{geo:short-arcs} to a minimal
diagram.

Consider an arc with an incidence on $\gamma_*$.  We claim that its other
incidence does not lie on $\gamma_*$.  If it did, the arc would be a
bridge---that is, deleting any one of its edges would disconnect the
diagram---and would be traversed by the outer contour at two indices
$t<t'$, both belonging to $\gamma_*$.  The vertices of the diagram that lie beyond the
bridge are then reached by the outer contour only at indices strictly
between $t$ and $t'$.  Since no breakpoint lies strictly between $t$ and
$t'$, none of these vertices is marked.  Deleting the bridge and the
component of the diagram beyond it removes the subpath
$(u_{t+1},\ldots,u_{t'})$ from $\gamma_*$ and changes no other side.  The
result is an admissible diagram for a shorter adjustable side in $P$ with
the same endpoints, with no more faces and fewer edges.  This contradicts
minimality.

Hence the other incidence of every arc on $\gamma_*$ lies on a bounded
face or on another side.  No other side has assigned copy $P$, and by the
proof of \cref{geo:short-arcs} no bounded face sharing an arc with
$\gamma_*$ is assigned to $P$.  By \cref{geo:short-arcs}(ii), every arc on
$\gamma_*$ therefore has length less than $\lambda\sigma(P)$, or length one
if it lies on a fixed side.  There are at most $2b$ arc incidences in
total, so
\[
 d_P(x,a)\le\abs{\gamma_*}\le2b\lambda\sigma(P)\le12(L+1)\lambda\sigma(P)
\]
by \eqref{geo:diagram-count}.  This proves \eqref{geo:avoidance-bound}.

To prove that $X$ is simple, consider a closed path in $X$ consisting of
at most two generator edges.  The first inequality in \eqref{geo:diagram-count}, with
$\lambda\leq1/24$, forces $F=0$.  Its minimal diagram is a tree, whose
outer word is freely trivial.  There are consequently no identity
generators, unintended coincidences of generator edges, or generator
involutions.  These are exactly the possible loops and multiple edges in
the formal undirected Cayley graph.
\end{proof}

We use the following standard criterion for hyperbolicity, a form of the
\emph{linear isoperimetric inequality}; see
\cite[III.H.2.1 and III.H.2.9]{BH}.  Let $Y$ be a connected graph whose
edges have length one.  Suppose that every closed path of length $n$ in
$Y$ is the outer contour of a van Kampen diagram mapped to $Y$ with at most
$Cn$ faces, each of whose contours has length at most $M$.  Then the
geodesic triangles of $Y$ are $\delta$-thin for a constant $\delta$
depending only on $C$ and $M$; that is, every point of a side of a
geodesic triangle lies within distance $\delta$ of the union of the other
two sides.  Neither local finiteness of $Y$ nor a finite list of
relations is required.

\begin{proposition}
\label{geo:hyperbolicity}
There is an integer $D_0$ independent of $d$, $S$, $\lambda$, and the
relator graphs, such that every vertex on a geodesic triangle in $\Xhat$
is at distance at most $D_0$ from a vertex on one of its other two sides.
\end{proposition}

\begin{proof}
We verify the preceding criterion with $C=12$ and $M=4$.  Let $\alpha$ be
a closed edge path of length $n>0$ in $\Xhat$; after a cyclic permutation
we may assume that it starts at a group vertex.  It has $m\le n$ sides.
Choose adjustable sides and a minimal diagram as above, and pass to arcs;
then $b\le6m$ by \cref{geo:short-arcs}.

For each arc $\beta$ we choose a path $\beta'$ in $\Xhat$ joining the images
of its endpoints, choosing the reverse path for the reversed arc.  If
$\beta$ lies on a fixed side, let $\beta'=\beta$, a single generator edge.
Otherwise one of the two regions containing $\beta$ is a bounded face or an
adjustable side with an assigned copy $P$, and we let $\beta'$ be the cone
passage through $c_P$ between the two endpoints.  On the contour of a
bounded face assigned to $P$, we then replace each path $\beta'$ by the
cone passage through $c_P$ with the same endpoints.  Each replacement
changes the path across a closed path of length at most four.  After all
replacements, every segment of the contour passes through $c_P$, so the
contour freely reduces to the trivial path.  Similarly, on an adjustable
side with assigned copy $P$, the replacements produce a path that freely
reduces to the original cone passage through $c_P$; if the adjustable side
has length zero, we insert the backtrack $(x,c_P,x)$.  Fixed sides are not
modified.  The number of closed paths of length at most four used is at
most the number of arc incidences on bounded faces and on adjustable
sides, hence at most $2b\le12n$.

These replacements also give a van Kampen diagram for $\alpha$ when the
original contours repeat vertices or edges.  By the van Kampen lemma, the
outer contour is freely equal to a product of conjugates of the contours
of the bounded faces.  Replace each of these contours, and each adjustable
side, by the products of closed paths of length at most four constructed
above.  The van Kampen construction applied to the resulting product gives
a diagram for $\alpha$ in $\Xhat$ with at most the stated number of faces,
each of contour length at most four.

The linear filling criterion now gives a hyperbolicity constant depending
only on $12$ and $4$.  Increasing it by one and taking an integer replaces
nearby points of edges by vertices and gives the asserted $D_0$.
\end{proof}

\subsection{Controlling intersections between copies}

In this subsection, we prove the main geometric result of the section.
We fix a root $R$, which orders the cone points by their distance from
$R$.  For a copy $P$ we choose a geodesic from its cone point toward $R$
and let $a$ be its first vertex in $P$.  We show that all intersections of
$P$ with copies whose cone points are no farther from $R$ lie in one ball
of $P$ centered at $a$ (\cref{geo:overlap}), and that this ball is a tree
with fewer than $\eta\abs P$ vertices (\cref{geo:overlap-tree}).

Fix the universal integers
\begin{equation}
 t=D_0+3,\qquad L_0=2+2t+2D_0.
 \label{geo:universal-constants}
\end{equation}

\begin{lemma}
\label{geo:overlap}
Let $R$ be any group vertex or cone point of $\Xhat$.  Let $P$ be a copy with
$c_P\neq R$, and choose one geodesic from $c_P$ to $R$.  Write $a\in P$
for its first group vertex.  Then every intersection $P\cap Q$ with a
different copy $Q$ satisfying
\[
 d_{\Xhat}(c_Q,R)\leq d_{\Xhat}(c_P,R)
\]
is contained in the same intrinsic ball
\begin{equation}
 B_P=B_P(a,r_P),\qquad
 r_P=\left\lfloor12(L_0+1)\lambda\sigma(P)\right\rfloor.
 \label{geo:overlap-radius}
\end{equation}
Here the ball means the subgraph of $P$ induced on its vertices.  Its
center $a$ is independent of both $Q$ and the chosen point of
intersection.  We call $B_P$ the \emph{overlap ball} of $P$ (with respect
to $R$).
\end{lemma}

\begin{proof}
Fix $x\in P\cap Q$ and abbreviate $A=c_P$ and $B=c_Q$.  The path
$A,x,B$ is geodesic of length two: different cone points have no direct edge.
Choose a geodesic from $B$ to $R$.  It cannot pass through $A$, because such a path
would have length $2+d(A,R)>d(B,R)$.  All distances in this proof, except
where marked $d_P$, are in $\Xhat$.

Suppose first that $d(A,R)\geq t$.  Let $z$ be the vertex at distance $t$
along the chosen geodesic from $A$ to $R$.  Its distance from every vertex
of $A,x,B$ is at least $t-2>D_0$.  Thinness of the triangle with vertices
$A,B,R$ therefore gives a vertex $y$ on the $B$-to-$R$ geodesic with
$d(y,z)\leq D_0$.  A shortest path from $y$ to $z$ avoids $A$, since
$d(A,z)=t>D_0$.  Concatenate the edge from $x$ to $B$, the segment from
$B$ to $y$, that shortest path, and the segment from $z$ back to $a$.
This path avoids $A$ and has length at most
\[
 1+(2+t+D_0)+D_0+(t-1)=L_0.
\]
Here $d(B,y)\leq d(B,A)+d(A,z)+d(z,y)\leq2+t+D_0$.

If $d(A,R)<t$, concatenate $x,B$, the geodesic from $B$ to $R$, and the
segment from $R$ back to $a$.  It again avoids $A$ and has length at most
\[
 1+d(B,R)+d(A,R)-1\leq2d(A,R)\leq L_0.
\]
Applying \cref{geo:avoidance} to either path gives
$d_P(x,a)\leq12(L_0+1)\lambda\sigma(P)$.  Intrinsic vertex distances are
integers, which gives \eqref{geo:overlap-radius}.  The initial choice of
$a$ involved only $P$, $R$, and one geodesic, as required.
\end{proof}

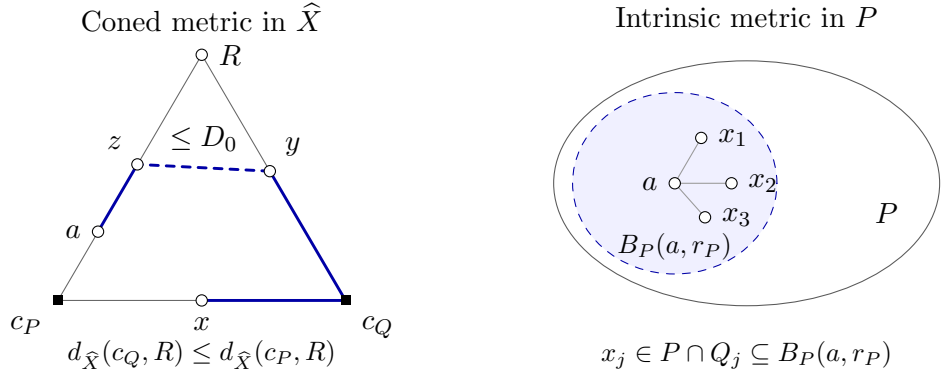
\begin{figure}[htbp]
\centering
\begin{tikzpicture}[
  x=1cm,y=1cm,font=\small,line cap=round,line join=round,
  group vertex/.style={circle,draw=black,fill=white,inner sep=1.5pt},
  cone center/.style={rectangle,fill=black,inner sep=1.8pt},
  route/.style={draw=blue!65!black,line width=1.1pt}
]
\node at (2.1,4.05) {Coned metric in $\Xhat$};
\node at (9.3,4.05) {Intrinsic metric in $P$};

\coordinate (coneP) at (0.2,0.3);
\coordinate (coneQ) at (4.0,0.3);
\coordinate (root) at (2.1,3.55);
\coordinate (overlap) at (2.1,0.3);
\coordinate (entrance) at (0.73,1.21);
\coordinate (farpoint) at (1.25,2.10);
\coordinate (nearpoint) at (3.0,2.01);

\draw[black!60] (coneP)--(root)--(coneQ)--(coneP);
\draw[route] (overlap)--(coneQ)--(nearpoint);
\draw[route,dashed] (nearpoint)--(farpoint)
  node[midway,above=2pt] {$\leq D_0$};
\draw[route] (farpoint)--(entrance);

\node[cone center,label=below left:{$c_P$}] at (coneP) {};
\node[cone center,label=below right:{$c_Q$}] at (coneQ) {};
\node[group vertex,label=right:{$R$}] at (root) {};
\node[group vertex,label=below:{$x$}] at (overlap) {};
\node[group vertex,label=left:{$a$}] at (entrance) {};
\node[group vertex,label=above left:{$z$}] at (farpoint) {};
\node[group vertex,label=above right:{$y$}] at (nearpoint) {};
\node[font=\footnotesize] at (2.1,-0.42)
  {$d_{\Xhat}(c_Q,R)\leq d_{\Xhat}(c_P,R)$};

\draw[black!65] (9.3,1.85) ellipse (2.55 and 1.62);
\draw[draw=blue!65!black,fill=blue!6,dashed]
  (8.35,1.85) ellipse (1.35 and 1.20);
\node at (11.15,1.45) {$P$};
\node[font=\footnotesize] at (8.35,1.02) {$B_P(a,r_P)$};
\coordinate (intrinsicA) at (8.35,1.85);
\coordinate (pointone) at (8.70,2.45);
\coordinate (pointtwo) at (9.10,1.85);
\coordinate (pointthree) at (8.75,1.40);
\draw[black!40] (intrinsicA)--(pointone);
\draw[black!40] (intrinsicA)--(pointtwo);
\draw[black!40] (intrinsicA)--(pointthree);
\node[group vertex,label=left:{$a$}] at (intrinsicA) {};
\node[group vertex,label=right:{$x_1$}] at (pointone) {};
\node[group vertex,label=right:{$x_2$}] at (pointtwo) {};
\node[group vertex,label=right:{$x_3$}] at (pointthree) {};
\node[font=\footnotesize] at (9.3,-0.42)
  {$x_j\in P\cap Q_j\subseteq B_P(a,r_P)$};
\end{tikzpicture}
\caption{Two metrics in \cref{geo:overlap}.  Left: the chosen geodesic in
the coned-off graph from $c_P$ to $R$ determines $a$.  In the case
$d(c_P,R)\ge t$ of the proof, thinness of triangles supplies the dashed
path from $z$ to $y$, and the blue path
$x,c_Q,y,z,a$ has length at most $L_0$ and avoids $c_P$.
Right: \cref{geo:avoidance} places every overlap with a copy whose cone point is no farther
from $R$ in the same intrinsic ball $B_P(a,r_P)$.  In a finite collection,
choosing $c_P$ farthest from $R$ includes every other copy.}
\label{geo:overlap-figure}
\end{figure}

\begin{lemma}
\label{geo:tree-balls}
Fix $d\geq3$ and $\eta>0$.  There is a girth threshold, depending only on
$d$ and $\eta$, such that every finite simple $d$-regular graph of girth
$\sigma$ above that threshold has the following property: every induced
ball of radius at most $\sigma/16$ is a tree with fewer than $\eta n$
vertices, where $n$ is the order of the graph.
\end{lemma}

\begin{proof}
If $2r+1<\sigma$, the induced ball of integer radius $r$ is a tree.  An
edge outside a breadth-first spanning tree would otherwise produce a
cycle of length at most $2r+1$.  Its order is at most
\[
 N_d(r)=1+d\sum_{j=0}^{r-1}(d-1)^j.
\]
Apply the same tree observation at radius
$R=\lfloor(\sigma-2)/2\rfloor$.  Regularity gives $n\geq N_d(R)$,
because all vertices before the last level have their full set of
distinct forward neighbors.  Consequently
\[
 \frac{\abs{B(v,\lfloor\sigma/16\rfloor)}}{n}
 \leq\frac{N_d(\lfloor\sigma/16\rfloor)}
              {N_d(\lfloor(\sigma-2)/2\rfloor)}
 \longrightarrow0\qquad(\sigma\longrightarrow\infty).
\]
This proves the claim, uniformly in the graph and its center.
\end{proof}

\begin{corollary}
\label{geo:overlap-tree}
Choose $\lambda$ sufficiently small that
\begin{equation}
 12(L_0+1)\lambda\leq1/16.
 \label{geo:lambda-choice}
\end{equation}
If the relator graphs satisfy the girth threshold in
\cref{geo:tree-balls}, every ball in \cref{geo:overlap} is an induced tree
with fewer than $\eta\abs{P}$ vertices.
\end{corollary}

\begin{proof}
The radius in \eqref{geo:overlap-radius} is at most $\sigma(P)/16$.
Apply \cref{geo:tree-balls} in the source graph of $P$.
\end{proof}

\subsection{The space of finitely supported cycles}

In this subsection, we show that every finitely supported cycle in the
Cayley graph, with coefficients in $\Ftwo$, can be written in exactly one
way as a sum of cycles in copies (\cref{geo:cycle-sum}).  Existence
follows from van Kampen diagrams, and uniqueness follows from the tree
property of the overlap balls.

For a graph $Y$, let $Z_1(Y;\Ftwo)$ denote its space of finitely supported cycles:
the finitely supported edge chains over $\Ftwo$ with zero boundary;
equivalently, the finite edge sets in which every vertex is incident to
an even number of selected edges (with loops counted twice).
The following is an algebraic direct sum, with no completion.

\begin{proposition}
\label{geo:cycle-sum}
Under the hypotheses of \cref{geo:overlap-tree}, inclusion of the copies
induces an internal direct-sum decomposition
\begin{equation}
 Z_1(X;\Ftwo)=\bigoplus_P Z_1(P;\Ftwo).
 \label{geo:cycle-decomposition}
\end{equation}
\end{proposition}

\begin{proof}
A finite cycle in $X$ is a sum of edge chains of closed paths.  Each such
path has trivial group label and hence a finite van Kampen diagram over
the presentation of $\Gamma$.  Taking its face boundaries modulo two
expresses its edge chain as a sum of closed-path chains in copies.  This
proves spanning.

For independence, suppose a finite sum of nonzero chains
$z_P\in Z_1(P;\Ftwo)$, taken in distinct copies, vanishes.  Fix a group
vertex $R$ and choose among these copies one whose cone point is farthest
from $R$.  Call it $P$.  By \cref{geo:overlap}, all intersections of $P$
with the other participating copies lie in one ball $B_P$; by
\cref{geo:overlap-tree}, that ball is an induced tree.  Every edge in the
support of $z_P$ must cancel against an edge of another participating
chain.  Both endpoints of every such edge therefore lie in $B_P$.
Inducedness places the edge itself in that tree.  A finite tree has zero
cycle space, contradicting $z_P\neq0$.
\end{proof}

\section{Expanders and subspaces of prescribed dimension}
\label{sec:data}

In this section, we choose the relator graphs and the linear spaces of
functions used to define the compact group $\Omega$.  For every
sufficiently large integer $k$, we choose a
Lubotzky--Phillips--Sarnak expander $\Theta_k$ with between $2k$ and $8k/3$
vertices.  We label the entire family by a fixed finite alphabet $S$.
We then choose a finite-dimensional vector space $E=\Ftwo^m$ with enough
nonzero vectors to specify the sets of outgoing edge labels.
The reference function $w_k$ and the elements of $\mathcal C_k$ are
functions from $\Vertices(\Theta_k)$ to $E$, whereas the dimension of
$\mathcal C_k$ is $k$ over $\Ftwo$.  Thus $\mathcal C_k$ has $2^k$
elements, independently of $m$.

We require every nonzero element of $\mathcal C_k$ and of its orthogonal
complement to be nonzero at more than $\eta n_k$ vertices.  The first
condition forces distinct functions in $\mathcal C_k$ to differ at many
vertices.  The second allows arbitrary values on a sufficiently small
vertex set to extend to an element of $\mathcal C_k$.  These are the two
properties needed for the spectral estimate and the probability
calculation, respectively.  The choices are made in the order
$S,m,\eta,K$, after the graphs and geometric constants have
been fixed.

\subsection{Expanders \texorpdfstring{$\Theta_k$}{Theta k}}

In this subsection, we choose the expanders $\Theta_k$, one for each
sufficiently large $k$, and label them before choosing the vector space
$E$.  The graphs are Lubotzky--Phillips--Sarnak Ramanujan graphs, and the
prime number theorem in an arithmetic progression shows that their
orders are dense enough to provide one graph of the required size for
each $k$.

We use two classical number-theoretic facts.  The Lubotzky--Phillips--Sarnak
construction \cite{LPS}, with the fixed prime $5$ and primes
$q\equiv1\pmod{20}$, gives connected $6$-regular graphs $X^{5,q}$, which are
Cayley graphs of the groups $\PSL_2(\mathbb F_q)$.  Their numbers of
vertices $n(q)$, their girths $\sigma(q)$, and their adjacency operators
$\Adj$ satisfy
\begin{equation}\label{data:lps}
 n(q)=\frac{q(q^2-1)}2,\qquad
 \sigma(q)\ge2\log_5q,\qquad
 6I-\Adj\ge (6-2\sqrt5)I,
\end{equation}
where the last inequality holds on the orthogonal complement of the
constant functions.
Their diameters are at most $2\log_5 n(q)+2\log_5 2+1$.
The prime number theorem in the progression $1\pmod{20}$ implies that
successive such primes, and consequently successive orders $n(q)$, have
ratio tending to one.  For a quantitative statement implying this fact,
see \cite{BMOR}.  Indeed, write
$\vartheta(x;20,1)=\sum\log q$, the sum over primes $q\le x$ with
$q\equiv1\pmod{20}$.  The prime number theorem in this progression states
that $\vartheta(x;20,1)=x/8+o(x)$, and this implies
\[
 \vartheta((1+\delta)x;20,1)-\vartheta(x;20,1)
 =\delta x/8+o(x)>0
\]
for each fixed $\delta>0$ and all sufficiently large $x$.

\begin{proposition}\label{data:family}
There exist a finite alphabet $S$, integers $m\ge1$ and $K\ge2$, and
labeled connected simple $d$-regular graphs $(\Theta_k)_{k\ge K}$, with
$d=6$, satisfying the hypotheses of \Cref{geo:hypotheses}.  We choose
$m$ so that
\begin{equation}\label{code:alphabet}
 2^m-1\ge
 \#\{J\subseteq S^{\pm1}:\abs J=d\},
 \qquad E=\Ftwo^m,
\end{equation}
and put
\[
 \begin{gathered}
 n_k=\abs{\Vertices(\Theta_k)},\qquad \sigma_k=\girth\Theta_k,\\
 \rho=6-2\sqrt5,\qquad \kappa=\rho/2,\\
 \eta=\frac1{100m},\qquad c=\frac{\eta}{2}\min\{\kappa,1\}.
 \end{gathered}
\]
The graphs can be chosen so that
\begin{enumerate}[label=\textup{(\roman*)}]
\item $2k\le n_k\le8k/3$ and $n_k>100$;
\item $dI-\Adj(\Theta_k)\ge\rho I$ on the complement of the constants,
and $|\partial W|\ge\kappa\min\{|W|,n_k-|W|\}$ for every vertex subset $W$,
where the edge boundary $\partial W$ is the set of edges with exactly one
endpoint in $W$;
\item $\sigma_k>\max\{16,20/c\}$, $\lambda\sigma_k>3$, and the overlap
balls from \Cref{geo:overlap} have radius at most $\sigma_k/16$;
\item every induced intrinsic ball of radius at most $\sigma_k/16$ is a
tree with fewer than $\eta n_k$ vertices.
\end{enumerate}
All these properties persist for the graphical presentation formed from
any subfamily, using the same labels and constants.
\end{proposition}
\begin{proof}
Let $D_0$ and $L_0$ be the universal constants in \Cref{sec:geometry}, and
choose once and for all
\begin{equation}\label{data:lambda}
 0<\lambda<1/24,\qquad 12(L_0+1)\lambda\le1/16.
\end{equation}
For every sufficiently large integer $k$, choose the least
Lubotzky--Phillips--Sarnak order $n(q)$ at least $2k$, and let $\Theta_k$
be a graph isomorphic to $X^{5,q}$.  The ratio assertion above gives
$n_k\le8k/3$ for all large $k$.  Different indices $k$ may give the same
unlabeled graph; such graphs are treated as distinct members of the
family and are labeled separately.

The constants in \eqref{data:lps} are uniform.  If $|W|=w$, the function
$\one_W-w/n_k$ has squared norm $w(n_k-w)/n_k$ and Laplacian energy
$|\partial W|$.  Thus
\[
 |\partial W|\ge\rho\frac{w(n_k-w)}{n_k}
 \ge\frac\rho2\min\{w,n_k-w\}.
\]
The diameter-to-girth ratio is bounded.  Also
$\sigma_k\ge c_0\log k$ for some $c_0>0$, by \eqref{data:lps} and
$n_k\asymp k$.  Choose an initial cutoff $K_0$ so that these estimates
hold and $\lambda\sigma_k>3$ for $k\ge K_0$.

By \Cref{lab:logarithmic,lab:main}, the family indexed by $k\ge K_0$ admits
reduced labels in a finite alphabet $S$ for which the labels of long
simple paths are globally unique.  The labeling theorem uses the degree
bound, the diameter-to-girth bound, the parameter $\lambda$, and the
girth multiplicity bound; none of these depends on $m$ or $\eta$.
Fix these labels.  Now choose $m$ as in \eqref{code:alphabet}, set
$\eta=1/(100m)$, and define $c$ as in the statement.

In a graph of girth $\sigma$, an induced ball of integer radius $r$ is a
tree whenever $2r+1<\sigma$: an additional edge, together with the two
root paths, would yield a shorter cycle.  Its number of vertices, for
$r\le\sigma/16$, is at most
\[
 1+d\sum_{j=0}^{\lfloor\sigma/16\rfloor-1}(d-1)^j.
\]
Regularity and the same tree argument at radius
$R=\lfloor(\sigma-2)/2\rfloor$ give a lower bound
$1+d\sum_{j=0}^{R-1}(d-1)^j$ for the graph order.  The ratio of these two
bounds tends to zero as $\sigma\to\infty$.  We may therefore choose a
final cutoff $K\ge\max\{K_0,2\}$ so that (iii), (iv), and $n_k>100$ hold
for all $k\ge K$.

We use the labels already chosen for the indices $k\ge K$.  By
\Cref{lab:embedding}, the relator graphs embed isometrically, giving
\Cref{geo:hypotheses}; \eqref{data:lambda} gives the asserted overlap
radius.  Restricting the labels to any further subfamily preserves their
uniqueness.  Applying the embedding theorem to that subfamily's own
presentation gives the last assertion, since the other estimates are
intrinsic to the individual graphs.
\end{proof}

\subsection{Reference functions and subspaces of prescribed dimension}
\label{code:section}

In this subsection, we construct
$w_k:\Vertices(\Theta_k)\to E$ and a $k$-dimensional
$\Ftwo$-linear subspace $\mathcal C_k$ containing $w_k$.  The value
$w_k(v)$ specifies the set of labels of the edges starting at $v$.
We choose $\mathcal C_k$ so that every nonzero function in
$\mathcal C_k$ or $\mathcal C_k^\perp$ is nonzero at a fixed positive
proportion of the vertices.

Fix the labeled full family and the vector space $E=\Ftwo^m$ from
\Cref{data:family}.  At a vertex $v$ of $\Theta_k$, let
\[
 J_k(v)\subseteq S^{\pm1}
\]
be the set of labels of the oriented edges of $\Theta_k$ starting at $v$.
Reducedness and $d$-regularity imply $\abs{J_k(v)}=d$.  By
\eqref{code:alphabet}, there is an injective assignment
$J\mapsto a_J$ from all $d$-element subsets of $S^{\pm1}$ to
$E\setminus\{0\}$.  Fix one such assignment, and define
$\mathcal L:E\to\{J\subseteq S^{\pm1}:\abs J=d\}\cup\{\varnothing\}$
and the \emph{reference function} $w_k\in E^{\Vertices(\Theta_k)}$ by
\begin{equation}\label{code:decoder}
 \mathcal L(a)=
 \begin{cases}
 J,&a=a_J\text{ for a }d\text{-element set }J,\\
 \varnothing,&\text{otherwise},
 \end{cases}
 \qquad
 w_k(v)=a_{J_k(v)}.
\end{equation}
Thus $\mathcal L(w_k(v))=J_k(v)$ for every vertex $v$, and $w_k(v)\ne0$
at every vertex.  These choices do not depend on a later subset of
indices or on the real number whose trace will be prescribed.

For a finite vertex set $F$, we regard $E^F$ as a vector space over
$\Ftwo$ with the nondegenerate bilinear form
\[
 \langle u,z\rangle=\sum_{v\in F}u(v)\cdot z(v),
 \qquad
 x\cdot y=\sum_{j=1}^m x_jy_j\quad(x,y\in E).
\]
Orthogonal complements are taken with respect to this form.
The support of a function $z:F\to E$ is
$\supp z=\{v\in F:z(v)\ne0\}$.  In particular, its size counts vertices,
not individual scalar coordinates of the values in $E$.

\begin{lemma}\label{code:exist}
For each $k\ge K$, there is an $\Ftwo$-linear subspace
$\mathcal C_k\subseteq E^{\Vertices(\Theta_k)}$ of dimension $k$ such that
$w_k\in\mathcal C_k$ and
\begin{equation}\label{code:distance}
 |\supp z|>\eta n_k
 \quad\text{for every nonzero }z\in\mathcal C_k\cup\mathcal C_k^\perp.
\end{equation}
Thus $\abs{\mathcal C_k}=2^k$.  Distinct functions in $\mathcal C_k$
differ at more than $\eta n_k$ vertices; this is the separation property
used in the spectral estimate.
\end{lemma}
\begin{proof}
Write $n=n_k$ and $w=w_k$.  The bounds $2k\le n\le8k/3$ in
\Cref{data:family}(i) give
\[
 3n/8\le k\le n/2.
\]
The ambient space $E^{\Vertices(\Theta_k)}$ has dimension $mn$ over
$\Ftwo$.  Sample $\mathcal C$ uniformly among its $k$-dimensional
subspaces containing the line $\langle w\rangle$.  For a fixed
$z\notin\langle w\rangle$, passing to the quotient by this line gives
\begin{equation}\label{code:random-primal}
 \Pr(z\in\mathcal C)
 =\frac{2^{k-1}-1}{2^{mn-1}-1}<2^{k-mn}.
\end{equation}
This follows by counting nonzero vectors in a uniform subspace.
Orthogonal complementation makes $\mathcal C^\perp$ a uniform
$(mn-k)$-dimensional subspace of the $(mn-1)$-dimensional space
$w^\perp$.  Hence, for $z\ne0$,
\begin{equation}\label{code:random-dual}
 \Pr(z\in\mathcal C^\perp)
 \le\frac{2^{mn-k}-1}{2^{mn-1}-1}<2^{1-k},
\end{equation}
where the probability is zero when $z\notin w^\perp$.  This argument uses
nondegeneracy of the ambient form and also applies when
$\langle w,w\rangle=0$.

We next count the functions that are nonzero at at most
$r=\lfloor\eta n\rfloor$ vertices.  Put $u=\eta/(1-\eta)$.  Since
$0<u<1$,
\[
 u^r\sum_{j=0}^r\binom nj
 \le\sum_{j=0}^r\binom nj u^j\le(1+u)^n,
\]
and therefore
\[
 \sum_{j=0}^r\binom nj\le2^{H_2(\eta)n},
 \qquad
 H_2(t)=-t\log_2t-(1-t)\log_2(1-t).
\]
The function $H_2$ is increasing on $(0,1/2)$, and
$H_2(1/100)<9/100$.  For the latter bound, one may use
$\log_2 100<7$, $-\log(1-t)\le t/(1-t)$, and $\log2>1/2$.
The choice $\eta=1/(100m)$ makes $m\eta=1/100$ independent of $m$.
Thus the number of possible nonzero values at a vertex is offset by
the smaller support threshold, and
$m\eta+H_2(\eta)<1/10<1/4$.  Consequently,
\begin{equation}\label{code:hamming}
 \begin{aligned}
 \#\{z:|\supp z|\le r\}
 &=\sum_{j=0}^r\binom nj(2^m-1)^j\\
 &\le2^{(m\eta+H_2(\eta))n}<2^{n/4}.
 \end{aligned}
\end{equation}
Every nonzero function counted here lies outside $\langle w\rangle$,
because $w$ is nonzero at every vertex.  The union bound and
\eqref{code:random-primal}--\eqref{code:random-dual} show that the
probability that either space contains a nonzero function with at most
$\eta n$ nonzero vertex values is less than
\[
 2^{n/4+k-mn}+2^{n/4+1-k}
 \le2^{-n/4}+2^{1-n/8}<1.
\]
Here we used $m\ge1$, the dimension bounds above, and $n>100$.
Thus a suitable subspace exists.  This extends the random-subspace
construction in \cite[Lemma~4.1]{OpenAI2026} to subspaces of prescribed
dimension consisting of functions with values in $\Ftwo^m$.
\end{proof}

The condition on $\mathcal C_k^\perp$ gives the following local extension
property.

\begin{lemma}\label{code:local-extension}
Let $D\subseteq\Vertices(\Theta_k)$ satisfy $\abs D\le\eta n_k$.
Then the restriction map
\[
 \mathcal C_k\longrightarrow E^D,\qquad u\longmapsto u|_D,
\]
is surjective.
\end{lemma}
\begin{proof}
If the image were a proper subspace of $E^D$, a nonzero linear
functional on $E^D$ would vanish on that image.  By the nondegenerate
bilinear form on $E^D$, it would have the form
$u\mapsto\sum_{v\in D}z(v)\cdot u(v)$ for a nonzero $z\in E^D$.
Extending $z$ by zero outside $D$ would give a nonzero element of
$\mathcal C_k^\perp$ with support of size at most $\eta n_k$,
contrary to \eqref{code:distance}.
\end{proof}

From now on we fix a subset $A\subseteq\{K,K+1,\ldots\}$ and keep only
the graphs $\Theta_i$ with $i\in A$, together with their labels and
subspaces.  We write $n_i$ for the number of vertices of $\Theta_i$ and
$k_i=i$ for the dimension of $\mathcal C_i$ over $\Ftwo$.  The uniform
estimates just established apply to every such subset.

\section{The crossed product, the subgraphs, and the group \texorpdfstring{$H$}{H}}
\label{mod:section}

In this section, we define the compact group $\Omega$ and an equivariant
family of subgraphs $\mathcal G_\omega$ of a Cayley graph.  We prove that
restriction to a copy $P$ of $\Theta_i$ is uniformly distributed on
$\mathcal C_P$, so that the probability of $\omega|_P=w_P$ is
$2^{-k_i}$.  The proof extends functions over finite collections of
copies and then uses compactness.  We next describe the discrete
Pontryagin dual $V$ of $\Omega$, prove that $H=V\rtimes\Gamma$ is
finitely generated, and give an explicit presentation of $V$.

We first explain the operator-algebraic framework.  The operators
$dI-\Adj(\mathcal G_\omega)$ are representations of one
element $T$ of
$C(\Omega)\rtimes_r\Gamma\cong\Cr(V\rtimes\Gamma)$.
\Cref{mod:representations,mod:laplacians} allow uniform spectral
estimates for all $\omega$ to pass to this reduced group algebra, and
express its canonical trace as a Haar integral.  This section and the
next two develop the construction of \cite[Sections~4--6]{OpenAI2026}
using spaces of functions with values in $\Ftwo^m$ and arbitrary
dimensions $k_i$ over $\Ftwo$.

\subsection{Graph operators in a crossed product}
\label{sec:regular}

In this subsection, we carry out Step 1 of the strategy in general.  Let
$\Gamma$ be a group generated by a finite set $S$, and let $X=\Cay(\Gamma,S)$ be
its Cayley graph.  We assume that $X$ is a simple graph, that is, the
elements of $S^{\pm1}$ are nontrivial and pairwise distinct in $\Gamma$; for
the groups of \Cref{sec:labeling} this is part of \cref{geo:avoidance}.  Let $\Omega$ be a compact metrizable
abelian group, written additively, with Haar probability measure $\mu$, on
which $\Gamma$ acts by continuous automorphisms $\omega\mapsto g\omega$.  Let
$V$ be the Pontryagin dual of $\Omega$, the discrete abelian group of
continuous homomorphisms from $\Omega$ to the circle, written additively;
for $v\in V$ we write $\chi_v(\omega)$ for the value of $v$ at $\omega$.  The
group $\Gamma$ acts on $V$ by $\chi_{gv}(\omega)=\chi_v(g^{-1}\omega)$, and we
put $H=V\rtimes\Gamma$, with multiplication $(v,g)(v',g')=(v+gv',gg')$.

For $g\in\Gamma$, write $u_g$ for the unitary associated with $(0,g)$ in
the group algebra of $H$.  Fourier transformation on $V$ identifies
$\Cr(V)$ with $C(\Omega)$, taking the unitary associated with $v\in V$
to $\chi_v$.  If $f\in C(\Omega)$, the covariance relation is
\[
 u_g f u_g^*=\alpha_g(f),
 \qquad \alpha_g(f)(\omega)=f(g^{-1}\omega),
\]
so that $\Cr(H)\cong C(\Omega)\rtimes_r\Gamma$.  Elements of the commutative
subalgebra $C(\Omega)$ are functions on $\Omega$, but general elements of the
crossed product are not.  Instead, each point $\omega\in\Omega$ gives the
following representation $\pi_\omega$ on $\ell^2(\Gamma)$.  After the unitary
$\xi\mapsto\xi(\,\cdot\,^{-1})$ of $\ell^2(\Gamma)$, which turns the right
translations in \eqref{mod:rep-group} into left translations, $\pi_\omega$ is
the regular representation of the crossed product induced from evaluation
of functions at $\omega$.  The norm formula of \cref{mod:representations}
allows uniform estimates for these operators to pass to $\Cr(H)$, and its
trace formula reduces traces to integrals over $\Omega$.

\begin{lemma}
 \label{mod:representations}
 For every $\omega\in\Omega$ there is a representation $\pi_\omega$ of
 $\Cr(H)\cong C(\Omega)\rtimes_r\Gamma$ on $\ell^2(\Gamma)$ with formulas
 \begin{align}
  (\pi_\omega(f)\xi)(h)&=f(h^{-1}\omega)\xi(h),
       &&f\in C(\Omega),\label{mod:rep-function}\\
  (\pi_\omega(u_g)\xi)(h)&=\xi(hg),
       &&g\in\Gamma.\label{mod:rep-group}
 \end{align}
 For every $b\in\Cr(H)$,
 \begin{equation}
  \label{mod:norm}
  \norm b=\sup_{\omega\in\Omega}\norm{\pi_\omega(b)},
 \end{equation}
 so the direct sum of these representations is faithful.
 The canonical group trace satisfies
 \begin{equation}
  \label{mod:trace-integral}
  \tau_H(b)=\int_\Omega
       \langle\pi_\omega(b)\delta_1,\delta_1\rangle\,d\mu(\omega).
 \end{equation}
\end{lemma}
\begin{proof}
The displayed operators obey the covariance relation, and so initially
define representations of the algebraic crossed product
$\C[V]\rtimes\Gamma=\C[H]$.
We check that these representations are bounded for the reduced norm.

Fourier transformation $\ell^2(V)\to L^2(\Omega,\mu)$ takes the coordinate
vector at $v$ to $\chi_v$.  The characters are an orthonormal basis:
orthogonality follows from Haar invariance, and their span is dense by the
Stone--Weierstrass theorem.  Apply this transformation to the $V$-variable
of $\ell^2(V\times\Gamma)$, and then make the unitary change of variables
\[
  (U_0F)(\omega,h)=F(h^{-1}\omega,h^{-1}).
\]
It is unitary because the action on $\Omega$ preserves Haar probability.
Let $\lambda_H$ denote the left regular representation of $\Cr(H)$ on
$\ell^2(H)$, and let $L^2(\Omega,\mu;\ell^2(\Gamma))$ be the Hilbert space of
square-integrable functions from $\Omega$ to $\ell^2(\Gamma)$.  Under the
resulting unitary
$U:\ell^2(H)\to L^2(\Omega,\mu;\ell^2(\Gamma))$, the left regular action
of $(v,1)$ becomes multiplication by $\chi_v(h^{-1}\omega)$, and that of
$(0,g)$ becomes the right shift $\xi(h)\mapsto\xi(hg)$.
These statements follow either by substitution into
$(v,g)^{-1}(w,h)=(g^{-1}(w-v),g^{-1}h)$, or directly from the two
changes of variables.  Consequently, for $b\in\C[H]$ and
$F\in L^2(\Omega,\mu;\ell^2(\Gamma))$,
\[
 \bigl(U\lambda_H(b)U^{-1}F\bigr)(\omega)=\pi_\omega(b)\bigl(F(\omega)\bigr)
 \qquad\text{for almost every }\omega\in\Omega.
\]

It follows immediately that
$\norm{\lambda_H(b)}\le\sup_\omega\norm{\pi_\omega(b)}$.
For the reverse inequality, we use the fact that the action on a finitely
supported vector varies continuously with $\omega$.  A large norm in
one fiber therefore persists on a set of positive Haar measure.
Fix $\omega_0$ and a finitely supported unit vector
$\xi\in\ell^2(\Gamma)$.
The function $\omega\mapsto\norm{\pi_\omega(b)\xi}$ is continuous,
because $b$ is a finite sum and only finitely many coordinates occur.
If its value at $\omega_0$ exceeds $t$, it exceeds $t$ on some nonempty
open set $O\subseteq\Omega$.  Such an $O$ has positive Haar measure:
finitely many translates of $O$ cover the compact group $\Omega$.
The unit vector
$\mu(O)^{-1/2}\one_O(\omega)\xi$ in $L^2(\Omega,\mu;\ell^2(\Gamma))$ then
has image of norm greater than $t$.  Hence
$\norm{\lambda_H(b)}>t$.
Finitely supported vectors are dense, so this proves
\eqref{mod:norm} on $\C[H]$.  Each $\pi_\omega$ therefore extends
contractively to $\Cr(H)$, and norm approximation proves
\eqref{mod:norm} there as well.  Uniform approximation by finite sums of
characters also gives \eqref{mod:rep-function} for every $f\in C(\Omega)$.

Finally, if $b=\sum_g f_g u_g\in\C[H]$, its diagonal coefficient at
$\delta_1$ in the fiber is $f_1(\omega)$.  The Haar integral of $f_1$ is
the coefficient of the identity element of $H$, namely $\tau_H(b)$.
Both sides of \eqref{mod:trace-integral} are continuous in the reduced
norm: the diagonal functions converge uniformly under norm approximation,
by contractivity of all the $\pi_\omega$.  This proves the formula for
arbitrary $b\in\Cr(H)$.
\end{proof}

Now let $d\ge1$, and suppose that every $\omega\in\Omega$ determines a
subgraph $\mathcal G_\omega$ of $X$, with vertex set $\Gamma$, in which every
vertex has degree at most $d$.  We assume that
\begin{equation}
 \label{mod:equivariant}
 \mathcal G_{g\omega}=g\,\mathcal G_\omega
 \qquad(g\in\Gamma,\ \omega\in\Omega),
\end{equation}
and that for each $s\in S^{\pm1}$ the set of $\omega$ for which the edge
$\{1,s\}$ belongs to $\mathcal G_\omega$ is open and closed.  Let
$F_s\in C(\Omega)$ be the indicator function of this set.

\begin{lemma}
 \label{mod:laplacians}
 The element
 \[
  T=d\cdot1-\sum_{s\in S^{\pm1}}F_s\,u_s\ \in\ C(\Omega)\rtimes_r\Gamma
 \]
 is self-adjoint, and for every $\omega\in\Omega$,
 \[
  \pi_\omega(T)=dI-\Adj(\mathcal G_\omega).
 \]
 In particular, $0\le\pi_\omega(T)\le2dI$ for every $\omega$, and the spectrum
 of $T$ is contained in $[0,2d]$.  If each $F_s$ is a finite linear
 combination of the characters $\chi_v$, then $T$ lies in the group ring
 $\C[H]$.
\end{lemma}
\begin{proof}
By \eqref{mod:rep-function} and \eqref{mod:rep-group},
$(\pi_\omega(F_su_s)\xi)(h)=F_s(h^{-1}\omega)\,\xi(hs)$.  By
\eqref{mod:equivariant}, $F_s(h^{-1}\omega)=1$ exactly when the edge $\{1,s\}$
belongs to $h^{-1}\mathcal G_\omega$, that is, when the edge $\{h,hs\}$ belongs to
$\mathcal G_\omega$.  Since distinct $s$ give distinct neighbors $hs$ of $h$,
summing over $s$ gives $\pi_\omega(T)=dI-\Adj(\mathcal G_\omega)$.  The adjoint
of $F_su_s$ is $u_{s^{-1}}F_s=\alpha_{s^{-1}}(F_s)\,u_{s^{-1}}$, and
$\alpha_{s^{-1}}(F_s)(\omega)=F_s(s\omega)$ is the indicator function of the
condition $\{1,s\}\in s\mathcal G_\omega$, that is,
$\{s^{-1},1\}\in\mathcal G_\omega$.  Hence $\alpha_{s^{-1}}(F_s)=F_{s^{-1}}$, and
$T$ is self-adjoint.  For a graph of degree at most $d$, the adjacency
operator has norm at most $d$ and $dI-\Adj$ is positive, since its
quadratic form is the sum of $(d-\deg h)\abs{\xi(h)}^2$ over the vertices
and of $\abs{\xi(h)-\xi(h')}^2$ over the edges $\{h,h'\}$.  This gives
$0\le\pi_\omega(T)\le2dI$, and \eqref{mod:norm} transfers these bounds to $T$.
The last assertion holds because $\chi_v$ is the image of the group
element $(v,1)$ of $H$.
\end{proof}

\subsection{The compact group \texorpdfstring{$\Omega$}{Omega} and the subgraphs}
\label{sec:omega}

In this subsection, we define $\Omega$ by linear conditions on the
restrictions of functions to the embedded graphs.  Conditions at both
endpoints of an edge then define $\mathcal G_\omega$.

Fix a subset $A\subseteq\{K,K+1,\ldots\}$ as at the end of
\Cref{sec:data}.  Write $n_i=\abs{\Vertices(\Theta_i)}$, and let
$\mathcal C_i\subseteq E^{\Vertices(\Theta_i)}$ be the subspace of
dimension $k_i=i$ over $\Ftwo$ chosen there, where $E=\Ftwo^m$ is the
fixed vector space of \Cref{sec:data}.
From now on, $\Gamma$, $S$, and $X$ are those of the geometric
construction:
\begin{align*}
 \Gamma&=\Gamma_A=\langle S\mid\text{labels of closed paths in }
                    \Theta_i,\ i\in A\rangle,\\
 X&=\Cay(\Gamma,S),
 \qquad P_{g,i}=g\varphi_i(\Theta_i).
\end{align*}
A function on the source graph is transported to a copy by its
identification $\iota_P=g\varphi_i$: a function $f$ on
$\Vertices(\Theta_i)$ gives the function $f\circ\iota_P^{-1}$ on
$P=P_{g,i}$.  In this way the subspace $\mathcal C_i$, its orthogonal
complement, the sets $J_i(v)$ of edge labels, and the reference function
$w_i$ give $\mathcal C_P,\mathcal C_P^\perp,J_P,w_P$ on every copy $P$.
All these orthogonal complements use the bilinear pairing
$\langle u,z\rangle=\sum_{h\in P}u(h)\cdot z(h)$ over $\Ftwo$.
We use the map
$\mathcal L:E\to\{J\subseteq S^{\pm1}:\abs J=d\}\cup\{\varnothing\}$
from \eqref{code:decoder}; thus $\mathcal L(w_P(h))=J_P(h)$.

The hypotheses are supplied by \cref{data:family,code:exist} and the
geometric lemmas.  Every nonzero function in either $\mathcal C_P$ or
$\mathcal C_P^\perp$ is nonzero at more than $\eta\abs P$ vertices.
For every root $R$ and copy $P$ with $c_P\ne R$ as in
\cref{geo:overlap}, that lemma supplies an induced intrinsic ball
$B_P$, which is a tree, satisfying
\begin{equation}
 \label{mod:ballbounds}
 \operatorname{radius}(B_P)\le\sigma(P)/16,
 \qquad \abs{B_P}<\eta\abs P.
\end{equation}
This ball contains all intersections of $P$ with different copies whose
cone points are no farther from $R$ than $c_P$.

Equip the group $E^\Gamma$ of functions $\Gamma\to E$ with pointwise
addition and the product topology.  It is compact, and $\Gamma$ acts on
it by translation,
$(g\omega)(h)=\omega(g^{-1}h)$.
Define
\begin{equation}
 \label{mod:omega}
 \Omega=
 \{\omega\in E^\Gamma:\omega|_P\in\mathcal C_P
       \text{ for every indexed copy }P\}.
\end{equation}
Each condition in \eqref{mod:omega} is a linear condition on finitely
many coordinates.  Thus $\Omega$ is a closed $\Ftwo$-linear subspace
of $E^\Gamma$, hence a compact abelian group.
Translation by $g$ carries the copy $P_{g',i}$, with its identification,
to $P_{gg',i}$ and carries $\mathcal C_{P_{g',i}}$ to
$\mathcal C_{P_{gg',i}}$.  Therefore $\Omega$ is invariant under the
action of $\Gamma$.  Denote its Haar probability measure by $\mu$.

Since every element of $\Omega$ has order at most two, each continuous
character of $\Omega$ takes values in $\{1,-1\}$.  We may therefore
identify its Pontryagin dual with
$\operatorname{Hom}_{\mathrm{cont}}(\Omega,\Ftwo)$ by assigning to
$v$ the character $\chi_v(\omega)=(-1)^{v(\omega)}$.  Put
\begin{equation}
 \label{mod:group}
 V=\operatorname{Hom}_{\mathrm{cont}}(\Omega,\Ftwo),
 \qquad H=H_A=V\rtimes\Gamma.
\end{equation}
The action on $V$ is $(gv)(\omega)=v(g^{-1}\omega)$.
Thus \Cref{sec:regular} applies to this $H$.  We prove its finite
generation and describe $V$ explicitly in \Cref{sec:dual}.

\begin{definition}
 \label{op:enabled}
 For $\omega\in\Omega$, let $\mathcal G_\omega$ be the subgraph of the
 Cayley graph $X$ with vertex set $\Gamma$ in which the edge $\{h,hs\}$
 is present exactly when
 \[
 s\in\mathcal L(\omega(h))
 \quad\text{and}\quad
 s^{-1}\in\mathcal L(\omega(hs)).
 \]
 A copy $P$ is \emph{complete} for $\omega$ if $\omega|_P=w_P$, and
 \emph{incomplete} otherwise.  Let $U_\omega$ be the set of vertices
 that do not belong to any complete copy.
\end{definition}

The graph $\mathcal G_\omega$ is an undirected simple subgraph of $X$,
and its degree is at most $d$, because each value of $\mathcal L$ is
either a $d$-element set or the empty set.
The identity $(g\omega)(gh)=\omega(h)$ gives
\eqref{mod:equivariant}.  Membership of the edge $\{1,s\}$ depends
only on $\omega(1)$ and $\omega(s)$, so the corresponding subset of
$\Omega$ is open and closed.

More explicitly, define
\[
 f_s(\omega)=
 \begin{cases}
 1,&s\in\mathcal L(\omega(1)),\\
 0,&s\notin\mathcal L(\omega(1)).
 \end{cases}
\]
Thus $f_s(\omega)=1$ means that the label $s$ of the oriented edge from
$1$ to $s$ belongs to the set specified at its initial vertex $1$.
The same edge, oriented from $s$ to $1$, has label $s^{-1}$.  Its label
belongs to the set specified at its initial vertex $s$ exactly when
$s^{-1}\in\mathcal L(\omega(s))$.
Since $(s^{-1}\omega)(1)=\omega(s)$, the function
$\alpha_s(f_{s^{-1}})(\omega)=f_{s^{-1}}(s^{-1}\omega)$ equals one
exactly when this second condition holds.  Hence the indicator function
of $\{\omega:\{1,s\}\in\mathcal G_\omega\}$ is
$F_s=f_s\alpha_s(f_{s^{-1}})$.
The covariance relation gives $F_su_s=f_su_sf_{s^{-1}}$, so
\cref{mod:laplacians} supplies the element
\begin{equation}
 \label{op:operator}
 T=d\one-\sum_{s\in S^{\pm1}}f_su_sf_{s^{-1}}
 \quad\text{in }\Cr(H).
\end{equation}
Since $X$ is simple by \cref{geo:avoidance},
\begin{equation}
 \label{op:fiber}
 \pi_\omega(T)=dI-\Adj(\mathcal G_\omega)
 \qquad(\omega\in\Omega).
\end{equation}
At the end of \Cref{sec:dual} we show that $T$ lies in $\C[H]$.

On a copy $P$ that is complete for $\omega$, the sets $\mathcal L(\omega(h))$ equal
$J_P(h)$.  Every edge of $P$ belongs to $\mathcal G_\omega$, and its
$d$ incident edges at each vertex exhaust the set specified there.
Thus $P$ is an entire connected component of $\mathcal G_\omega$.
Conversely, if a copy $P$ is a connected component of $\mathcal G_\omega$,
then $\mathcal L(\omega(h))$ contains the $d$ labels $J_P(h)$ at each
$h\in P$.  The two sets therefore agree.  Injectivity of
$J\mapsto a_J$ gives $\omega(h)=w_P(h)$, so $P$ is complete.

Distinct complete copies are disjoint.  If two met, both would equal
the same connected component.  Apply \Cref{geo:overlap} with root $R$
a common vertex.  The two cone points both have distance one from $R$,
so their common vertex set would lie in a ball of the first copy with
fewer than $\eta\abs P<\abs P$ vertices.  This is impossible for distinct
indexed copies.  In particular, no edge of $\mathcal G_\omega$ joins
$U_\omega$ to its complement.

\subsection{Uniform distribution of restrictions to each copy}
\label{sec:distribution}

In this subsection, we prove that every function in $\mathcal C_P$
extends to an element of $\Omega$.  The extension property of
\Cref{code:local-extension} handles the intersections of finitely many
copies, and compactness then handles all copies simultaneously.
Surjectivity of the restriction map makes its Haar distribution
uniform.

\begin{lemma}
 \label{mod:marginals}
 For a copy $P=P_{g,i}$, the restriction map
 \[
 \operatorname{res}_P:\Omega\longrightarrow E^P,
 \qquad\omega\longmapsto\omega|_P,
 \]
 has image $\mathcal C_P$ and sends Haar probability to uniform
 measure on $\mathcal C_P$.  Consequently, for every $a\in\mathcal C_P$,
 \begin{equation}
 \label{mod:probability}
 \mu\{\omega:\omega|_P=a\}=2^{-k_i}.
 \end{equation}
 In particular, $P$ is complete with probability $2^{-k_i}$.
\end{lemma}
\begin{proof}
By definition, $\operatorname{res}_P(\Omega)\subseteq\mathcal C_P$.
Fix a copy $P_0$ and a function $a_0\in\mathcal C_{P_0}$.
We first show that, for every finite family $\mathcal F$ of distinct
copies containing $P_0$, there is a function on
$\bigcup_{P\in\mathcal F}P$ whose restriction to each $P$ belongs to
$\mathcal C_P$ and whose restriction to $P_0$ is $a_0$.

We argue by induction on $\abs{\mathcal F}$.  The assertion is
immediate for $\mathcal F=\{P_0\}$.  Otherwise root the coned-off graph
at $R=c_{P_0}$ and choose $Q\in\mathcal F\setminus\{P_0\}$ with
$d_{\Xhat}(c_Q,R)$ maximal.  Every other cone point in the family is
no farther from $R$ than $c_Q$.  Since $c_Q\ne R$, \Cref{geo:overlap}
and \eqref{mod:ballbounds} give
\[
 D_Q:=
 Q\cap\bigcup_{P\in\mathcal F\setminus\{Q\}}P
 \subseteq B_Q,
 \qquad
 \abs{D_Q}<\eta\abs Q.
\]
By induction there is a function on the union of the remaining copies
with the required restrictions.  Its values on $D_Q$ extend to a
function in $\mathcal C_Q$ by \Cref{code:local-extension}, transported
from the source graph to $Q$.  The two functions agree on their common
domain $D_Q$, so together they define the required function on the
whole union.

Extend the function constructed on the union of each finite family of
copies arbitrarily to $\Gamma$.
In the compact space $E^\Gamma$, the conditions
$\omega|_P\in\mathcal C_P$ and $\omega|_{P_0}=a_0$ are closed.
The preceding induction shows that these closed sets have the
finite intersection property.  Compactness therefore gives a function
$\omega\in\Omega$ with $\omega|_{P_0}=a_0$.
Thus the restriction map is surjective.

It is a continuous surjective homomorphism from the compact group
$\Omega$ to the finite additive group $\mathcal C_P$.
The pushforward of Haar probability is translation-invariant, so each
of its $\abs{\mathcal C_P}=2^{k_i}$ elements has the same mass.
This proves \eqref{mod:probability}, including the assertion for
$a=w_P$.
\end{proof}

The proof concerns the restriction to each individual copy.  The trace
calculation combines these marginal probabilities with the fact that
at most one complete copy contains the identity.

\subsection{The dual group and its presentation}
\label{sec:dual}

In this subsection, we prove that the group $H$ in \eqref{mod:group}
is finitely generated.  We then express $V$ as a quotient of a free
$\Ftwo[\Gamma]$-module.  This presentation will be used for the
finite-stage approximations in \Cref{sec:maximal}.

For $h\in\Gamma$ and $1\le j\le m$, define the continuous linear
functional
\[
 v_{h,j}:\Omega\longrightarrow\Ftwo,
 \qquad v_{h,j}(\omega)=\omega(h)_j,
 \qquad v_{*,j}=v_{1,j}.
\]
The action on the dual satisfies $gv_{h,j}=v_{gh,j}$.

\begin{proposition}\label{mod:finite-generation}
The $\Ftwo[\Gamma]$-module $V$ is generated by
$v_{*,1},\ldots,v_{*,m}$.  The group $H=V\rtimes\Gamma$ is generated
by $(v_{*,j},1)$, $1\le j\le m$, and $(0,s)$, $s\in S$.
\end{proposition}
\begin{proof}
Let $v:\Omega\to\Ftwo$ be a continuous linear functional.
Its kernel is open, so there is a finite set $F\subseteq\Gamma$ such
that $v(\omega)=0$ whenever $\omega|_F=0$.
Consequently $v$ factors through the finite-dimensional subspace
$\operatorname{res}_F(\Omega)\subseteq E^F$.
Extend the resulting linear functional to $E^F$ by extending a basis
of that subspace to a basis of $E^F$.
The extension is a linear combination of the scalar coordinate
functionals, so $v$ is a finite linear combination of the $v_{h,j}$.
Since $v_{h,j}=hv_{*,j}$, the module assertion follows.
Finally,
\[
 (0,h)(v_{*,j},1)(0,h)^{-1}=(hv_{*,j},1)
\]
gives the stated generators of the semidirect product.
\end{proof}

Let
\[
 E^{(\Gamma)}
 =\{y:\Gamma\to E:\supp y\text{ is finite}\}
 \cong\bigl(\Ftwo[\Gamma]\bigr)^m.
\]
Its standard scalar basis consists of $e_{h,j}$, where $e_{h,j}(h)$
is the $j$th standard basis vector of $E$ and $e_{h,j}$ vanishes
elsewhere.  For $y\in E^{(\Gamma)}$ and $\omega\in E^\Gamma$, put
\[
 \langle y,\omega\rangle
 =\sum_{h\in\Gamma}y(h)\cdot\omega(h)\in\Ftwo.
\]
Extend the functions in $\mathcal C_P^\perp$ by zero outside $P$, and
form the algebraic sum
\begin{equation}
 \label{mod:relations}
 W=\sum_{g\in\Gamma,\ i\in A}\mathcal C_{P_{g,i}}^\perp
 \subseteq E^{(\Gamma)}.
\end{equation}
Each element of $W$ is a sum of finitely many such functions.
Translation carries the summand on $P_{g,i}$ to the summand on
$P_{hg,i}$, so $W$ is a $\Ftwo[\Gamma]$-submodule.

\begin{lemma}\label{mod:functions}
A function $\omega\in E^\Gamma$ lies in $\Omega$ if and only if
$\langle z,\omega\rangle=0$ for every $z\in W$.
The homomorphism
\[
 E^{(\Gamma)}\longrightarrow V,
 \qquad
 y\longmapsto\bigl(\omega\longmapsto\langle y,\omega\rangle\bigr)
\]
is surjective with kernel $W$.  Hence it induces an isomorphism
\begin{equation}\label{mod:presentation}
 V\cong E^{(\Gamma)}/W
 \cong\bigl(\Ftwo[\Gamma]\bigr)^m/W
\end{equation}
of $\Ftwo[\Gamma]$-modules, taking $[e_{h,j}]$ to $v_{h,j}$.
Under this identification, the associated character is
$\chi_{[y]}(\omega)=(-1)^{\langle y,\omega\rangle}$.
\end{lemma}
\begin{proof}
By the nondegenerate form on $E^P$, the condition
$\langle z,\omega\rangle=0$ for every $z\in\mathcal C_P^\perp$ is
equivalent to
$\omega|_P\in(\mathcal C_P^\perp)^\perp=\mathcal C_P$.
Applying this to all copies proves the first assertion.
Surjectivity of the displayed homomorphism follows from the
finite-coordinate description in \Cref{mod:finite-generation}.
The first assertion gives one inclusion for its kernel, namely that
every element of $W$ vanishes on $\Omega$.

Conversely, let $y\notin W$.  Choose a linear functional
$\ell:E^{(\Gamma)}/W\to\Ftwo$ with $\ell([y])=1$.
Such a functional exists by extending the nonzero vector $[y]$ to a
basis.  Define $\omega:\Gamma\to E$ by
$\omega(h)_j=\ell([e_{h,j}])$.
Then $\langle z,\omega\rangle=\ell([z])$ for every
$z\in E^{(\Gamma)}$, so $\omega$ annihilates $W$ and belongs to
$\Omega$ by the first assertion.  But
$\langle y,\omega\rangle=1$, showing that $y$ does not belong to the
kernel.  This proves \eqref{mod:presentation}.
The formula $\langle gy,g\omega\rangle=\langle y,\omega\rangle$
proves equivariance, and the character formula follows from the
identification in \eqref{mod:group}.
\end{proof}

Finally, the functions $f_s$ in \eqref{op:operator} belong to
$\C[V]$.  For $a\in E$, their constituent indicator functions have the
finite Fourier expansion
\[
 \one_{\{\omega(1)=a\}}(\omega)
 =
 2^{-m}\sum_{t\in E}(-1)^{t\cdot a}
 \chi_{\sum_{j=1}^m t_jv_{*,j}}(\omega).
\]
Indeed, the sum is $2^{-m}\sum_{t\in E}
(-1)^{t\cdot(\omega(1)-a)}$, which is one when $\omega(1)=a$ and
zero otherwise.
Each $f_s$ is a finite sum of these indicator functions, so
$T\in\C[H]\subseteq\Cr(H)$ by \cref{mod:laplacians}.

\section{The uniform spectral gap}
\label{op:section}

In this section, we carry out Step 5 of the strategy: we prove that the
operator $T\in\C[H]$ of \eqref{op:operator} has a spectral gap,
$\spec(T)\subseteq\{0\}\cup[\eps,2d]$.  After the constants in
\Cref{sec:data} are fixed, $\eps>0$ is independent of the subfamily
and of $\omega$ (\cref{op:gap}).  By \eqref{op:fiber},
$\pi_\omega(T)=dI-\Adj(\mathcal G_\omega)$ is the graph Laplacian plus
multiplication by the nonnegative function
$h\mapsto d-\deg_{\mathcal G_\omega}(h)$.  By
\eqref{mod:norm}, it suffices to prove a uniform gap for these operators.
Their kernels consist of the functions that are
constant on each complete copy and vanish elsewhere.  On a complete copy,
the gap comes from the expansion of the relator graph.  On an incomplete
copy $P$, the functions $\omega|_P$ and $w_P$ differ at more than
$\eta\abs P$ vertices, forcing many edges of $P$ to be missing from
$\mathcal G_\omega$.  We estimate this loss first inside one copy
(\Cref{op:local}), and then on the whole complement of the complete
copies, using the overlap estimates of \Cref{sec:geometry}
(\Cref{op:global}); \Cref{op:common} combines the two estimates.

We choose the constants
\begin{equation}
 \label{op:constants}
 \eta=1/(100m),
 \qquad c=\frac\eta2\min\{\kappa,1\},
 \qquad \sigma_i>\max\{16,20/c\},
\end{equation}
where every relator graph has edge expansion
$\abs{\partial B}\ge\kappa\min\{\abs B,n_i-\abs B\}$.

\subsection{Degree defects on individual copies}
\label{op:local}

In this subsection, we show that on an incomplete copy the subgraph
$\mathcal G_\omega$ misses a definite proportion of the edges at every set
of vertices (\cref{op:localdeficit}).

For $B\subseteq P$, let $e_{\omega,P}(B)$ denote the number of edges of
$P$ that lie in $\mathcal G_\omega$ and have both endpoints in $B$, and let
$\partial_PB$ be the edge boundary of $B$ in $P$, the set of edges of $P$
with exactly one endpoint in $B$.  The
\emph{degree defect} $d\abs B-2e_{\omega,P}(B)$ counts incidences: it is the
number of pairs $(h,e)$ with $h\in B$ and $e$ an edge of $P$ at $h$ that is
not an edge of $\mathcal G_\omega$ with both endpoints in $B$.  An edge of
$P$ leaving $B$ contributes once, and an edge of $P$ with both endpoints in
$B$ that is missing from $\mathcal G_\omega$ contributes twice.  Expansion
and the lower bound on the number of vertices where $\omega|_P$ differs
from $w_P$ control these two kinds of incidences.

\begin{lemma}
 \label{op:localdeficit}
 If $P$ is incomplete for $\omega$, then every $B\subseteq P$ satisfies
 \begin{equation}
  \label{op:localineq}
  d\abs B-2e_{\omega,P}(B)\ge c\abs B.
 \end{equation}
\end{lemma}
\begin{proof}
Write $n=\abs P$ and
$M=\{h\in P:\omega(h)\ne w_P(h)\}$.
By the definition \eqref{mod:omega} of $\Omega$, $\omega|_P$
belongs to $\mathcal C_P$; so does $w_P$.
Their difference is a nonzero function in $\mathcal C_P$ on an
incomplete copy.  Equation~\eqref{code:distance} therefore gives
$\abs M>\eta n$.
At a vertex $h\in M$, the set $\mathcal L(\omega(h))$ is either empty or a
$d$-element set different from $J_P(h)$.  This uses the
injectivity of $J\mapsto a_J$.  At least one edge of $P$ at each such vertex is therefore missing from
$\mathcal G_\omega$.

Put $b=\abs B$.  Counting the edges of $P$ at vertices of $B$ that are
not internal edges of $\mathcal G_\omega$ gives both bounds
\begin{align*}
 db-2e_{\omega,P}(B)&\ge\abs{M\cap B},\\
 db-2e_{\omega,P}(B)&\ge\abs{\partial_P B}.
\end{align*}
If $n-b\ge\eta n/2$, then
$\min\{b,n-b\}\ge(\eta/2)b$, and expansion gives
$\abs{\partial_P B}\ge(\eta\kappa/2)b\ge cb$.
If $n-b<\eta n/2$, then
$\abs{M\cap B}\ge\abs M-(n-b)>\eta n/2\ge(\eta/2)b\ge cb$.
The case $B=\varnothing$ is immediate.
\end{proof}

\subsection{A global degree-defect estimate}
\label{op:global}

In this subsection, we extend the degree-defect estimate from single copies
to every finite set $Y\subseteq U_\omega$ (\cref{op:globaldeficit}).
Recall that a copy $P$ is complete for $\omega$ when $\omega|_P=w_P$.
Such a copy is an entire connected component of $\mathcal G_\omega$,
with all its edges present.  Thus $Y\subseteq U_\omega$ means that no
vertex of $Y$ belongs to any such copy.
Directly adding the local degree-defect inequalities of \Cref{op:local} would count overlapping vertices
repeatedly.  We instead count independent cycles of $\mathcal G_\omega$,
assign each contribution to vertices outside an overlap tree, and use
Euler's formula to recover the edge bound.

By \Cref{geo:cycle-sum}, every finitely supported cycle in $X$ is uniquely
a sum of cycles in copies.  We consider the copies in order of
decreasing distance of their cone points from a fixed base point $R$.  By
the overlap lemma, a copy $P$ meets the copies that come later in this
order only inside the tree $B_P$.  This allows us to bound the dimension
contributed by $P$ in terms of the number of vertices of $Y$ that lie in
$P$ but not in $B_P$.  These vertex sets are pairwise disjoint for
different copies, so the bounds can be added.  Additional cycles can
occur when a connected component outside the overlap tree has more than
one edge joining it to that tree;
large girth bounds the number of such attachments.

\begin{lemma}
 \label{op:globaldeficit}
 For every finite $Y\subseteq U_\omega$, let $\mathcal A$ be the graph
 with vertex set $Y$ and all edges of $\mathcal G_\omega$ joining two vertices of $Y$.  Then
 \begin{equation}
  \label{op:globalineq}
  d\abs Y-2\abs{\Edges(\mathcal A)}\ge(c/2)\abs Y.
 \end{equation}
\end{lemma}
\begin{proof}
For a finite graph $B$, write
$b_1(B)=\dim_{\Ftwo}Z_1(B;\Ftwo)$, including isolated vertices in $B$.
Writing $\pi_0(\mathcal A)$ for the set of connected components of
$\mathcal A$, Euler's identity gives
\[
 \abs{\Edges(\mathcal A)}
   =\abs Y-\abs{\pi_0(\mathcal A)}+b_1(\mathcal A).
\]
Set $\beta=d/2-1-c/4>0$.  It is enough to prove
\begin{equation}
 \label{op:ranktarget}
 b_1(\mathcal A)\le\beta\abs Y.
\end{equation}

\emph{Step 1: the cycle-space decomposition.}
Regard $L=Z_1(\mathcal A;\Ftwo)$ as a subspace of
$Z_1(X;\Ftwo)$.  By \cref{geo:cycle-sum},
\[
 Z_1(X;\Ftwo)=\bigoplus_P Z_1(P;\Ftwo),
\]
an internal algebraic direct sum; write $\operatorname{pr}_P$ for the
projection onto the summand $Z_1(P;\Ftwo)$, so that $\operatorname{pr}_Pz$
is the summand of $z$ supported on $P$.  This summand may contain edges
outside $\mathcal A$; these edges cancel in the sum over the copies.
Only finitely many copies $P$ have $\operatorname{pr}_P(L)\ne0$:
take a basis of the finite-dimensional space $L$; each basis vector has
nonzero projections onto only finitely many summands.
Call the corresponding copies $P_1,\ldots,P_N$, ordered by nonincreasing
distance of their cone points from a fixed group vertex $R$.
Ties may be resolved arbitrarily.  For $1\le j\le N+1$, put
\[
 L_j=\{z\in L:\operatorname{pr}_{P_i}z=0\text{ for }i<j\}.
\]
Since $L_1=L$, $L_{N+1}=0$, and
$L_{j+1}=\ker(\operatorname{pr}_{P_j}|_{L_j})$, rank-nullity gives
\begin{equation}
 \label{op:rankdecomposition}
 \dim L=\sum_{j=1}^N\dim\operatorname{pr}_{P_j}(L_j).
\end{equation}
If $L=0$, the desired estimate follows directly; the following argument
also allows its empty sum interpretation.

For $z\in L_j$, Step~2 shows that every edge of
$\operatorname{pr}_{P_j}z$ outside $\mathcal A$ lies in the tree $B_{P_j}$.
The restriction to $L_j$ ensures that such an edge can cancel only
against a summand from a later copy.

\emph{Step 2: localization of the projected cycles.}
Fix $j$, put $P=P_j$, and let $B_P$ be the ball given by \Cref{geo:overlap}
with base point $R$; it spans a tree.  It contains every vertex shared with
one of $P_{j+1},\ldots,P_N$.  Set
\[
 D_P=B_P\cup(\mathcal A\cap P),
 \qquad Y_P=(Y\cap P)\setminus B_P,
\]
where $\mathcal A\cap P$ has vertex set $Y\cap P$ and the edges common
to both graphs.  The vertices of $Y_P$ lie in no copy $P_{j+1},\ldots,P_N$.
We shall bound the dimension of $\operatorname{pr}_P(L_j)$ in terms of
$\abs{Y_P}$.  First, we claim
\begin{equation}
 \label{op:localizecycle}
 \operatorname{pr}_P(L_j)\subseteq Z_1(D_P;\Ftwo).
\end{equation}
Indeed, let an edge $e$ occur in $\operatorname{pr}_Pz$ for $z\in L_j$.
If $e$ is absent from $\mathcal A$, its coefficient must cancel against
that of another remaining copy, because the total chain $z$ is supported
on $\mathcal A$.  The two endpoints of $e$ are consequently in $B_P$.
Since $e$ is an edge of $P$ and $B_P$ is induced in $P$, the whole edge
belongs to $B_P$.  Every other edge of the projected chain is in
$\mathcal A\cap P$.  The projected chain is already a cycle in $P$,
which proves \eqref{op:localizecycle}.  Hence
\begin{equation}
 \label{op:rankDP}
 \dim\operatorname{pr}_P(L_j)\le b_1(D_P).
\end{equation}

\emph{Step 3: a bound for $b_1(D_P)$ in terms of $\abs{Y_P}$.}
If $P$ is complete, then $Y\cap P=\varnothing$, so $D_P=B_P$ is a tree.
Suppose that $P$ is incomplete, and write $\sigma=\sigma(P)$ and
$r=\operatorname{radius}(B_P)$.  Let $C$ range over the connected
components obtained by deleting the vertices of $B_P$ from $D_P$.
Their vertex sets partition $Y_P$.
Let $l_C$ be the number of edges of $D_P$ joining a vertex of $C$ to a
vertex of $B_P$; we call these edges the \emph{attaching edges} of $C$.
The tree $B_P$ has no cycles, but several attaching edges of the same
component $C$ create new cycles.  Contracting $B_P$ to one vertex makes
this count explicit: when $l_C\ge1$, the first attaching edge connects $C$
to the contracted tree, and each further attaching edge increases the
dimension of the cycle space by one.  A component with $l_C=0$
contributes only $b_1(C)$.
Thus Euler's formula gives
\begin{equation}
 \label{op:attachmentsformula}
 b_1(D_P)=\sum_C\bigl(b_1(C)+\max\{0,l_C-1\}\bigr).
\end{equation}
All edges of $C$ are edges of $P$ lying in $\mathcal G_\omega$.  Moreover,
all such edges between vertices of $C$ occur in $C$, because $\mathcal A$
contains all edges of $\mathcal G_\omega$ joining vertices of $Y$.  If $C$ contains a cycle, then
$\abs C\ge\sigma$.  Applying \cref{op:localdeficit} to its vertex set,
and using connectedness, yields
\begin{equation}
 \label{op:componentrank}
 \begin{aligned}
 b_1(C)&=\abs{\Edges(C)}-\abs C+1\\
       &\le(d/2-1-c/2+1/\sigma)\abs C.
 \end{aligned}
\end{equation}
For an acyclic $C$, the same inequality holds since $b_1(C)=0$ and the
coefficient on the right is positive: $d\ge3$ and $c\le1/200$.

It remains to bound the term $\max\{0,l_C-1\}$.  When $C$ has at least
two attaching edges, large girth forces their endpoints in $C$ to be far
apart, so $C$ must contain many vertices.  Take two distinct attaching
edges of $C$,
with endpoints $v,v'$ in $C$ and $b,b'$ in $B_P$.
A shortest path in $C$ from $v$ to $v'$, those two edges, and the tree path
from $b'$ to $b$ form a simple cycle in $P$.  This remains true if one
pair of endpoints agrees, using a path of length zero; both pairs cannot
agree because $P$ is simple.  The tree path has length at most $2r$.
Thus
\begin{equation}
 \label{op:attachmentseparation}
 \dist_C(v,v')\ge\sigma-2r-2>\sigma/2,
\end{equation}
where $\dist_C$ denotes the path metric of $C$, and where
$r\le\sigma/16$ and $\sigma>16$ were used.
In particular, the endpoints in $C$ of different attaching edges are distinct.
If $l_C\ge2$, the intrinsic balls in $C$ of radius
$\lfloor\sigma/4\rfloor$ centered at these endpoints are pairwise
disjoint.  Each contains at least $\lfloor\sigma/4\rfloor+1>\sigma/4$
vertices: use an initial segment of a shortest path to the endpoint of
another attaching edge.  It follows that
\begin{equation}
 \label{op:attachmentbound}
 \max\{0,l_C-1\}\le(4/\sigma)\abs C.
\end{equation}
The same inequality is immediate for $l_C=0$ or $1$.

Combining \eqref{op:attachmentsformula}--\eqref{op:attachmentbound},
and using $5/\sigma<c/4$, gives
\begin{align}
 b_1(D_P)
 &\le(d/2-1-c/2+5/\sigma)\abs{Y_P}\notag\\
 &\le\beta\abs{Y_P}.
 \label{op:chargedrank}
\end{align}
As noted earlier, this last bound also holds for a complete copy: both
sides are then zero.

\emph{Step 4: summation over the copies $P_1,\ldots,P_N$.}
If $j<k$, the cone point of $P_k$ is no farther from $R$ than that of $P_j$.
The overlap lemma gives $P_j\cap P_k\subseteq B_{P_j}$, so
the set $Y_{P_j}$ misses all of $P_k$.
Thus the vertex sets used in \eqref{op:chargedrank} are pairwise disjoint
subsets of $Y$.  The overlap lemma uses a non-strict distance comparison,
so this conclusion also covers ties in the ordering.
Equations \eqref{op:rankdecomposition}, \eqref{op:rankDP}, and
\eqref{op:chargedrank} now give
\[
 \dim L\le\beta\sum_{j=1}^N\abs{Y_{P_j}}
          \le\beta\abs Y,
\]
which is \eqref{op:ranktarget}.  Returning to Euler's identity,
\begin{align*}
 \abs{\Edges(\mathcal A)}
 &\le(1+\beta)\abs Y-\abs{\pi_0(\mathcal A)}\\
 &\le(d/2-c/4)\abs Y.
\end{align*}
This is equivalent to \eqref{op:globalineq}.
\end{proof}

\subsection{The uniform spectral gap}
\label{op:common}

In this subsection, we combine the estimates of the two preceding
subsections with the expansion of the complete copies to prove the
uniform spectral gap (\cref{op:gap}).  The global degree-defect estimate
gives a lower bound on the energy of every finitely
supported function on $U_\omega$.  The passage from sets to functions is
obtained by integrating the degree defects of the level sets of $\abs\xi^2$.
The missing degree in $\mathcal G_\omega$ contributes a nonnegative vertex
potential,
which is kept in this calculation.

\begin{proposition}
 \label{op:gap}
 Put
 \[
  \eps=\min\{\rho,c^2/(8d)\}>0.
 \]
 For every $\omega\in\Omega$,
 \begin{equation}
  \label{op:fiberspectrum}
  \spec(\pi_\omega(T))\subseteq\{0\}\cup[\eps,2d].
 \end{equation}
 Its kernel is exactly the Hilbert-space direct sum of the one-dimensional
 spaces of constants on the complete copies.
 In particular,
 \begin{equation}
  \label{op:algebraspectrum}
  \spec_{\Cr(H)}(T)\subseteq\{0\}\cup[\eps,2d].
 \end{equation}
\end{proposition}
\begin{proof}
On $U_\omega$, let $\deg_\omega(x)$ be the degree of $x$ in
$\mathcal G_\omega$ and put $q(x)=d-\deg_\omega(x)\ge0$.
No edge of $\mathcal G_\omega$ joins $U_\omega$ to its complement.
For a finitely supported complex vector $\xi$ on $U_\omega$, the energy is
\begin{equation}
 \label{op:energy}
 \begin{aligned}
 \mathcal E(\xi)
 &=\langle\pi_\omega(T)\xi,\xi\rangle\\
 &=\sum_{\{x,y\}}\abs{\xi(x)-\xi(y)}^2
       +\sum_xq(x)\abs{\xi(x)}^2.
 \end{aligned}
\end{equation}
Here and below the edge sums are over the edges of
$\mathcal G_\omega|_{U_\omega}$, and the vertex sums are over $U_\omega$.
For a finite $F\subseteq U_\omega$, write $\partial_\omega F$ for its
edge boundary in this graph.  Its degree defect is
\[
 d\abs F-2\abs{\Edges(\mathcal G_\omega|_F)}
  =\abs{\partial_\omega F}+\sum_{x\in F}q(x).
\]
Set $s(x)=\abs{\xi(x)}$ and $F_t=\{x:s(x)^2>t\}$ for $t\ge0$.  An edge
$\{x,y\}$ lies in $\partial_\omega F_t$ exactly when $t$ lies between
$s(x)^2$ and $s(y)^2$, and $x\in F_t$ exactly when $t<s(x)^2$.  Hence
\[
 \int_0^\infty\abs{\partial_\omega F_t}\,dt=\sum_{\{x,y\}}\abs{s(x)^2-s(y)^2},
 \qquad
 \int_0^\infty\sum_{x\in F_t}q(x)\,dt=\sum_xq(x)s(x)^2,
\]
and $\int_0^\infty\abs{F_t}\,dt=\norm\xi^2$.  Applying
\cref{op:globaldeficit} to each $F_t$ and integrating over $t\ge0$ gives
\begin{equation}
 \label{op:coarea}
 (c/2)\norm\xi^2
   \le\sum_{\{x,y\}}\abs{s(x)^2-s(y)^2}
        +\sum_xq(x)s(x)^2.
\end{equation}
Cauchy--Schwarz applies to the edge terms and vertex terms together.
Define its two quadratic factors by
\begin{align*}
 A_s&=\sum_{\{x,y\}}\abs{s(x)-s(y)}^2+\sum_xq(x)s(x)^2,\\
 B_s&=\sum_{\{x,y\}}(s(x)+s(y))^2+\sum_xq(x)s(x)^2.
\end{align*}
Then \eqref{op:coarea} gives
\[
 (c/2)\norm\xi^2\le A_s^{1/2}B_s^{1/2}.
\]
The inequality
$\abs{\abs{\xi(x)}-\abs{\xi(y)}}\le\abs{\xi(x)-\xi(y)}$
shows that $A_s\le\mathcal E(\xi)$.  For the other factor,
\begin{align*}
 B_s&\le\sum_x(2\deg_\omega(x)+q(x))s(x)^2\\
    &\le2d\norm\xi^2.
\end{align*}
Consequently,
\[
 (c/2)\norm\xi^2\le\mathcal E(\xi)^{1/2}\sqrt{2d}\,\norm\xi.
\]
If $\xi\ne0$, division and squaring yield
$\mathcal E(\xi)\ge c^2\norm\xi^2/(8d)$; the zero vector satisfies the
same estimate.  By density and boundedness this holds throughout
$\ell^2(U_\omega)$.

On a complete copy $P$, the operator is the graph Laplacian
$dI-\Adj(P)$.  Its kernel consists of constants and its restriction to
their orthogonal complement is at least $\rho I$, by
\cref{data:family}.  The complete copies are disjoint isolated components,
so the decomposition into their component spaces and $\ell^2(U_\omega)$
proves the asserted kernel description and the uniform lower bound
$\eps$ off the kernel.  These statements remain true with infinitely
many components, because the lower bounds are uniform.
Together with $0\le\pi_\omega(T)\le2dI$ from \cref{mod:laplacians}, this
proves \eqref{op:fiberspectrum}.

The direct sum of the $\pi_\omega$ is faithful by
\cref{mod:representations}.  To pass the spectral estimate to this direct
sum, put $\Sigma=\{0\}\cup[\eps,2d]$.  For $z\notin\Sigma$, every
summand has a resolvent satisfying
\[
 \norm{(\pi_\omega(T)-zI)^{-1}}
       \le\dist(z,\Sigma)^{-1}.
\]
Their direct sum is therefore a bounded inverse for the corresponding
operator on the full direct sum.  Its spectrum is contained in $\Sigma$.
A faithful representation of a unital $C^*$-algebra preserves spectra,
which proves \eqref{op:algebraspectrum}.
\end{proof}

\section{The projection and its trace}
\label{trace:section}

In this section, we define the projection $b_A=\chi(T)$ by continuous
functional calculus and compute its canonical trace:
$\tau(b_A)=\sum_{k\in A}2^{-k}$ (\cref{trace:exact}).  Since the
constants of the construction do not depend on $A$, every such binary
subseries is realized, and these numbers fill an interval
(\cref{trace:subsets,trace:small-interval}).  This is Step 6 of the
strategy.

Choose a continuous real function $\chi$ on $[0,2d]$ satisfying
$\chi(0)=1$ and $\chi=0$ on $[\eps,2d]$, and set
\begin{equation}
 \label{trace:projection}
 b_A=\chi(T)\in\Cr(H_A).
\end{equation}
By \cref{op:gap}, $\chi^2=\chi$ on the spectrum of $T$, so $b_A$ is a
projection.  Functional calculus commutes with every $\pi_\omega$;
hence $\pi_\omega(b_A)$ is the orthogonal projection onto
$\ker\pi_\omega(T)$.  The uniform spectral gap permits this continuous
functional calculus directly in $\Cr(H_A)$.

\begin{proposition}
 \label{trace:exact}
 The projection just constructed satisfies
 \begin{equation}
  \label{trace:series}
  \tau_{H_A}(b_A)=\sum_{k\in A}2^{-k}.
 \end{equation}
\end{proposition}
\begin{proof}
On a copy $P$ of type $i$ that is complete for $\omega$, that is,
$\omega|_P=w_P$, the normalized constant vector is
$n_i^{-1/2}\one_P$.  Its rank-one projection has diagonal $1/n_i$ at
each vertex of that copy.  On $U_\omega$ the kernel projection is zero.
Write $D(\omega)=\langle\pi_\omega(b_A)\delta_1,\delta_1\rangle$.
Then
\[
 D(\omega)=\begin{cases}
   1/n_i,&1\text{ lies in a complete copy of type }i,\\
   0,&1\in U_\omega.
  \end{cases}
\]
The complete copy in the first line is unique.

Exactly $n_i$ indexed copies of type $i$ contain the identity.  Indeed,
their translating elements are
$g=\varphi_i(v)^{-1}$ as $v$ ranges over the $n_i$ distinct embedded
vertices of $\Theta_i$.  Distinct indexed copies are distinct in the
geometric construction.  Each such copy $P$ is complete exactly on the
set $\{\omega:\omega|_P=w_P\}$, which is open and closed in $\Omega$ and has
measure $2^{-k_i}$ by \cref{mod:marginals}.  For different copies containing
the identity, including copies of different types, these sets are
pairwise disjoint, because distinct complete copies are disjoint.
Applying \eqref{mod:trace-integral} and nonnegative countable summation
therefore gives
\[
 \tau_{H_A}(b_A)
   =\sum_{i\in A}n_i\frac1{n_i}2^{-k_i}
   =\sum_{k\in A}2^{-k},
\]
since $k_i=i$.
\end{proof}

\begin{corollary}
 \label{trace:subsets}
 With the finite alphabet, the vector space $E=\Ftwo^m$, the graphs, and
 the subspaces fixed as in
 \cref{data:family,code:exist}, every subset
 $A\subseteq\{K,K+1,\ldots\}$ gives a finitely generated group $H_A$
 and a projection $b_A\in\Cr(H_A)$ satisfying
 \[
    \tau_{H_A}(b_A)=\sum_{k\in A}2^{-k}.
 \]
\end{corollary}
\begin{proof}
Take exactly the graphs indexed by $A$ and use their already chosen
labels and subspaces.  They satisfy the same hypotheses of the geometric
lemmas, which therefore apply to the new presentation with the same
constants.  The inequalities \eqref{code:distance} also hold for every
chosen graph.
The proof also covers finite and empty subfamilies.  If $A=\varnothing$,
then $\Gamma_A=F(S)$ and $V_A=(\Ftwo[F(S)])^m$.  There are no complete
copies, so \Cref{op:gap} gives $b_A=0$.
\end{proof}

\begin{corollary}
\label{trace:small-interval}
Set $\delta=2^{1-K}$.  For every $q\in[0,\delta]$ there is a set
$A\subseteq\{K,K+1,\ldots\}$ for which the preceding construction gives
\[
 b_A\in\Cr(H_A),\qquad \tau_{H_A}(b_A)=q.
\]
\end{corollary}
\begin{proof}
Every number in $[0,2^{1-K}]$ can be written as
$q=\sum_{k\ge K}\epsilon_k2^{-k}$ with $\epsilon_k\in\{0,1\}$.
For $q<2^{1-K}$, take a binary expansion; at the upper endpoint,
take $\epsilon_k=1$ for every $k\ge K$.  Apply \Cref{trace:subsets}
with $A=\{k\ge K:\epsilon_k=1\}$.
\end{proof}

\section{Realizing the prescribed trace values}
\label{sec:prescribed}

In this section, we complete the proofs of \Cref{thm:main} and
\Cref{cor:nonsurjective}.  For each sufficiently small
nonnegative number $q$, \Cref{trace:small-interval} gives a semidirect
product $H_A=V_A\rtimes\Gamma_A$ and a projection in $\Cr(H_A)$ of
trace $q$, with $A$ depending on $q$.  Taking the direct product with a
finite cyclic group then gives every prescribed trace in $(0,1)$.
This completes Step 7 of the strategy.

\begin{lemma}\label{prescribed:rational}
For integers $N\ge1$ and $0\le a\le N$, the reduced algebra of the finite
cyclic group $C_N=\Z/N\Z$ of order $N$ contains a projection of canonical
trace $a/N$.
\end{lemma}
\begin{proof}
If $u$ is the generator and $\zeta=e^{2\pi i/N}$, the elements
\[
 e_j=\frac1N\sum_{t=0}^{N-1}\zeta^{-jt}u^t,
 \qquad j=0,\ldots,N-1,
\]
are mutually orthogonal self-adjoint projections.  These assertions follow
by summing the geometric series $\sum_{t=0}^{N-1}\zeta^{\ell t}$.
Their traces all equal $1/N$.  The sum of any $a$ of them has the desired
trace.  Empty sums and $N=1$ are included.
\end{proof}

Tensoring with $C^*(C_N)$ has a simple interpretation.  Since
$C^*(C_N)\cong\C^N$, the algebra $C^*(C_N)\otimes\Cr(H)$ is a direct
sum of $N$ copies of $\Cr(H)$, and its canonical trace is the average of
their canonical traces.  If we put the identity projection in $a$
summands and a projection of trace $q$ in the remaining summands, the
resulting trace is $(a+(N-a)q)/N$.  The following formula expresses this
construction without choosing the individual summands.

\begin{lemma}\label{prescribed:tensor}
Let $H$ be a discrete group and let $b\in\Cr(H)$ be a projection of
canonical trace $q$.  If $e\in C^*(C_N)$ is a projection of trace $a/N$,
then
\[
 B=e\otimes1+(1-e)\otimes b
 \in\Cr(C_N\times H)
\]
is a projection of canonical trace
\[
 \tau_{C_N\times H}(B)=\frac aN+\left(1-\frac aN\right)q.
\]
\end{lemma}
\begin{proof}
The identification
\[
 \Cr(C_N\times H)=C^*(C_N)\otimes\Cr(H)
\]
follows from the tensor-product decomposition of the left regular
representation.  Under this identification, the canonical trace is the
product of the two canonical traces.  The summands in $B$ are
self-adjoint projections and their product is zero.  Thus $B$ is a
projection, and applying the product trace gives the displayed formula.
\end{proof}

\begin{proof}[Proof of \Cref{thm:main}]
Fix the cutoff $K\ge2$ from \Cref{data:family}, and put
\[
 N=2^K,\qquad \delta=2^{1-K}=\frac2N.
\]
For the given $r\in(0,1)$, define
\begin{equation}\label{prescribed:parameters}
 s=\min\{r,1-r\},\qquad
 a=\lfloor Ns\rfloor,\qquad
 q=\frac{Ns-a}{N-a}.
\end{equation}
Since $0<s\le1/2$, we have $a\le N/2$, and hence
\[
 0\le q<\frac1{N-a}\le\frac2N=\delta.
\]
By \Cref{trace:small-interval}, choose
$A=A(r)\subseteq\{K,K+1,\ldots\}$ such that
\begin{equation}\label{prescribed:digits}
 q=\sum_{k\in A}2^{-k},
\end{equation}
and let $b_A\in\Cr(H_A)$ be the corresponding projection of trace $q$.
When $q=0$, take $A=\varnothing$ and $b_A=0$.
Choose $e\in C^*(C_N)$ with trace $a/N$ by
\Cref{prescribed:rational}, and set
\begin{equation}\label{prescribed:group}
 G_r=C_N\times H_A
     \cong(C_N\times V_A)\rtimes\Gamma_A,
\end{equation}
where $\Gamma_A$ acts trivially on $C_N$ and by its given action on $V_A$.
This is a finitely generated group.  Its normal subgroup
$C_N\times V_A$ is locally finite and abelian.

By \Cref{prescribed:tensor}, the projection
\[
 B=e\otimes1+(1-e)\otimes b_A\in\Cr(G_r)
\]
has trace
\[
 \tau_{G_r}(B)
   =\frac aN+\frac{N-a}{N}\frac{Ns-a}{N-a}=s.
\]
Set $e_r=B$ when $r\le1/2$, and $e_r=1-B$ when $r>1/2$.
In both cases $\tau_{G_r}(e_r)=r$.
\end{proof}

\begin{remark}
All constants, including $m$, $K$, and $N=2^K$, are chosen before $r$.
Only the subset $A(r)$, the projection $e\in C^*(C_N)$,
and the choice between $B$ and $1-B$ depend on $r$.
The number of generators is uniformly bounded: $H_A$ has at most
$|S|+m$ generators, and $G_r$ has at most $|S|+m+1$.
The relations encode the binary digits of the number $q$ in
\eqref{prescribed:parameters}.
\end{remark}

For a discrete group $G$, we use the unnormalized trace
\[
 (\Tr_\ell\otimes\tau_G)((a_{ij}))=\sum_{i=1}^\ell\tau_G(a_{ii})
\]
on $M_\ell(\Cr(G))$.  It is invariant under equivalence of projections
and compatible with stabilization, so it induces the trace homomorphism
\begin{align*}
 \tau_{G,*}:K_0(\Cr(G))&\longrightarrow\R,\\
 [P]-[Q]&\longmapsto
 (\Tr\otimes\tau_G)(P)-(\Tr\otimes\tau_G)(Q).
\end{align*}

\begin{corollary}\label{prescribed:allreal}
For every $t\in\R$ there exist a finitely generated group $G$ and
$x\in K_0(\Cr(G))$ with $\tau_{G,*}(x)=t$.
\end{corollary}
\begin{proof}
Write $t=j+r$, with $j\in\Z$ and $0\le r<1$.  If $r>0$, use
\Cref{thm:main} and set $x=j[1]+[e_r]$.  If $r=0$, the class $j[1]$ over
the trivial group suffices.
\end{proof}

\subsection{The obstruction to surjectivity of the reduced Baum--Connes map}

In this subsection, we deduce from L\"uck's trace theorem that, for
irrational $r$, the class of $e_r$ is not in the image of the reduced
Baum--Connes map.

Let $\underline EG$ be the classifying space for proper actions, and write
\[
 \mu_0^G:K_0^G(\underline EG)\longrightarrow K_0(\Cr(G))
\]
for the reduced Baum--Connes map
\cite{BC,BCH}, with its usual direct-limit domain over $G$-compact
supports; see also \cite{Valette}.  The following is the form
of L\"uck's trace theorem needed here \cite[Theorem~0.3]{Luck}:
\begin{equation}\label{prescribed:luck}
 \tau_{G,*}\bigl(\im\mu_0^G\bigr)
 \subseteq\Lambda^G:=\Z[1/|F|:F\le G\text{ finite}]\subseteq\Q.
\end{equation}
More precisely, the cited theorem identifies the trace range after
extension of scalars to $\Lambda^G$.  Applying it to $1\otimes y$ for a
class $y$ in the image of the integral map gives the inclusion displayed
above.

\begin{lemma}\label{prescribed:countable}
For a countable discrete group $G$, the group $K_0(\Cr(G))$ and its
canonical trace range are countable.
\end{lemma}
\begin{proof}
The algebra $\Cr(G)$ is separable, so the space of projections in each
matrix algebra over it is a separable metric space.  Projections whose
norm distance is less than one are Murray--von Neumann equivalent.
Thus every equivalence class is open in that projection space.  A
second-countable space has only countably many disjoint nonempty open
sets, and hence only countably many such classes.  Taking the countable
union over matrix sizes, and then forming formal differences, proves
that $K_0(\Cr(G))$ is countable.  Its image under the trace is countable
as well.
\end{proof}

\begin{proof}[Proof of \Cref{cor:nonsurjective}]
For irrational $r$, the trace of $[e_r]$ is $r$, while
\eqref{prescribed:luck} gives a rational trace for every class in the
image of the reduced Baum--Connes map.
This proves the integral obstruction.  The trace extends to the
rationalized target by
\[
 K_0(\Cr(G_r))\otimes_\Z\Q\longrightarrow\R,
 \qquad x\otimes q\longmapsto q\tau_{G_r,*}(x).
\]
A class in the image of the rationalized map is a finite rational linear
combination of classes in the image of the integral map, so it still has
rational trace.  The class
$[e_r]\otimes1$ has irrational trace $r$ and remains outside the image.
\end{proof}

\subsection{Remarks}
\label{prescribed:remarks}

For a discrete group $G$, we define the invariant
\[
 T(G)=\{\tau_G(e):e=e^*=e^2\in\Cr(G)\}.
\]
Thus $T(G)\subseteq[0,1]$ is the set of values of the canonical
trace on projections in the reduced group $C^*$-algebra itself.  The
range $\tau_{G,*}(K_0(\Cr(G)))$ also allows differences of
classes of projections in matrix algebras.  The set $T(G)$ is
invariant under group isomorphisms.  If $\Cr(G)$ has a unique tracial
state, then any $C^*$-algebra isomorphism
$\Cr(G)\cong\Cr(\Lambda)$ carries the canonical trace to the canonical
trace, and hence $T(G)=T(\Lambda)$.  Thus $T(G)$ is also an invariant of
the reduced group $C^*$-algebra among groups with the unique trace
property; see Breuillard--Kalantar--Kennedy--Ozawa~\cite{BKKO}.

The finite central factor of $G_r$ gives nontrivial central projections
in $\Cr(G_r)$, so this algebra is not simple and does not have a unique
tracial state.  To obtain simple algebras with unique trace, define
\begin{equation}\label{prescribed:simple-group}
 \widetilde G_r=G_r*\Z.
\end{equation}
The subgroup inclusion induces a trace-preserving embedding
$\Cr(G_r)\hookrightarrow\Cr(\widetilde G_r)$.  Indeed, the restriction
of the left regular representation of $\widetilde G_r$ to $G_r$ is an
orthogonal sum of copies of the left regular representation of $G_r$.
Thus $r\in T(\widetilde G_r)$.  Since both free factors are nontrivial
and $\Z$ has more than two elements, the theorem of Paschke--Salinas
\cite[Theorem~1.1]{PS} implies that $\Cr(\widetilde G_r)$ is simple and
has a unique tracial state.

It follows that the family $\{\Cr(\widetilde G_r):r\in(0,1)\}$ contains
continuum many pairwise nonisomorphic simple reduced group $C^*$-algebras
with unique tracial states.  By \Cref{prescribed:countable}, each set
$T(\widetilde G_r)$ is countable.  Since algebra isomorphisms preserve
the unique trace, each algebra isomorphism type can account for only
countably many prescribed values of $r$.  The continuum of values
therefore gives continuum many algebra isomorphism types.  The sets
$T(\widetilde G_r)$ themselves also take continuum many distinct values,
since each is countable and their union contains $(0,1)$.

\section{The maximal strong Novikov conjecture}
\label{sec:maximal}

In this section, we prove the maximal strong Novikov conjecture for the
groups $G_r$ constructed in \Cref{thm:main} and for the groups
$\widetilde G_r=G_r*\Z$ in \eqref{prescribed:simple-group}.  We first
express these groups as direct limits of groups of finite asymptotic dimension, with surjective
connecting homomorphisms.  We then prove that injectivity of maximal
assembly passes to such limits, following the colimit methods of
Bartels--Echterhoff--L\"uck~\cite{BEL}.

For a discrete group $G$, let $C^*_{\max}(G)$ be the completion of
$\C[G]$ in the supremum of the norms of its unitary representations.
The maximal assembly map is
\[
 \mu_*^{G,\max}:K_*^G(\underline EG)
       \longrightarrow K_*(C^*_{\max}(G)).
\]
As in the introduction, equivariant $K$-homology is taken with
$G$-compact supports.  The natural map $EG\to\underline EG$ induces
\[
 \iota_G:K_*(BG)=K_*^G(EG)\longrightarrow K_*^G(\underline EG).
\]
Here $K_*(BG)$ is ordinary $K$-homology with compact supports.  The
maximal strong Novikov conjecture asserts that
$\nu_*^G:=\mu_*^{G,\max}\circ\iota_G$ is rationally injective.

\begin{theorem}\label{max:main}
For every $r\in(0,1)$ and either $G=G_r$ or $G=\widetilde G_r$, the
maximal assembly map
\[
 \mu_*^{G,\max}:K_*^G(\underline EG)
       \longrightarrow K_*(C^*_{\max}(G))
\]
is injective in both degrees.  Consequently,
\[
 \nu_*^G\otimes\Q:
 K_*(BG)\otimes\Q
       \longrightarrow K_*(C^*_{\max}(G))\otimes\Q
\]
is injective, so both groups satisfy the maximal strong Novikov
conjecture.
\end{theorem}

The Novikov conjecture plays a central role in the classification of
higher dimensional manifolds~\cite{Weinberger} and has been studied
through $KK$-theory~\cite{HK,Kas}, cyclic cohomology~\cite{Connes86,CGM,CM}, and
quantitative $K$-theory and related geometric
methods~\cite{Yu98,Yu00,STY,GHW,GWY,GWXY}.
For the proof below, we use the consequence of
\cite[Theorem~6.1]{STY} that groups of finite asymptotic dimension have
integrally injective reduced assembly maps for proper actions.

In this subsection, we prove that each $G_r$ and each $\widetilde G_r$
is a direct limit of finitely generated groups of finite asymptotic
dimension, with surjective connecting homomorphisms.

Recall that a metric space has asymptotic dimension at most $q$ if, for
every $R>0$, it has a uniformly bounded cover that is the union of
$q+1$ families, with distinct members of each family at distance greater
than $R$.  For a finitely generated group, we use a word metric; the
asymptotic dimension does not depend on the finite generating set.

\begin{proposition}\label{max:stages}
For every $r\in(0,1)$ there are sequences of finitely generated groups
and surjective homomorphisms
\[
 G_{r,1}\longrightarrow G_{r,2}\longrightarrow\cdots,
 \qquad
 \widetilde G_{r,1}\longrightarrow\widetilde G_{r,2}
       \longrightarrow\cdots
\]
such that
\[
 G_r\cong\varinjlim_n G_{r,n},
 \qquad
 \widetilde G_r\cong\varinjlim_n\widetilde G_{r,n},
\]
and every group in either sequence has finite asymptotic dimension.
\end{proposition}

\begin{proof}
Fix a nonempty set $A\subseteq\{K,K+1,\ldots\}$ and choose an increasing
exhaustion by nonempty finite sets $A^{(n)}$.  Let
\[
 \Gamma_{A,n}=\langle S\mid\text{labels of closed paths in }
                    \Theta_i,\ i\in A^{(n)}\rangle.
\]
Adding relators gives surjections between consecutive groups, and
their direct limit is the graphical group $\Gamma_A$ used to construct
$H_A$.

Put $X_{A,n}=\Cay(\Gamma_{A,n},S)$.
By \Cref{data:family}, each finite subfamily satisfies the hypotheses of
\Cref{geo:hyperbolicity}.  Its coned-off Cayley graph
$\widehat X_{A,n}$ is therefore hyperbolic.  Put
$D_n=\max_{i\in A^{(n)}}\diam\Theta_i$.
Every two-edge passage through a cone vertex can be replaced by a path
in its copy of length at most $D_n$.  Thus, for group vertices $x,y$,
\[
 d_{\widehat X_{A,n}}(x,y)\le d_{X_{A,n}}(x,y)
 \le\max\{1,D_n\}\,d_{\widehat X_{A,n}}(x,y).
\]
Every cone vertex is at distance one from the Cayley graph, so the two
graphs are quasi-isometric.  It follows that $\Gamma_{A,n}$ is a
hyperbolic group and has finite asymptotic dimension \cite{Roe}.

We also form the module at each stage.  Use the fixed base vertices in
the graphs, and the same path words at every stage, to define the based
maps $\varphi_{i,n}:\Theta_i\to\Gamma_{A,n}$ for $i\in A^{(n)}$.
For $h\in\Gamma_{A,n}$ and $1\le j\le m$, write $e_{h,j}$ for the
standard basis vector of $(\Ftwo[\Gamma_{A,n}])^m$ in coordinate $j$
at $h$.  For $g\in\Gamma_{A,n}$ and
$a=(a_v)\in\mathcal C_i^\perp$, with
$a_v=(a_{v,1},\ldots,a_{v,m})\in E=\Ftwo^m$, put
\[
 z_{i,g,a}^{(n)}=
       \sum_{v\in\Vertices\Theta_i}\sum_{j=1}^m
                 a_{v,j}e_{g\varphi_{i,n}(v),j}.
\]
Let $W_{A,n}\subseteq(\Ftwo[\Gamma_{A,n}])^m$ be the $\Ftwo$-linear
span of these vectors, and define
\[
 V_{A,n}=(\Ftwo[\Gamma_{A,n}])^m/W_{A,n},
 \qquad H_{A,n}=V_{A,n}\rtimes\Gamma_{A,n}.
\]
These are the quotient presentations
\eqref{mod:relations} and \eqref{mod:presentation} for the finite subfamily.
The quotient homomorphisms between the graphical groups carry each
displayed module relation to the corresponding relation at the next
stage.  They therefore induce compatible surjections on the modules
and the semidirect products.

These maps give
\begin{equation}\label{max:H-limit}
 H_A\cong\varinjlim_n H_{A,n}.
\end{equation}
Indeed, every element of $W_A$ is a finite sum of translated relations,
so only finitely many graph indices are involved.  These indices occur
at some common stage, and any equality between the finitely many group
words involved also holds at a sufficiently large stage.  This proves
$V_A\cong\varinjlim_n V_{A,n}$; the compatible actions then give
\eqref{max:H-limit}.

The group $H_{A,n}$ is generated by the images of $S$ and the $m$
vectors $[e_{1,j}]$, since conjugating $[e_{1,j}]$ by
$g\in\Gamma_{A,n}$ gives $[e_{g,j}]$.  Its normal subgroup $V_{A,n}$
is an elementary abelian $2$-group and is locally finite.  In the metric
inherited from a word metric on
$H_{A,n}$ it has asymptotic dimension zero: for every $R$, the finite
set $B_R(1)\cap V_{A,n}$ generates a finite subgroup, whose left cosets
are uniformly bounded and mutually $R$-separated.  The extension
theorem for asymptotic dimension \cite[Theorem~7]{BD} now gives
\[
 \operatorname{asdim}H_{A,n}
 \le\operatorname{asdim}V_{A,n}
       +\operatorname{asdim}\Gamma_{A,n}
 =\operatorname{asdim}\Gamma_{A,n}<\infty.
\]

For $A=\varnothing$, there are no module relations and
$H_\varnothing=(\Ftwo[F(S)])^m\rtimes F(S)$.
The same extension argument gives finite asymptotic dimension, since
$F(S)$ is free.  Use the constant sequence $H_{\varnothing,n}=H_\varnothing$.

For the set $A=A(r)$ in \eqref{prescribed:digits}, put
\[
 G_{r,n}=C_N\times H_{A,n},
 \qquad \widetilde G_{r,n}=G_{r,n}*\Z,
 \qquad N=2^K.
\]
A direct product with a fixed finite group and a free product with a
fixed group commute with these direct limits.  The induced connecting
homomorphisms are surjective, so the two limits are $G_r$ and
$\widetilde G_r$, respectively.  Since the factor $C_N$ is finite,
$G_{r,n}$ is quasi-isometric to $H_{A,n}$ and has finite asymptotic
dimension.  The free-product estimate \cite[Proposition~10]{BD} gives
\[
 \operatorname{asdim}\widetilde G_{r,n}
 \le\max\{1,\operatorname{asdim}G_{r,n}\}<\infty.
\]
This proves the proposition, including $A=\varnothing$.
\end{proof}

For irrational $r$, neither $G_r$ nor $\widetilde G_r$ coarsely embeds
into Hilbert space.  Indeed, the number $q$ in
\eqref{prescribed:parameters} is irrational: if $q$ were rational,
then $s=(a+(N-a)q)/N$ and hence $r$ would be rational.  Thus $A(r)$ is
infinite, and $\Gamma_A$ contains isometric copies of expanders of
unbounded size with uniform degree and spectral-gap bounds.  For the
natural generating sets, the inclusion $\Gamma_A\hookrightarrow H_A$
is isometric because projection onto $\Gamma_A$ is a $1$-Lipschitz
retraction.  The inclusions $H_A\hookrightarrow C_N\times H_A=G_r$
and $G_r\hookrightarrow G_r*\Z=\widetilde G_r$ are also isometric.
A coarse embedding of either group would therefore give uniformly
controlled coarse embeddings of these expanders, contradicting the
expander Poincar\'e inequality; see \cite{Osajda,WYbook}.
Consequently, both groups have infinite asymptotic dimension, since
finite asymptotic dimension implies coarse embeddability into Hilbert
space~\cite{Yu00}.

\subsection{Injectivity under direct limits}

Every group homomorphism induces a homomorphism of maximal group
$C^*$-algebras, so the surjective connecting maps in our construction
induce maps on their $K$-theory.  At a fixed word-metric scale, a
sufficiently late group in a surjective sequence has the same relations
as the limiting group at that scale.  This observation lets us lift a
class in the proper assembly domain to a finite stage.  Continuity of
maximal group $C^*$-algebras then lets us test whether its index vanishes
at a later stage.

Bartels--Echterhoff--L\"uck prove inheritance of assembly isomorphisms
under arbitrary directed colimits, assuming bijectivity for every
subgroup of every approximating group
\cite[Theorem~3.5(ii)]{BEL}.  For the convenience of readers, we state
the following proposition, assuming only injectivity for each group
in the limit sequence.  We give the lifting argument at a fixed scale $R$
and use their induction and continuity results.

\begin{proposition}\label{max:limit-injectivity}
Let $L_1\to L_2\to\cdots$ be a sequence of finitely generated groups
with surjective connecting homomorphisms, and let
$L=\varinjlim_n L_n$.  If $\mu_*^{L_n,\max}$ is injective for every
$n$, then $\mu_*^{L,\max}$ is injective.
\end{proposition}

\begin{proof}
Choose a finite symmetric generating set of $L_1$ and give every group
the word metric associated to its image.  Write $q_n:L_n\to L$ for
the quotient homomorphism.
For a finitely generated group $J$ and an integer $R\ge1$, let
$P_R(J)$ be the Rips complex: its vertices are the elements of $J$,
and its simplices are the finite subsets of diameter at most $R$.
Its translation action is proper and cocompact; after barycentric
subdivision it is a proper $J$-CW-complex.
The union over $R$ is the space of finitely supported probability
measures on $J$, with its simplicial topology.  The fixed-point space
of a finite subgroup is nonempty and contractible, by convexity, and an
infinite subgroup has no fixed point.  This gives a model for
$\underline EJ$.  Every finite collection of simplex orbits lies in
some $P_R(J)$, so the usual compact-support convention gives
\begin{equation}\label{max:Rips-domain}
 K_*^J(\underline EJ)\cong\varinjlim_R K_*^J(P_R(J)).
\end{equation}

Fix $R$.  There are only finitely many words of length at most $3R$
in the common generating alphabet.  Each such word that represents
the identity in $L$ represents the identity at some finite stage.
Consequently, for all sufficiently large $n$,
\begin{equation}\label{max:short-kernel}
 N_n\cap B_{L_n}(1,3R)=\{1\},
 \qquad N_n:=\ker q_n.
\end{equation}
We claim that $q_n$ induces an isomorphism of $L$-complexes
\begin{equation}\label{max:Rips-quotient}
 N_n\backslash P_R(L_n)\cong P_R(L).
\end{equation}

To lift a simplex $\{g_0,\ldots,g_t\}$ of $P_R(L)$, choose words
$w_i$ of length at most $R$ representing $g_0^{-1}g_i$, with $w_0=1$,
and choose words $u_{ij}$ of length at most $R$ representing
$g_i^{-1}g_j$.  The relations
\[
 w_i^{-1}w_j u_{ij}^{-1}=1
\]
have length at most $3R$ and therefore hold in $L_n$ by
\eqref{max:short-kernel}.  If $\widetilde g_0$ is any lift of $g_0$,
the vertices $\widetilde g_0w_i$ thus form a simplex of $P_R(L_n)$.
Two such lifts differ by a common element of $N_n$: after matching
their first vertices by a left translation in $N_n$, corresponding
vertices differ by a kernel element of length at most $2R$ and hence
coincide.  Also, $q_n$ identifies no two vertices of a simplex.
These facts prove \eqref{max:Rips-quotient}.

The action of $N_n$ on $P_R(L_n)$ is free, including on the interiors
of simplices.  Indeed, if $h\in N_n$ fixes a point, it permutes the
vertices of its supporting simplex.  For any such vertex $v$,
\[
 v^{-1}hv\in N_n,\qquad |v^{-1}hv|=d(v,hv)\le R,
\]
so \eqref{max:short-kernel} implies $h=1$.
Induction along $q_n$, together with \eqref{max:Rips-quotient},
therefore gives an isomorphism
\begin{equation}\label{max:induction}
 K_*^{L_n}(P_R(L_n))\xrightarrow{\ \cong\ }
 K_*^L(P_R(L)).
\end{equation}
Here we use the induction isomorphism for a group homomorphism whose
kernel acts freely on the given space; see
\cite[Lemma~1.5 and Theorem~5.1]{BEL}.  The induced $L$-space is
$L\times_{L_n}P_R(L_n)=N_n\backslash P_R(L_n)$.

Let $x\in K_*^L(\underline EL)$ satisfy
$\mu_*^{L,\max}(x)=0$.  By \eqref{max:Rips-domain}, choose a
representative $x_R\in K_*^L(P_R(L))$.  By
\eqref{max:induction}, it has a lift
$z_n\in K_*^{L_n}(P_R(L_n))$.  Denote the image of $z_n$ in
$K_*^{L_n}(\underline EL_n)$ by $y_n$.
Maximal group $C^*$-algebras commute with the group direct limit by
their universal property.  Continuity of operator $K$-theory gives
\begin{equation}\label{max:analytic-limit}
 K_*(C^*_{\max}(L))\cong
       \varinjlim_m K_*(C^*_{\max}(L_m));
\end{equation}
see also \cite[Lemma~5.2]{BEL}.
The class $\mu_*^{L_n,\max}(y_n)$ has zero image in this direct limit,
so it becomes zero at some stage $m\ge n$.

Let $\alpha_{n,m}:L_n\to L_m$ be the connecting homomorphism.
Since $\ker\alpha_{n,m}\subseteq N_n$, its kernel acts freely on
$P_R(L_n)$.  The same induction isomorphism therefore applies to
$z_n$, giving a class on the induced $L_m$-space
$L_m\times_{L_n}P_R(L_n)$.  This space is proper because its
stabilizers are images of finite stabilizers.  Map it to
$\underline EL_m$ and denote the resulting class by $y_m$.
Naturality of maximal assembly under group homomorphisms
\cite[Theorem~5.1]{BEL} gives
\[
 \mu_*^{L_m,\max}(y_m)=0.
\]
Injectivity at stage $m$ implies $y_m=0$.  By functoriality of
induction, its image in $K_*^L(\underline EL)$ is $x$, and hence $x=0$.
\end{proof}

\begin{proof}[Proof of \Cref{max:main}]
Fix either $G=G_r$ or $G=\widetilde G_r$, and let $(J_n)$ be its
sequence of approximating groups from \Cref{max:stages}.  Each $J_n$
has finite asymptotic dimension and therefore admits a coarse embedding
into Hilbert space \cite{Yu00}.  The work of Yu~\cite{Yu98,Yu00} and
Skandalis--Tu--Yu~\cite[Theorem~6.1]{STY} gives integral injectivity of
the reduced assembly map
\[
 \mu_*^{J_n}:K_*^{J_n}(\underline EJ_n)
          \longrightarrow K_*(\Cr(J_n)).
\]
The theorem of Skandalis--Tu--Yu applies to this proper assembly
domain, without a finiteness assumption on the classifying space.
Since reduced assembly factors as maximal assembly followed by the
map induced by $C^*_{\max}(J_n)\to\Cr(J_n)$, maximal assembly is
injective at every stage.  \Cref{max:limit-injectivity} now proves
integral injectivity of $\mu_*^{G,\max}$.

Finally, the comparison map $\iota_G$ is rationally injective for
every discrete group $G$.  The identity contribution to the equivariant
Chern character is the rational homology of $G\backslash\underline EG$,
with the appropriate even or odd degrees \cite{LuckChern}.  To identify
this homology with ordinary group homology, consider
\[
 EG\times_G\underline EG\longrightarrow G\backslash\underline EG.
\]
Over a cell with stabilizer $F$, the corresponding fiber is $BF$.
Since $F$ is finite, its positive-degree rational homology vanishes;
the cellular spectral sequence therefore shows that this map is a
rational homology equivalence.  Its source is homotopy equivalent to
$BG$, because $\underline EG$ is contractible.  Naturality of the
Chern character now identifies the comparison map on the identity
contribution with the ordinary homological Chern character.
Consequently $\iota_G\otimes\Q$ is injective.  Composing it with
$\mu_*^{G,\max}\otimes\Q$ proves the second assertion.
\end{proof}

\section*{Acknowledgement}

We acknowledge that the geometric and the spectral techniques in this article
are developed from arguments in the OpenAI preprint \cite{OpenAI2026}.
ChatGPT and Claude were used in language editing and grammar refinement.
All proofs are solely developed by the author.

\end{document}